\documentclass[reqno,11pt]{amsart}
\usepackage{amsmath,amsfonts,amssymb,amsxtra,latexsym,amscd,enumerate,amsthm,verbatim}

\usepackage{graphicx}

\usepackage[margin=1.40in]{geometry}
\numberwithin{equation}{section}

\def\be{{\beta}}

\def\eps{\epsilon}

\def\varep{\varepsilon}

\def\phii{\widetilde{\varphi}}

\def\al{\alpha}

\newtheorem{theorem}{Theorem}[section]
\newtheorem{lemma}[theorem]{Lemma}

\newtheorem{proposition}[theorem]{Proposition}
\newtheorem{definition}[theorem]{Definition}
\newtheorem{remark}[theorem]{Remark}

\def\beq{\begin{equation}}
\def\eeq{\end{equation}}

\def\beq{\begin{equation}}
\def\eeq{\end{equation}}

\begin{document}

\title{On the global regularity for a Wave-Klein-Gordon coupled system}

\author{Alexandru D. Ionescu}
\address{Princeton University}
\email{aionescu@math.princeton.edu}

\author{Benoit Pausader}
\address{Brown University}
\email{benoit\_pausader@brown.edu}

\thanks{The first author was supported in part by NSF grant DMS-1600028 and NSF-FRG grant DMS-1463753. The second author was supported in part by NSF grant DMS-1362940 and by a Sloan Research fellowship. Part of this work was done while the authors were staying at the IHES in the summer of $2016$.}

\begin{abstract}
In this paper we consider a coupled Wave-Klein-Gordon system in 3D, and prove global regularity and modified scattering for small and smooth initial data with suitable decay at infinity. This system was derived in \cite {WaConf} and \cite{LeMa} as a simplified model for the global nonlinear stability of the Minkowski space-time for self-gravitating massive fields.  
\end{abstract}

\maketitle

\setcounter{tocdepth}{1}

\tableofcontents

\section{Introduction}\label{equ}

We consider the Wave-Klein-Gordon (W-KG) system in $3+1$ dimensions,
\begin{equation}\label{on1}
\begin{split}
-\square u&=A^{\alpha\beta}\partial_\alpha v\partial_\beta v+Dv^2,\\
(-\square +1)v&=uB^{\alpha\beta}\partial_\alpha\partial_\beta v+Euv,
\end{split}
\end{equation}
where $u,v$ are real-valued functions, and $A^{\alpha\beta}$, $B^{\alpha\beta}$, $D$, and $E$ are real constants. Without loss of generality we may assume that $A^{\alpha\beta}=A^{\beta\alpha}$ and $B^{\alpha\beta}=B^{\beta\alpha}$, $\alpha,\beta\in\{0,1,2,3\}$.   

The system \eqref{on1} was derived by Wang \cite{WaConf} and LeFloch-Ma \cite{LeMa} as a model for the full Einstein-Klein-Gordon (E-KG) system
\begin{equation}\label{on1.3}
Ric_{\al\be}=\mathbf{D}_\al\psi\mathbf{D}_\be\psi+(1/2)\psi^2 g_{\al\be},\qquad \square_g\psi=\psi.
\end{equation}
Intuitively, the deviation of the Lorentzian metric $g$ from the Minkowski metric is replaced by a scalar function $u$, and the massive scalar field $\psi$ is replaced by $v$. The system \eqref{on1} keeps the same linear structure as the Einstein-Klein-Gordon equations in harmonic gauge, but only keeps, schematically, quadratic interactions that involve the massive scalar field (the semilinear terms in the first equation and the quasilinear terms in the second equation coming from the reduced wave operator).

A natural question in the context of evolution equations is the question of global sta\-bi\-li\-ty of certain physical solutions. For example, for the Einstein-vacuum equation, global stability of the Minkowski space-time is a central theorem in General Relativity, due to Christodoulou-Klainerman \cite{ChKl} (see also the more recent papers of Klainerman-Nicolo \cite{KlNi}, Lindblad-Rodnianski \cite{LiRo}, Bieri-Zipser \cite{BiZi}, and Speck \cite{Sp}). 

In our case, this question was first addressed by Wang \cite{WaConf} and LeFloch-Ma \cite{LeMa}, who proved global stability for the (W-KG) system in the case of small, smooth, and compactly supported perturbations. Global stability of the full (E-KG) system was then proved by LeFloch-Ma \cite{LeMaEKG}, in the case of small and smooth perturbations that agree with a Schwarzschild solution outside a compact set (see also the outline of a similar theorem by Wang \cite{Wa1}).

The analysis in \cite{WaConf,LeMa,LeMaEKG,Wa1} relies on refinements of the hyperbolic foliation method (see also \cite{LeMaHyp} for a longer exposition of the method). To implement this method one needs to first have control of the solution on an initial hyperboloid, and then propagate this control to the interior region. As a result, this approach appears to be restricted to the case when one can establish such good control on an initial hyperboloid. Due to the finite speed of propagation, this is possible in  the case of compactly supported data for the (W-KG) system, and in the case of data that agrees with $(S_m,0)$ outside a compact set for the (E-KG) system (here $S_m$ is a Schwarzschild solution with mass $m\ll 1$). In the Einstein-vacuum case, the corresponding global regularity result for such "restricted data" was proved by Lindblad-Rodnianski \cite{LiRo}. See also the work of Friedrich \cite{Fr} for an earlier semi-global result.  

Our goal in this paper is to initiate the study of global solutions for the systems \eqref{on1} and \eqref{on1.3}, in the case of small and smooth data that decay at suitable rates at infinity, but are not necessarily compactly supported. This case is  physically relevant because of the large family of asymptotically flat initial data sets. We consider here only the simpler (W-KG) model \eqref{on1}, and hope to return to the full Einstein-Klein-Gordon model in the future.

Our framework in this paper is inspired by the recent advances in the global existence theory for quasilinear {\it{dispersive}} models, such as plasma models and water waves. We rely on a combination of energy estimates and Fourier analysis. At a very general level one should think that energy estimates are used, in combination with vector-fields, to control high regularity norms of the solutions, while the Fourier analysis is used, mostly in connection with normal forms, analysis of resonant sets, and a special "designer" norm, to prove dispersion and decay in lower regularity norms.

\subsection{The main theorem} Our main theorem concerns the global regularity of the (W-KG) system \eqref{on1}, for small initial data $(u_0,v_0)$. To state this theorem precisely we need some notation. We define the operators on $\mathbb{R}^3$
\begin{equation}\label{qaz1}
\Lambda_{wa}:=|\nabla|,\qquad \Lambda_{kg}:=\langle\nabla\rangle=\sqrt{|\nabla|^2+1}.
\end{equation}
We define also the Lorentz vector-fields $\Gamma_j$ and the rotation vector-fields $\Omega_{jk}$,
\begin{equation}\label{qaz2}
\Gamma_j:=x_j\partial_t+t\partial_j,\qquad \Omega_{jk}:=x_j\partial_k-x_k\partial_j,
\end{equation}
for $j,k\in\{1,2,3\}$. These vector-fields commute with both the wave operator and the Klein-Gordon operator. For any $\alpha=(\alpha_1,\alpha_2,\alpha_3)\in(\mathbb{Z}_+)^3$ we define
\begin{equation}\label{qaz3}
\partial^\alpha:=\partial_1^{\alpha_1}\partial_2^{\alpha_2}\partial_3^{\alpha_3},\qquad \Omega^\alpha:=\Omega_{23}^{\alpha_1}\Omega_{31}^{\alpha_2}\Omega_{12}^{\alpha_3},\qquad \Gamma^\alpha:=\Gamma_1^{\alpha_1}\Gamma_2^{\alpha_2}\Gamma_3^{\alpha_3}.
\end{equation}
For any $n\in\mathbb{Z}_+$ we define $\mathcal{V}_n$ as the set of differential operators of the form
\begin{equation}\label{qaz3.1}
\mathcal{V}_n:=\big\{\Gamma_1^{a_1}\Gamma_2^{a_2}\Gamma_3^{a_3}\Omega_{23}^{b_1}\Omega_{31}^{b_2}\Omega_{12}^{b_3}\partial_0^{\alpha_0}\partial_1^{\alpha_1}\partial_2^{\alpha_2}\partial_3^{\alpha_3}:\sum_{j=1}^3(a_j+b_j)+\sum_{k=0}^3\alpha_k\leq n\big\}.
\end{equation}

To state our main theorem we need to introduce several Banach spaces of functions on $\mathbb{R}^3$.

\begin{definition}\label{spaces}
For $a\geq 0$ let $H^a$ denote the usual Sobolev spaces of index $a$ on $\mathbb{R}^3$. We define also the Banach spaces $H^{a,b}_\Omega$, $a,b\in\mathbb{Z}_+$, by the norms
\begin{equation}\label{qaz4}
\|f\|_{H^{a,b}_{\Omega}}:=\sum_{|\alpha|\leq b}\|\Omega^\alpha f\|_{H^a}.
\end{equation} 
We also define the weighted Sobolev spaces $H^{a,b}_{S,wa}$ and $H^{a,b}_{S,kg}$ by the norms
\begin{equation}\label{qaz4.2}
\|f\|_{H^{a,b}_{S,wa}}:=\sum_{|\beta'|\leq |\beta|\leq b}\|x^{\beta'}\partial^{\beta} f\|_{H^a},\qquad \|f\|_{H^{a,b}_{S,kg}}:=\sum_{|\beta|,|\beta'|\leq b}\|x^{\beta'}\partial^{\beta} f\|_{H^a},
\end{equation}
where $x^{\beta'}=x_1^{\beta'_1}x_2^{\beta'_2}x_3^{\beta'_3}$ and $\partial^\beta:=\partial_1^{\beta_1}\partial_2^{\beta_2}\partial_3^{\beta_3}$. Notice that $H^{a,b}_{S,kg}\hookrightarrow H^{a,b}_{S,wa}\hookrightarrow H^{a,b}_{\Omega}\hookrightarrow H^a$.
\end{definition}

We are now ready to state our main theorem, which is a global regularity result for the system \eqref{on1}. For simplicity we will assume that $B^{00}=0$. This can be achieved by replacing a term like $B^{00}\partial_0^2v$ with $B^{00}(\Delta v-v)$ (which can be incorporated with the other quadratic terms in the nonlinearity), at the expense of creating cubic order terms in the second equation in \eqref{on1} which do not change the analysis. 

\begin{theorem}\label{main} Assume that $N_0:=40$, $N_1:=3$, $d:=10$, and $A^{\al\be},B^{\al\be},D,E\in\mathbb{R}$, $B^{00}=0$. Assume that $u_0,\dot{u}_0,v_0,\dot{v}_0:\mathbb{R}^3\to\mathbb{R}$ are real-valued initial data satisfying the smallness assumptions
\begin{equation}\label{ml1}
\sum_{n=0}^{N_1}\big[\|\,|\nabla|^{-1/2}U^{wa}_0\|_{H^{N(n),n}_{S,wa}}+\|U^{kg}_0\|_{H^{N(n),n}_{S,kg}}\big]\leq\varep_0\leq\overline{\varep},
\end{equation}
where $\overline{\varep}$ is a sufficiently small constant (depending only on the constants $A^{\alpha\beta}$, $B^{\alpha\beta}$, $D$, $E$ in \eqref{on1}), $N(0)=N_0+3d$, $N(n)=N_0-dn$ for $n\geq 1$, and
\begin{equation}\label{ml1.1}
U^{wa}_0:=\dot{u}_0-i\Lambda_{wa} u_0,\qquad U^{kg}_0:=\dot{v}_0-i\Lambda_{kg} v_0.
\end{equation}
Then there is a unique real-valued global solution $(u,v)$ of the system \eqref{on1} with $(|\nabla|^{1/2}u,v)\in C([0,\infty):H^{N(0)}\times H^{N(0)+1})\cap C^1([0,\infty):H^{N(0)-1}\times H^{N(0)})$ , with initial data 
\begin{equation*}
u(0)=u_0,\qquad\partial_tu(0)=\dot{u}_0,\qquad v(0)=v_0,\qquad\partial_tv(0)=\dot{v}_0.
\end{equation*} 
Moreover, with $\overline{\delta}=10^{-7}$, the solution $(u,v)$ satisfies the energy bounds with slow growth, 
\begin{equation}\label{ml2}
\sup_{n\leq N_1,\,\mathcal{L}\in\mathcal{V}_n}\big\{\||\nabla|^{-1/2}(\partial_t-i\Lambda_{wa})\mathcal{L}u(t)\|_{H^{N(n)}}+\|(\partial_t-i\Lambda_{kg})\mathcal{L}v(t)\|_{H^{N(n)}}\big\}\lesssim\varep_0(1+t)^{\overline{\delta}},
\end{equation}
for any $t\in[0,\infty)$, and the pointwise decay bounds
\begin{equation}\label{ml2.1}
\sum_{|\alpha|+\alpha_0\leq N_0-2d}\big\{\|\partial^\alpha\partial_0^{\alpha_0} u(t)\|_{L^\infty}+\|\partial^\alpha\partial_0^{\alpha_0} v(t)\|_{L^\infty}\big\}\lesssim\varep_0(1+t)^{\overline{\delta}-1}.
\end{equation}
\end{theorem}

We conclude this subsection with several remarks.

\begin{remark}\label{rem}

(i) The hypothesis on the data \eqref{ml1} can be expressed easily in terms of the physical variables $u_0$, $\dot{u}_0$, $v_0$, $\dot{v}_0$, which are related to the real and the imaginary parts of the normalized variables $U^{wa}_0$ and $U^{kg}_0$. It can also be expressed in the Fourier space, i.e.
\begin{equation}\label{ml1.11}
\sum_{|\beta'|\leq\gamma\leq n}\|\,|\xi|^{-1/2+\gamma}\langle\xi\rangle^{N(n)}\partial^{\beta'}_\xi\widehat{U^{wa}_0}\|_{L^2_\xi}+\sum_{|\beta'|,\gamma\leq n}\|\,|\xi|^{\gamma}\langle\xi\rangle^{N(n)}\partial^{\beta'}_\xi\widehat{U^{kg}_0}\|_{L^2_\xi}\lesssim\varep_0,
\end{equation}
for $n\in[0,N_1]$.  

(ii) The low frequency structure of the wave component, in particular the $|\nabla|^{-1/2}$ multiplier, is important. This is due to the fact that the bilinear interactions in the (W-KG) system are resonant only when the frequency of the wave component is $0$, so we need precise control of these low frequencies. See also subsection \ref{NESS} below for a discussion of some of the main bilinear interactions that involve low frequencies of the wave component. 

(iii) One can derive more information about the global solution $(u,v)$ as part of the bootstrap argument. In fact, the solution satisfies the main bounds \eqref{bootstrap3.1}--\eqref{bootstrap3.4} in Proposition \ref{bootstrap}. 

At a qualitative level, we also provide a precise description of the asymptotic behavior of the solution. More precisely, the wave component $u$ scatters linearly{\footnote{The linear scattering here is likely due to the very simple semilinear equation for $u$. In the case of the full Einstein-Klein-Gordon system, one expects modified scattering for the metric components as well.}} (in a weaker norm), in the sense that there exists a profile $V^{wa}_\infty\in L^2(\mathbb{R}^3)$ such that 
\begin{equation}\label{ScatB}
\Vert(\partial_t-i\Lambda_{wa})u(t)-e^{-it\Lambda_{wa}}V_\infty^{wa}\Vert_{L^2}\to 0,\quad \hbox{as } t\to\infty.
\end{equation}       
On the other hand, the Klein-Gordon component undergoes nonlinear scattering, in the sense that there exists a profile $V^{kg}_\infty$ such that
\begin{equation}\label{ModScat}
\begin{split}
&\Vert(\partial_t-i\Lambda_{kg})v(t)-e^{-it\Lambda_{kg}+i\Theta(\xi,t)}V_\infty^{kg}\Vert_{L^2}\to 0,\quad \hbox{as } t\to\infty,\\
&\text{where }\Theta(\xi,t):=\mathfrak{q}_+(\xi)\int_0^tu_{low}(s\nabla\Lambda_{kg}(\xi),s)\,ds.
\end{split}
\end{equation}
Here $\mathfrak{q}_+(\xi)$ denotes a suitable multiplier that depends on the coefficients $B^{\al\be}$ and E, and $u_{low}$ is a low-frequency truncation of $u$. The phase $\Theta(\xi,t)$ is only relevant if $u_{low}$ is not integrable along Klein-Gordon characteristics, which is the case in our problem. See subsection \ref{NESS} below.
\end{remark}

\subsection{Overview of the proof}

The system \eqref{on1} is a quasilinear system of hyperbolic and dispersive equations. For general such systems, even small and smooth data can lead to finite-time blow-up \cite{Jo} and the analysis depends on fine properties of the propagation of small waves (i.e. the linearized operator)  and on the precise structure of the nonlinearity (null forms).

One of the main difficulties in the analysis of \eqref{on1} comes from the fact that we have a genuine system in the sense that the linear evolution admits different speeds of propagation, corresponding to wave and Klein-Gordon propagation. As a result the set of ``characteristics'' is more involved and one has a more limited set of geometric vector-fields available.

On a more technical level, it turns out that the main feature of the system above is the slow decay of the low frequencies of wave component $u$ in the {\it interior} of the light cone and in particular along the characteristics associated to the Klein-Gordon operator. This nonlinear effect ultimately leads to modified scattering for the Klein-Gordon component. This effect would not be present if the quasilinear term was of the form $\partial u\cdot \nabla^2v$, in which case one recovers linear scattering, see \cite{Kat}.

A number of important techniques have emerged over the years in the study of hyperbolic systems of wave-type, starting with seminal contributions of John, Klainerman, Shatah, Simon, Christodoulou, and Alinhac \cite{Alin2,Alin3,Ch,ChKl,Jo,JoKl,Kl2,KlVf,Kl,Kl4,Sh,Si}. These include the vector-field method, the normal form method, and the isolation of null structures. In our case, the nonlinearity does not present a null structure, but has a delicate resonant pattern, and the coupled system has a limited set of vector-fields which we use in our analysis.

Our approach, which builds on these early contributions can be traced back to ideas introduced by Delort-Fang-Xue \cite{Del,DeFa,DeFaXu}, Germain-Masmoudi-Shatah \cite{GeMaSh,GeMaSh2}, Gustafson-Nakanishi-Tsai \cite{GuNaTs}, and developed by the authors and coauthors \cite{De,DeIoPa,DeIoPaPu,GuIoPa,GuPa,IoPa1,IoPa2,IoPu2,IoPu,KaPu}. We also refer to \cite{Geo,Kat} for additional works on systems of wave and Klein-Gordon equations, and to \cite{GeMa,IoPa2,De} for recent work on systems of Klein-Gordon equations with different speeds.

In this paper we use a combination of energy estimates and Fourier analysis to control our solutions. More precisely, we prove: 

\begin{itemize}
\item Energy estimates to control high Sobolev norms and weighted norms using the vector-fields in $\mathcal{V}_n$. All the energy bounds are allowed to grow slowly in time, at various rates. These energy bounds are also transferred to prove $L^2$ bounds with slow time growth on the linear profiles and their derivatives in the Fourier space. 

\item Dispersive estimates, which lead to sharp decay. These are uniform bounds in time (i.e. without slow time growth), in a suitable lower regularity $Z$ norm. The choice of this "designer" norm is important, and we construct it using a space-frequency atomic decomposition of the profiles of the solution, as in some of our earlier papers{\footnote{In general, one should think of the $Z$ norm as being connected to the location and the shape of the set of space-time resonances of the system, as in \cite{IoPa2,GuIoPa, DeIoPa, DeIoPaPu}. In our case here there are no nontrivial space-time resonances and the construction is simpler.}} starting with \cite{IoPa1}. At this stage, in order to prove uniform bounds it is important to identify a nonlinear correction of the phase and prove nonlinear scattering. 
\end{itemize}

\subsection{Nonlinear effects}\label{NESS}
In this subsection we isolate two of the main nonlinear interactions in the system \eqref{on1}, and explain their relevance in the proof.

\subsubsection{The low frequency structure of the wave component} By inspection of \eqref{on1} we observe first that bounding the quadratic terms amounts to control trilinear integrals of the form
\begin{equation*}
I=\iint u\cdot \partial^\alpha v\cdot \partial^\beta v \,\,dx dt.
\end{equation*}
Thus the resonant analysis is controlled by (essentially) only one type of quadratic phase
\begin{equation}\label{PhaseReso}
\Phi(\xi_1,\xi_2)=\Lambda_{kg}(\xi_1)\pm\Lambda_{kg}(\xi_2)\pm \Lambda_{wa}(\xi_1+\xi_2),\qquad\vert\Phi\vert\gtrsim (1+\vert\xi_1\vert+\vert\xi_2\vert)^{-2}\vert\xi_1+\xi_2\vert.
\end{equation}
Thus we expect that the interactions where the wave component has small frequency, in particular when $t\vert\xi_1+\xi_2\vert\lesssim 1$, will play an important role in the analysis.

One of the main difficulties in proving energy estimates comes from the imbalance in the quasilinear term $u\nabla^2v$, since energy estimates only lead to control of derivatives of $u$. To illustrate this, apply a commuting vector-field $\Gamma$ to \eqref{on1} to get a system of the schematic form
\begin{equation*}
\begin{split}
-\Box (\Gamma u)&=v\cdot \Gamma v,\\
(-\Box+1)(\Gamma v)&=(\Gamma u)\partial^2v+\left\{u\cdot\partial^2(\Gamma v)+u\cdot\partial^2v\right\}.
\end{split}
\end{equation*}
The terms in curly bracket in the equation for $v$ can be treated easily (assuming bootstrap energy estimates for $v$) and will be discarded for the following discussion. The first term in the equation for $v$ is more problematic because standard energy estimates would only allow us to control energy functionals of the form
\begin{equation*}
\mathcal{E}^2_w(\Gamma u)\approx \Vert \nabla_{x,t}\Gamma u\Vert_{L^2}^2,\qquad \mathcal{E}^2_{kg}(\Gamma v)\approx \Vert \Gamma v\Vert_{H^1}^2.
\end{equation*}
Thus $\Gamma u\cdot\partial^2 v$ is not well controlled when $u$ has small frequencies $\approx 1/t$, and we have an unwanted growth factor of up to $t$.

In order to compensate for this and recover the missing derivative, we use the faster (optimal) decay of the Klein-Gordon solution in two steps and the special structure of the system
\begin{equation*}
KG\times KG\to Wave,\qquad Wave\times KG\to KG.
\end{equation*}
This allows us to control first $\vert\nabla\vert^{-1/2}\Gamma u$ in energy norm. Indeed, assuming that $u$ is located at frequencies $\vert\xi\vert\approx 1/t$, the first equation gives
\begin{equation*}
\begin{split}
\partial_t\mathcal{E}_w(\vert\nabla\vert^{-1/2}\Gamma u)&\lesssim \Vert \vert\nabla\vert^{-1/2}P_{\sim 1/t}(v\cdot\Gamma v)\Vert_{L^2}\lesssim (1+t)^{1/2}\Vert v\Vert_{L^\infty}\cdot\mathcal{E}_{kg}(\Gamma v).
\end{split}
\end{equation*}
We would thus obtain an acceptable contribution, at least as long as we can show that $v$ decays pointwise at the optimal rate $(1+t)^{-3/2}$. On the other hand, if we now compute the corresponding contribution in the energy estimates for $v$, we obtain (discarding the easy terms, and assuming again that $u$ is located at frequencies $\approx 1/t$)
\begin{equation*}
\begin{split}
\partial_t\mathcal{E}_{kg}(\Gamma v)&\lesssim \Vert \Gamma u\cdot \partial^2 v\Vert_{L^2}\lesssim \Vert \vert\nabla\vert^{-1/2}\cdot (\nabla_{x,t}\vert\nabla\vert^{-1/2}\Gamma u)\Vert_{L^2}\Vert v\Vert_{W^{2,\infty}}\\
&\lesssim (1+t)^{1/2}\Vert v\Vert_{W^{2,\infty}}\cdot\mathcal{E}_w(\vert\nabla\vert^{-1/2}\Gamma u).
\end{split}
\end{equation*}
Once again we obtain an acceptable contribution, as long as we can show that $v$ has optimal decay in time. To prove this optimal decay{\footnote{In fact, we will not prove optimal decay for the full function $v$, but we will decompose $v=v_{\infty}+v_{2}$ where $v_{\infty}$ has optimal pointwise decay and $v_2$ is suitably small in $L^2$.}} we need to use the $Z$-norm, identify a nonlinear phase correction, and prove modified scattering for $v$ (see below for a discussion of this step).

This scheme allows us to deal with the contribution of the frequencies $|\xi|\approx 1/t$ coming from $u$ (and also explains the factor $|\nabla|^{-1/2}$ in the energy functionals for $u$). To deal with the contribution of larger frequencies we can start integrating by parts in time (the method of normal forms) and use the lower bound \eqref{PhaseReso} on the resonance phases.

\subsubsection{Long-range perturbations and modified scattering} Assume for simplicity that we consider radial solutions of the system
\begin{equation}\label{non2B}
\begin{split}
\Box u+\vert\nabla_{x,t}v\vert^2+v^2&=0,\qquad (u(0),\partial_tu(0))=(0,0),\\
(-\Box+1)v-u\Delta v&=0,\qquad (v(0),\partial_tv(0))=(\chi,0),\qquad\mathbf{1}_{\{\vert x\vert\le 1\}}\le\chi\le\mathbf{1}_{\{\vert x\vert\le 2\}}.
\end{split}
\end{equation}
Assume that $v$ decays no faster than a linear solution,
\begin{equation*}
\begin{split}
\vert\nabla_{x,t}v(t)\vert+\vert v(t)\vert\ge\varepsilon\langle t\rangle^{-3/2}\mathbf{1}_{\{\vert x\vert\le t/2\}}.
\end{split}
\end{equation*}
Using the explicit form of the linear propagator for the wave equation, and in particular, the fact that it is nonnegative, one can see that
\begin{equation}\label{LBB}
u(x,t)\gtrsim\varepsilon^2/\langle t\rangle,\qquad \vert x\vert\le t/4.
\end{equation}
Thus we see that $u$ has a substantial (i.e. non integrable) presence inside the light cone, where the characteristics of the Klein-Gordon equation are located. This is already a departure from linear behavior (although this only affects the behavior of $u$ on large spatial scales and disappears in the energy of $u$, as the scattering statement in $\dot{H}^1$ in \eqref{ScatB} suggests). In addition, since $u$ is nonnegative, there can be no gain by averaging (i.e. no normal form), and the contribution from $u\Delta v$ in \eqref{non2B} is indeed a long range quasilinear perturbation. 

To control a norm that does not grow (the $Z$ norm) we need to identify the correct asymptotic behavior and the correct nonlinear oscillations. Conjugating by the linear flow, letting
\begin{equation*}
V^{kg}(t)=e^{it\Lambda_{kg}}\left(\partial_t-i\Lambda_{kg}\right)v(t),
\end{equation*}
 leads to a nonlinear equation for the Klein-Gordon profile
\begin{equation*}
\begin{split}
\partial_t\widehat{V^{kg}}(\xi,t)&=\frac{1}{(2\pi)^3}\int_{\mathbb{R}^3}e^{it\Lambda_{kg}(\xi)}\widehat{u}(\eta,t)[B^{\alpha\beta}\widehat{\partial_\alpha\partial_\beta v}(\xi-\eta,t)+E\widehat{v}(\xi-\eta,t)]\,d\eta.
\end{split}
\end{equation*}
We write the right-hand side in terms of the linear profile $V^{kg}$ and extract the resonant interaction that corresponds to the case when $u$ has low frequencies, see \eqref{PhaseReso}. This leads to an equation of the form
\begin{equation*}
\begin{split}
\partial_t\widehat{V^{kg}}(\xi,t)&=\frac{i}{(2\pi)^3}\int_{\{\vert\eta\vert\ll 1\}}e^{it\left[\Lambda_{kg}(\xi)-\Lambda_{kg}(\xi-\eta)\right]}\widehat{u}(\eta,t)\mathfrak{q}_+(\xi-\eta)\widehat{V^{kg}}(\xi-\eta,t)d\eta +l.o.t.\\
&=i\mathfrak{q}_+(\xi)\widehat{V^{kg}}(\xi,t)\cdot \frac{1}{(2\pi)^3}\int_{\{\vert\eta\vert\ll 1\}}e^{it\left[\eta\cdot\nabla\Lambda_{kg}(\xi)+O(\vert\eta\vert^2)\right]}\widehat{u}(\eta,t)d\eta +l.o.t.\\
\end{split}
\end{equation*}
where $\mathfrak{q}_+$ denotes a real-valued multiplier. Discarding the quadratic error in the phase and performing the Fourier inversion leads to the ODE
\begin{equation*}
\partial_t\widehat{V^{kg}}(\xi,t)=i\mathfrak{q}_+(\xi)u_{low}(t\nabla\Lambda_{kg}(\xi),t)\cdot \widehat{V^{kg}}(\xi,t)+l.o.t.
\end{equation*}
This leads to a phase correction (written explicitly in \eqref{ModScat}) corresponding to integrating the effect of the quasilinear term along the characteristics of the Klein-Gordon flow $(t\nabla\Lambda_{kg}(\xi),t)$. This is consistent with a choice of $Z$-norm for $v$ controlling the amplitude of the solutions pointwise in Fourier space, but allowing for an additional oscillating phase, see \eqref{sec5.1}.

\subsection{Organization} The rest of the paper is concerned with the proof of Theorem \ref{main}. In section \ref{NotationsF} we introduce the main notation, define the main $Z$-norm, and state the main bootstrap proposition. In section \ref{SLSec} we prove various lemmas, such as dispersive linear bounds and some bounds on quadratic phases. In section \ref{dtv} we use the the bootstrap assumptions and elliptic analysis to derive various bounds on the unknowns and their vector-field derivatives. 

We then start the proof of the main bootstrap proposition in section \ref{enerEst}, where we obtain improved energy estimates. In section \ref{DiEs} we translate the estimate on vector-fields applied to functions into weighted bounds on the linear profiles. In the last two sections, we recover uniform control of the $Z$-norm. This involves first isolating the modification to linear scattering in section \ref{DiEs2}, where we also control the Klein-Gordon solution. Finally in section \ref{DiEs3} we control the $Z$-norm for the wave unknown.

\section{Function spaces and the main proposition}\label{NotationsF}

\subsection{Notation, atomic decomposition, and the $Z$-norm}\label{defznorm} We start by summarizing our main definitions and notations.

\subsubsection{Littlewood-Paley projections}

We fix $\varphi:\mathbb{R}\to[0,1]$ an even smooth 
function supported in $[-8/5,8/5]$ and equal to $1$ in $[-5/4,5/4]$. For simplicity of notation, we also 
let $\varphi:\mathbb{R}^3\to[0,1]$ denote the corresponding radial function on $\mathbb{R}^3$. Let
\begin{equation*}
\varphi_{k}(x):=\varphi(|x|/2^k)-\varphi(|x|/2^{k-1})\text{ for any }k\in\mathbb{Z},\qquad \varphi_I:=\sum_{m\in I\cap\mathbb{Z}}\varphi_m\text{ for any }I\subseteq\mathbb{R}.
\end{equation*}
For any $B\in\mathbb{R}$ let 
\begin{equation*}
\varphi_{\leq B}:=\varphi_{(-\infty,B]},\quad\varphi_{\geq B}:=\varphi_{[B,\infty)},\quad\varphi_{<B}:=\varphi_{(-\infty,B)},\quad \varphi_{>B}:=\varphi_{(B,\infty)}.
\end{equation*}
For any $a<b\in\mathbb{Z}$ and $j\in[a,b]\cap\mathbb{Z}$ let
\begin{equation}\label{Alx80}
\varphi^{[a,b]}_j:=
\begin{cases}
\varphi_{j}\quad&\text{ if }a<j<b,\\
\varphi_{\leq a}\quad&\text{ if }j=a,\\
\varphi_{\geq b}\quad&\text{ if }j=b.
\end{cases}
\end{equation}

For any $x\in\mathbb{Z}$ let $x^{+}=\max(x,0)$ and $x^-:=\min(x,0)$. Let
\begin{equation*}
\mathcal{J}:=\{(k,j)\in\mathbb{Z}\times\mathbb{Z}_+:\,k+j\geq 0\}.
\end{equation*}
For any $(k,j)\in\mathcal{J}$ let
\begin{equation*}
\phii^{(k)}_j(x):=
\begin{cases}
\varphi_{\leq -k}(x)\quad&\text{ if }k+j=0\text{ and }k\leq 0,\\
\varphi_{\leq 0}(x)\quad&\text{ if }j=0\text{ and }k\geq 0,\\
\varphi_j(x)\quad&\text{ if }k+j\geq 1\text{ and }j\geq 1,
\end{cases}
\end{equation*}
and notice that, for any $k\in\mathbb{Z}$ fixed, $\sum_{j\geq-\min(k,0)}\phii^{(k)}_j=1$. 

Let $P_k$, $k\in\mathbb{Z}$, denote the operator on $\mathbb{R}^2$ defined by the Fourier multiplier $\xi\to \varphi_k(\xi)$. 
Let $P_{\leq B}$ (respectively $P_{>B}$) denote the operators on $\mathbb{R}^2$ defined by the Fourier 
multipliers $\xi\to \varphi_{\leq B}(\xi)$ (respectively $\xi\to \varphi_{>B}(\xi)$). For $(k,j)\in\mathcal{J}$ let $Q_{jk}$ denote the operator
\begin{equation}\label{qjk}
(Q_{jk}f)(x):=\phii^{(k)}_j(x)\cdot P_kf(x).
\end{equation}
In view of the uncertainty principle the operators $Q_{jk}$ are relevant only when $2^j2^k\gtrsim 1$, which explains the definitions above.

\subsubsection{Linear profiles and norms} An important role will be played by the {\it{normalized solutions}} $U^{wa},U^{kg}$ and their associated {\it{profiles}} $V^{wa},V^{kg}$ defined by
\begin{equation}\label{variables4}
\begin{split}
&U^{wa}(t):=\partial_tu(t)-i\Lambda_{wa} u(t),\qquad\, U^{kg}(t):=\partial_tv(t)-i\Lambda_{kg} v(t),\\
&V^{wa}(t):=e^{it\Lambda_{wa}}U^{wa}(t),\qquad\qquad\,V^{kg}(t):=e^{it\Lambda_{kg}}U^{kg}(t),\\
\end{split}
\end{equation}
where, as before, $\Lambda_{wa}=|\nabla|$ and $\Lambda_{kg}=\sqrt{1+|\nabla|^2}$. More generally, for differential operators $\mathcal{L}\in\mathcal{V}_{N_1}$ we define $U^{wa}_{\mathcal{L}},U^{kg}_{\mathcal{L}},V^{wa}_{\mathcal{L}},V^{kg}_{\mathcal{L}}$ by
\begin{equation}\label{variables4L}
\begin{split}
&U^{wa}_{\mathcal{L}}(t):=(\partial_t-i\Lambda_{wa}) ({\mathcal{L}}u)(t),\qquad\, U^{kg}_{\mathcal{L}}(t):=(\partial_t-i\Lambda_{kg}) ({\mathcal{L}}v)(t),\\
&V^{wa}_{\mathcal{L}}(t):=e^{it\Lambda_{wa}}U^{wa}_{\mathcal{L}}(t),\qquad\qquad\,\,\,\,\, V^{kg}_{\mathcal{L}}(t):=e^{it\Lambda_{kg}}U^{kg}_{\mathcal{L}}(t).
\end{split}
\end{equation}
We define also
\begin{equation}\label{notation}
\begin{split}
&U^{wa,-}_{\mathcal{L}}:=\overline{U^{wa}_{\mathcal{L}}},\qquad U^{kg,-}_{\mathcal{L}}:=\overline{U^{kg}_{\mathcal{L}}};\qquad V^{wa,-}_{\mathcal{L}}:=\overline{V^{wa}_{\mathcal{L}}},\qquad V^{kg,-}_{\mathcal{L}}:=\overline{V^{kg}_{\mathcal{L}}},\\
&U^{wa,+}_{\mathcal{L}}:=U^{wa},\qquad U^{kg,+}:=U^{kg}_{\mathcal{L}};\qquad V^{wa,+}_{\mathcal{L}}:=V^{wa},\qquad V^{kg,+}_{\mathcal{L}}:=V^{kg}.
\end{split}
\end{equation}
The functions $\mathcal{L}u,\mathcal{L}v$ can be recovered from the normalized variables $U^{wa}_{\mathcal{L}},U^{kg}_{\mathcal{L}}$ by the formulas
\begin{equation}\label{on5}
\begin{split}
&\partial_0 (\mathcal{L}u)=(U^{wa}_{\mathcal{L}}+\overline{U^{wa}_{\mathcal{L}}})/2,\qquad \Lambda_{wa}(\mathcal{L}u)=i(U^{wa}_{\mathcal{L}}-\overline{U^{wa}_{\mathcal{L}}})/2,\\
&\partial_0 (\mathcal{L}v)=(U^{kg}_{\mathcal{L}}+\overline{U^{kg}_{\mathcal{L}}})/2,\qquad\,\, \Lambda_{kg}(\mathcal{L}v)=i(U^{kg}_{\mathcal{L}}-\overline{U^{kg}_{\mathcal{L}}})/2.
\end{split}
\end{equation}
The system \eqref{on1} gives, for any $\mathcal{L}\in\mathcal{V}_{N_1}$,
\begin{equation}\label{on6.1L}
\begin{split}
(\partial_t+i\Lambda_{wa})U^{wa}_{\mathcal{L}}&=\mathcal{N}^{wa}_{\mathcal{L}}:=\mathcal{L}[A^{\alpha\beta}\partial_\alpha v\partial_\beta v+Dv^2],\\
(\partial_t+i\Lambda_{kg})U^{kg}_{\mathcal{L}}&=\mathcal{N}^{kg}_{\mathcal{L}}:=\mathcal{L}[uB^{\alpha\beta}\partial_\alpha\partial_\beta v+Euv].
\end{split}
\end{equation}

Let
\begin{equation}\label{on9.1}
\mathcal{P}:=\{(wa,+),(wa,-),(kg,+),(kg,-)\}.
\end{equation}
Let $\Lambda_{wa,+}(\xi):=\Lambda_{wa}(\xi)=|\xi|$, $\Lambda_{wa,-}(\xi):=-\Lambda_{wa,+}(\xi)$, $\Lambda_{kg,+}(\xi):=\Lambda_{kg}(\xi)=\sqrt{|\xi|^2+1}$, $\Lambda_{kg,-}(\xi):=-\Lambda_{kg,+}(\xi)$. For any $\sigma,\mu,\nu\in\mathcal{P}$ we define the associated {\it{quadratic phase function}}
\begin{equation}\label{on9.2}
\Phi_{\sigma\mu\nu}:\mathbb{R}^3\times\mathbb{R}^3\to\mathbb{R},\qquad \Phi_{\sigma\mu\nu}(\xi,\eta):=\Lambda_{\sigma}(\xi)-\Lambda_{\mu}(\xi-\eta)-\Lambda_\nu(\eta).
\end{equation}

We are now ready to define the main norms.

\begin{definition}\label{MainZDef}
For any $x\in\mathbb{R}$ let $x^+=\max(x,0)$ and $x^-=\min(x,0)$. Let $\delta:=10^{-10}$, $\kappa:=20\sqrt{\delta}=2\times 10^{-4}$, $d=10$, and $d':=3d/2$. We define the spaces $Z_{wa},Z_{kg}$, by the norms
\begin{equation}\label{sec5}
\|f\|_{Z_{wa}}:=\sup_{k\in\mathbb{Z}}\big\{2^{(N_0-d')k^+}2^{k^-(1/2+\kappa)}\sum_{j\geq\max(0,-k)}2^j\|Q_{jk}f\|_{L^2}\big\}
\end{equation}
and
\begin{equation}\label{sec5.1}
\|f\|_{Z_{kg}}:=\sup_{k\in\mathbb{Z}}\big\{2^{(N_0-d')k^+}2^{k^-(1/2-\kappa)}\|\widehat{P_kf}\|_{L^\infty}+2^{(N(0)-2)k^+}2^{-k^--\kappa k^-}\|P_kf\|_{L^2}\big\}.
\end{equation}
\end{definition}

We remark that the norms $Z_{wa}$ and $Z_{kg}$ are applied to the profiles $V^{wa}$ and $V^{kg}$, not to the normalized solutions $U^{wa}$ and $U^{kg}$.

\subsection{The main bootstrap proposition}\label{bootstrap0} Our main result is the following proposition:

\begin{proposition}\label{bootstrap} Assume that $(u,v)$ is a solution to \eqref{on1} on some time interval $[0,T]$, $T\geq 1$, with initial data $(u_0,\dot{u}_0,v_0,\dot{v}_0)$ satisfying the assumptions \eqref{ml1}, and define $U^{wa}_{\mathcal{L}},U^{kg}_{\mathcal{L}},V^{wa}_{\mathcal{L}},V^{kg}_{\mathcal{L}}$ as before. Assume also that, for any $t\in[0,T]$, the solution satisfies the bootstrap hypothesis
\begin{equation}\label{bootstrap2.1}
\sup_{n\leq N_1,\,\mathcal{L}\in\mathcal{V}_n}\big\{\||\nabla|^{-1/2}U_\mathcal{L}^{wa}(t)\|_{H^{N(n)}}+\|U_\mathcal{L}^{kg}(t)\|_{H^{N(n)}}\big\}\leq\varep_{1}\langle t\rangle^{H(n)\delta},
\end{equation}
\begin{equation}\label{bootstrap2.2}
\begin{split}
\sup_{n\leq N_1-1,\,\mathcal{L}\in\mathcal{V}_n,\,l\in\{1,2,3\}}\sup_{k\in\mathbb{Z}}\,2^{N(n+1)k^+}\big\{&2^{k/2}\|\varphi_k(\xi)(\partial_{\xi_l}\widehat{V^{wa}_{\mathcal{L}}})(\xi,t)\|_{L^2_\xi}\\
&+2^{k^+}\|\varphi_k(\xi)(\partial_{\xi_l}\widehat{V^{kg}_{\mathcal{L}}})(\xi,t)\|_{L^2_\xi}\big\}\leq\varep_{1}\langle t\rangle^{H(n+1)\delta},
\end{split}
\end{equation}
and
\begin{equation}\label{bootstrap2.4}
\|V^{wa}(t)\|_{Z_{wa}}+\|V^{kg}(t)\|_{Z_{kg}}\leq\varep_{1},
\end{equation}
where $\langle t\rangle:=\sqrt{1+t^2}$, $\varep_1=\varep_0^{2/3}$, $\delta=10^{-10}$, $d=10$, 
\begin{equation}\label{fvc1}
H(0)=1,\qquad\qquad\,H(n)=200n-190\text{ for }n\in\{1,\ldots,N_1\},
\end{equation}
and
\begin{equation}\label{fvc1.1}
N(0)=N_0+3d,\qquad N(n)=N_0-dn\text{ for }n\in\{1,\ldots,N_1\}.
\end{equation}
Then the following improved bounds hold, for any $t\in[0,T]$,
\begin{equation}\label{bootstrap3.1}
\sup_{n\leq N_1,\,\mathcal{L}\in\mathcal{V}_n}\big\{\||\nabla|^{-1/2}U_\mathcal{L}^{wa}(t)\|_{H^{N(n)}}+\|U_\mathcal{L}^{kg}(t)\|_{H^{N(n)}}\big\}\lesssim\varep_{0}\langle t\rangle^{H(n)\delta},
\end{equation}
\begin{equation}\label{bootstrap3.2}
\begin{split}
\sup_{n\leq N_1-1,\,\mathcal{L}\in\mathcal{V}_n,\,l\in\{1,2,3\}}\sup_{k\in\mathbb{Z}}\,2^{N(n+1)k^+}\big\{&2^{k/2}\|\varphi_k(\xi)(\partial_{\xi_l}\widehat{V^{wa}_{\mathcal{L}}})(\xi,t)\|_{L^2_\xi}\\
&+2^{k^+}\|\varphi_k(\xi)(\partial_{\xi_l}\widehat{V^{kg}_{\mathcal{L}}})(\xi,t)\|_{L^2_\xi}\big\}\lesssim\varep_0\langle t\rangle^{H(n+1)\delta},
\end{split}
\end{equation}
and
\begin{equation}\label{bootstrap3.4}
\|V^{wa}(t)\|_{Z_{wa}}+\|V^{kg}(t)\|_{Z_{kg}}\lesssim\varep_0.
\end{equation}
\end{proposition}

We will show in Proposition \ref{IniDat} below that the hypothesis \eqref{ml1} implies that desired conclusions \eqref{bootstrap3.1}--\eqref{bootstrap3.4} at time $t=0$. Given Proposition \ref{bootstrap}, Theorem \ref{main} follows using a local existence result and a continuity argument. The rest of this paper is concerned with the proof of Proposition \ref{bootstrap}.

The bounds \eqref{bootstrap2.1} and \eqref{bootstrap3.1} provide high order energy control on the main variables $U_{\mathcal{L}}^{wa}$ and $U_{\mathcal{L}}^{kg}$. Notice that all the energy functionals are allowed to grow slowly in time. Notice also that there is a certain {\it{energy hierarchy}} expressed in terms of the parameters $H(n)$ and $N(n)$, in the sense that the variables with more vector-fields are allowed to grow slightly faster compared with those with fewer vector-fields, in weaker Sobolev spaces.

The bounds \eqref{bootstrap2.2} and \eqref{bootstrap3.2} are our main $L^2$ bounds on the derivatives of the profiles $V_{\mathcal{L}}^{wa}$ and $V_{\mathcal{L}}^{kg}$ in the Fourier space. They correspond to weighted bounds in the physical space and can be linked to the energy estimates using the key identities in Lemma \ref{ident}. 

The bounds \eqref{bootstrap2.4} and \eqref{bootstrap3.4} are our main dispersive bounds. Notice that these bounds are more precise than the Sobolev bounds, in the sense that the solutions are not allowed to grow slowly in time in the $Z$-norm, but at a lower order of derivatives and without vector-fields. To prove these dispersive bounds it is important to first renormalize the Klein-Gordon profile $V^{kg}$ and prove modified scattering.

\section{Some lemmas}\label{SLSec} In this section we collect several lemmas that are used in the rest of the paper. We start with a lemma that is used often in integration by parts arguments. See \cite[Lemma 5.4]{IoPa2} for the proof.

\begin{lemma}\label{tech5} Assume that $0<\eps\leq 1/\eps\leq K$, $N\geq 1$ is an integer, and $f,g\in C^{N+1}(\mathbb{R}^3)$. Then
\begin{equation}\label{ln1}
\Big|\int_{\mathbb{R}^3}e^{iKf}g\,dx\Big|\lesssim_N (K\eps)^{-N}\big[\sum_{|\alpha|\leq N}\eps^{|\alpha|}\|D^\alpha_xg\|_{L^1}\big],
\end{equation}
provided that $f$ is real-valued,
\begin{equation}\label{ln2}
|\nabla_x f|\geq \mathbf{1}_{{\mathrm{supp}}\,g},\quad\text{ and }\quad\|D_x^\alpha f \cdot\mathbf{1}_{{\mathrm{supp}}\,g}\|_{L^\infty}\lesssim_N\eps^{1-|\alpha|},\,2\leq |\alpha|\leq N+1.
\end{equation}
\end{lemma}

To bound bilinear operators, we often use the following simple lemma.

\begin{lemma}\label{L1easy}
(i) Assume that $l\geq 2$, $f_1,\ldots, f_l,f_{l+1}\in L^2(\mathbb{R}^3)$, and $M:(\mathbb{R}^3)^l\to\mathbb{C}$ is a continuous compactly supported function. Then
\begin{equation}\label{ener62}
\begin{split}
\Big|\int_{(\mathbb{R}^3)^l}M(\xi_1,\ldots,\xi_l)\cdot\widehat{f_1}(\xi_1)\cdot\ldots\cdot \widehat{f_l}(\xi_l)\cdot\widehat{f_{l+1}}(-\xi_1-\ldots-\xi_l)\,d\xi_1\ldots d\xi_l\Big|\\
\lesssim \big\|\mathcal{F}^{-1}M\big\|_{L^1((\mathbb{R}^3)^l)}\|f_1\|_{L^{p_1}}\cdot\ldots\cdot\|f_{l+1}\|_{L^{p_{l+1}}},
\end{split}
\end{equation}
for any exponents $p_1,\ldots,p_{l+1}\in[1,\infty]$ satisfying $1/p_1+\ldots+1/p_{l+1}=1$. 

(ii) As a consequence, if $q,p_2,p_3\in[1,\infty]$ satisfy $1/p_2+1/p_3=1/q$ then
\begin{equation}\label{ener62.1}
\Big\|\mathcal{F}_{\xi}^{-1}\Big\{\int_{\mathbb{R}^3}M(\xi,\eta)\widehat{f}(\eta)\widehat{g}(-\xi-\eta)\,d\eta\Big\}\Big\|_{L^{q}}\lesssim \big\|\mathcal{F}^{-1}M\big\|_{L^1}\|f\|_{L^{p_2}}\|g\|_{L^{p_3}},
\end{equation}
\end{lemma}

Our next lemma is often used in integration by parts in time arguments (normal forms).

\begin{lemma}\label{pha2} (i) Assume that $\Phi_{\sigma\mu\nu}$ is as in \eqref{on9.2}. If $|\xi|,|\xi-\eta|,|\eta|\in[0,b]$, $1\leq b$, then
\begin{equation}\label{pha3}
\begin{split}
&|\Phi_{\sigma\mu\nu}(\xi,\eta)|\geq |\xi|/(4b^2)\qquad \text{ if }\,\,(\sigma,\mu,\nu)=((wa,\iota),(kg,\iota_1),(kg,\iota_2)),\\
&|\Phi_{\sigma\mu\nu}(\xi,\eta)|\geq |\eta|/(4b^2)\qquad\text{ if }\,\,(\sigma,\mu,\nu)=((kg,\iota),(kg,\iota_1),(wa,\iota_2)).
\end{split}
\end{equation}

(ii) Assume that $k,k_1,k_2\in \mathbb{Z}$ and $n$ is a multiplier such that $\|\mathcal{F}^{-1}n\|_{L^1(\mathbb{R}^3\times\mathbb{R}^3)}\leq 1$. Let $\overline{k}=\max(k,k_1,k_2)$. If $(\sigma,\mu,\nu)=((wa,\iota),(kg,\iota_1),(kg,\iota_2))$ then
\begin{equation}\label{pha4}
\big\|\mathcal{F}^{-1}\{\Phi_{\sigma\mu\nu}(\xi,\eta)^{-1}n(\xi,\eta)\cdot\varphi_k(\xi)\varphi_{k_1}(\xi-\eta)\varphi_{k_2}(\eta)\}\big\|_{L^1(\mathbb{R}^3\times\mathbb{R}^3)}\lesssim 2^{-k}2^{4\overline{k}^+}.
\end{equation}
Moreover, if $(\sigma,\mu,\nu)=((kg,\iota),(kg,\iota_1),(wa,\iota_2))$ then
\begin{equation}\label{pha4.1}
\big\|\mathcal{F}^{-1}\{\Phi_{\sigma\mu\nu}(\xi,\eta)^{-1}n(\xi,\eta)\cdot\varphi_k(\xi)\varphi_{k_1}(\xi-\eta)\varphi_{k_2}(\eta)\}\big\|_{L^1(\mathbb{R}^3\times\mathbb{R}^3)}\lesssim 2^{-k_2}2^{4\overline{k}^+}.
\end{equation}
\end{lemma}

\begin{proof} (i) The  bounds follow from the elementary inequalities
\begin{equation}\label{par4.11}
\begin{split}
\sqrt{1+x^2}+\sqrt{1+y^2}-(x+y)&\geq 1/(2b),\\
x+\sqrt{1+y^2}-\sqrt{1+(x+y)^2}&\geq x/(4b^2),
\end{split}
\end{equation}
which hold if $x,y,x+y\in[0,b]$. The second inequality can be proved by setting $F(x):=x+\sqrt{1+y^2}-\sqrt{1+(x+y)^2}$ and noticing that $F'(x)\geq 1/(4b^2)$ as long as $y,x+y\in[0,b]$.

(ii) By symmetry, it suffices to prove \eqref{pha4}. Also, since
\begin{equation}\label{pha4.15}
\|\mathcal{F}^{-1}(fg)\|_{L^1}\lesssim \|\mathcal{F}^{-1}f\|_{L^1}\|\mathcal{F}^{-1}g\|_{L^1},
\end{equation}
without loss of generality we may assume that $n\equiv 1$ and $\iota=+$. Let
\begin{equation}\label{pha4.2}
m(v,\eta):=2^{-k}\Phi_{\sigma\mu\nu}(2^kv,\eta)^{-1}=\frac{1}{|v|-2^{-k}\Lambda_{kg,\iota_1}(\eta-2^kv)-2^{-k}\Lambda_{kg,\iota_2}(\eta)}.
\end{equation}
For \eqref{pha4} it suffices to prove that
\begin{equation}\label{pha4.3}
\big\|\mathcal{F}^{-1}\{m(v,\eta)\cdot\varphi_0(v)\varphi_{k_1}(\eta-2^kv)\varphi_{k_2}(\eta)\}\big\|_{L^1(\mathbb{R}^3\times\mathbb{R}^3)}\lesssim 2^{4\overline{k}^+}.
\end{equation}
We consider two cases, depending on the signs $\iota_1$ and $\iota_2$.

{\bf{Case 1.}} $\iota_1\neq \iota_2$. By symmetry we may assume that $\iota_2=-,\iota_1=+$, so
\begin{equation}\label{pha4.4}
\begin{split}
m(v,\eta)&=\frac{1}{|v|-2^{-k}\sqrt{1+|\eta-2^kv|^2}+2^{-k}\sqrt{1+|\eta|^2}}\\
&=\frac{2^k|v|+\sqrt{1+|\eta|^2}+\sqrt{1+|\eta-2^kv|^2}}{2(|v|\sqrt{1+|\eta|^2}+v\cdot\eta)}\\
&=\frac{\big[2^k|v|+\sqrt{1+|\eta|^2}+\sqrt{1+|\eta-2^kv|^2}\big]\big[|v|\sqrt{1+|\eta|^2}-v\cdot\eta\big]}{2[|v|^2+|v|^2|\eta|^2-(v\cdot\eta)^2]}.
\end{split}
\end{equation} 
The first identity follows by algebraic simplifications, after multiplying both the numerator and the denominator by $|v|+2^{-k}\sqrt{1+|\eta-2^kv|^2}+2^{-k}\sqrt{1+|\eta|^2}$. The second identity follows by multiplying both the numerator and the denominator by $|v|\sqrt{1+|\eta|^2}-v\cdot\eta$. The numerator in the formula above is a sum of simple products and its contribution is a factor of $2^{2\overline{k}^+}$. In view of the general bound \eqref{pha4.15}, for \eqref{pha4.3} it suffices to prove that, for $l\geq 0$
\begin{equation}\label{pha4.45}
\Big\|\int_{\mathbb{R}^3\times\mathbb{R}^3}e^{ix\cdot v}e^{iy\cdot\eta}\frac{1}{|v|^2+|v|^2|\eta|^2-(v\cdot\eta)^2}\varphi_0(v)\varphi_{\leq l}(\eta)\,dv d\eta\Big\|_{L^1_{x,y}}\lesssim 2^{2l}.
\end{equation}

We insert thin angular cutoffs in $v$, i.e. factors of the form $\varphi_{\leq -l-10}(v_2)\varphi_{\leq -l-10}(v_3)$. Due to rotation invariance it suffices to prove that
\begin{equation*}
\Big\|\int_{\mathbb{R}^3\times\mathbb{R}^3}e^{ix\cdot v}e^{iy\cdot\eta}\frac{\varphi_{\leq -l-10}(v_2)\varphi_{\leq -l-10}(v_3)}{|v|^2+|v|^2|\eta|^2-(v\cdot\eta)^2}\varphi_0(v)\varphi_{\leq l}(\eta)\,dv d\eta\Big\|_{L^1_{x,y}}\lesssim 1.
\end{equation*}
We make the changes of variables $v_1\leftrightarrow w_1,v_2\leftrightarrow 2^{-l}w_2,v_3\leftrightarrow 2^{-l}w_3$, $\eta_1\leftrightarrow 2^l\rho_1,\eta_2\leftrightarrow \rho_2,\eta_3\leftrightarrow \rho_3$. After rescaling the spatial variables appropriately, it suffices to prove that
\begin{equation}\label{pha4.5}
\begin{split}
\Big\|\int_{\mathbb{R}^3\times\mathbb{R}^3}e^{ix\cdot w}e^{iy\cdot\rho}&m'(w,\rho)\varphi_{[-4,4]}(w_1)\varphi_{\leq -10}(w_2)\varphi_{\leq -10}(w_3)\\
&\varphi_{\leq 4}(\rho_1)\varphi_{\leq l+4}(\rho_2)\varphi_{\leq l+4}(\rho_3)\,dw d\rho\Big\|_{L^1_{x,y}}\lesssim 1,
\end{split}
\end{equation}
where
\begin{equation*}
\begin{split}
m'(w,\rho):=\big\{w_1^2&(1+\rho_2^2+\rho_3^2)+\rho_1^2(w_2^2+w_3^2)\\
&+2^{-2l}(w_2^2+w_3^2+(w_2\rho_3-w_3\rho_2)^2)-2\rho_1w_1(w_2\rho_2+w_3\rho_3)\big\}^{-1}.
\end{split}
\end{equation*}
It is easy to see that $|m'(w,\rho)|\approx (1+|\rho|^2)^{-1}$ and $|D^\al_wD_\rho^\beta m'(w,\rho)|\lesssim (1+|\rho|^2)^{-1-|\beta|/2}$ in the support of the integral, for all multi-indeces $\alpha$ and $\beta$ with $|\alpha|\leq 4$, $|\beta|\leq 4$. The bound \eqref{pha4.5} follows by a standard integration by parts argument, which completes the proof of \eqref{pha4.3}.

{\bf{Case 2.}} $\iota_1=\iota_2$. If $\iota_1=\iota_2=+$ then we write, as in \eqref{pha4.4},
\begin{equation*}
\begin{split}
m(v,\eta)&=\frac{1}{|v|-2^{-k}\sqrt{1+|\eta-2^kv|^2}-2^{-k}\sqrt{1+|\eta|^2}}\\
&=\frac{-\big[2^k|v|-\sqrt{1+|\eta|^2}+\sqrt{1+|\eta-2^kv|^2}\big]\big[|v|\sqrt{1+|\eta|^2}+v\cdot\eta\big]}{2[|v|^2+|v|^2|\eta|^2-(v\cdot\eta)^2]}.
\end{split}
\end{equation*}
On the other hand, if $\iota_1=\iota_2=-$ then we write, as in \eqref{pha4.4},
\begin{equation*}
\begin{split}
m(v,\eta)&=\frac{1}{|v|+2^{-k}\sqrt{1+|\eta-2^kv|^2}+2^{-k}\sqrt{1+|\eta|^2}}\\
&=\frac{\big[2^k|v|+\sqrt{1+|\eta|^2}-\sqrt{1+|\eta-2^kv|^2}\big]\big[|v|\sqrt{1+|\eta|^2}-v\cdot\eta\big]}{2[|v|^2+|v|^2|\eta|^2-(v\cdot\eta)^2]}.
\end{split}
\end{equation*}
The desired conclusion follows in both cases using \eqref{pha4.45} and the general bound \eqref{pha4.15}. In fact, since $\|\mathcal{F}^{-1}\{\varphi_0(v)(2^k|v|\pm\sqrt{1+|\eta|^2}\mp\sqrt{1+|\eta-2^kv|^2})\}\|_{L^1(\mathbb{R}^3\times\mathbb{R}^3)}\lesssim 2^k$, we get a stronger bound when $\sigma=(wa,\iota)$ and $\mu=\nu\in\{(kg,+),(kg,-)\}$,
\begin{equation}\label{pha4.9}
\big\|\mathcal{F}^{-1}\{\Phi_{\sigma\mu\nu}(\xi,\eta)^{-1}n(\xi,\eta)\cdot\varphi_k(\xi)\varphi_{k_1}(\xi-\eta)\varphi_{k_2}(\eta)\}\big\|_{L^1(\mathbb{R}^3\times\mathbb{R}^3)}\lesssim 2^{3\overline{k}^+},
\end{equation}
as desired
\end{proof}

\subsubsection{Linear estimates}\label{line} We prove now several linear estimates. 

\begin{lemma}\label{LinEstLem} (i) For any $f\in L^2(\mathbb{R}^3)$ and $(k,j)\in\mathcal{J}$ let
\begin{equation}\label{defin}
f_{j,k}:=P'_{k}Q_{jk}f,\qquad Q_{\leq jk}f:=\sum_{j'\in[\max(-k,0),j]}Q_{j'k}f,\qquad f_{\leq j,k}:=P'_{k}Q_{\leq jk}f,
\end{equation}
where $P'_k=P_{[k-2,k+2]}$. Then, for any $\alpha\in(\mathbb{Z}_+)^3$,
\begin{equation}\label{Linfty3.4}
\|D^\alpha_\xi\widehat{f_{j,k}}\|_{L^2}\lesssim 2^{|\alpha|j}\|\widehat{Q_{jk}f}\|_{L^2},\qquad \|D^\alpha_\xi\widehat{f_{j,k}}\|_{L^\infty}\lesssim 2^{|\alpha|j}\|\widehat{Q_{jk}f}\|_{L^\infty}.
\end{equation}
Moreover we have
\begin{equation}\label{Linfty3.3}
\|\widehat{f_{j,k}}\|_{L^\infty}\lesssim \min\big\{2^{3j/2} \|Q_{jk}f\|_{L^2},2^{j/2-k}2^{\delta(j+k)/8}\|Q_{jk}f\|_{H^{0,1}_\Omega}\big\},
\end{equation}
\begin{equation}\label{Linfty3.33}
\|\widehat{f_{j,k}}(r\theta)\|_{L^2(r^2dr)L^\infty_\theta}\lesssim 2^{j+k} \|Q_{jk}f\|_{L^2},
\end{equation}
\begin{equation}\label{Linfty3.34}
\|\widehat{f_{j,k}}(r\theta)\|_{L^2(r^2dr)L^p_\theta}\lesssim_p \|Q_{jk}f\|_{H^{0,1}_\Omega},\qquad p\in[2,\infty),
\end{equation}
and
\begin{equation}\label{Linfty3.31}
\|\widehat{Q_{jk}f}-\widehat{f_{j,k}}\|_{L^\infty}\lesssim 2^{3j/2}2^{-4(j+k)} \|P_kf\|_{L^2}.
\end{equation}
In particular, if 
\begin{equation}\label{consu}
\sup_{j\geq -k^-}\|Q_{jk}f\|_{H^{0,1}_\Omega}\leq A,\qquad\sup_{j\geq -k^-}2^{j+k}\|Q_{jk}f\|_{H^{0,1}_\Omega}\leq B
\end{equation}
for some $k\in\mathbb{Z}$ and $A\leq B\in [0,\infty)$, then
\begin{equation}\label{consu2}
\|\widehat{P_k f}\|_{L^\infty}\lesssim 2^{-3k/2}A^{(1-\delta)/2}B^{(1+\delta)/2}.
\end{equation}

(ii) For any $t\in\mathbb{R}$, $(k,j)\in\mathcal{J}$, and $f\in L^2(\mathbb{R}^3)$ we have
\begin{equation}\label{Linfty1}
\|e^{-it\Lambda_{wa}}f_{j,k}\|_{L^\infty}\lesssim 2^{3k/2}\min(1,2^j\langle t\rangle^{-1})\|Q_{jk}f\|_{L^2}.
\end{equation}
Moreover, if $|t|\geq 1$ and $j\geq\max(-k,0)$, then we have the stronger bounds
\begin{equation}\label{Linfty1.5}
\|\varphi_{[-100,100]}(\langle t\rangle^{-1}x)(e^{-it\Lambda_{wa}}f_{j,k})(x)\|_{L^\infty_x}\lesssim \langle t\rangle^{-1}2^{k/2}(1+\langle t\rangle 2^{k})^{\delta/8}\|Q_{jk}f\|_{H^{0,1}_\Omega};
\end{equation}
\begin{equation}\label{Linfty1.6}
\|e^{-it\Lambda_{wa}}f_{j,k}\|_{L^\infty}\lesssim \langle t\rangle^{-1}2^{k/2}(1+\langle t\rangle 2^{k})^{\delta/8}\|Q_{jk}f\|_{H^{0,1}_\Omega}\qquad \text{ if }\,2^{j}\leq 2^{-10}\langle t\rangle;
\end{equation}
\begin{equation}\label{Linfty1.1}
\|e^{-it\Lambda_{wa}}f_{\leq j,k}\|_{L^\infty}\lesssim 2^{2k}\langle t\rangle^{-1}\|\widehat{Q_{\leq jk}f}\|_{L^\infty}\qquad \text{ if }\,2^{j}\lesssim\langle t\rangle^{1/2} 2^{-k/2}.
\end{equation}

(iii) For any $t\in\mathbb{R}$, $(k,j)\in\mathcal{J}$, and $f\in L^2(\mathbb{R}^3)$ we have
\begin{equation}\label{Linfty3}
\begin{split}
\|e^{-it\Lambda_{kg}}f_{j,k}\|_{L^\infty}\lesssim \min\big\{2^{3k/2},2^{3k^+}\langle t\rangle^{-3/2}2^{3j/2}\big\}\|Q_{jk}f\|_{L^2}.
\end{split}
\end{equation}
Moreover, if $1\leq 2^{2k^--20}\langle t\rangle$ and $j\geq\max(-k,0)$, then we have the stronger bounds
\begin{equation}\label{Linfty3.6}
\|e^{-it\Lambda_{kg}}f_{j,k}\|_{L^\infty}\lesssim 2^{5k^+}\langle t\rangle^{-3/2}2^{j/2-k^-}(\langle t\rangle 2^{2k^-})^{\delta/8}\|Q_{jk}f\|_{H^{0,1}_\Omega}\qquad \text{ if }\,2^{j}\leq 2^{k^--20}\langle t\rangle;
\end{equation}
\begin{equation}\label{Linfty3.1}
\|e^{-it\Lambda_{kg}}f_{\leq j,k}\|_{L^\infty}\lesssim 2^{5k^+}\langle t\rangle^{-3/2}\|\widehat{Q_{\leq jk}f}\|_{L^\infty}\qquad \text{ if }\,2^{j}\lesssim\langle t\rangle^{1/2}.
\end{equation}

(iv) The bounds \eqref{Linfty1.6}, \eqref{Linfty1.1}, and \eqref{Linfty3.6} can be improved by using super-localization in frequency. Indeed, for $n\geq 4$ and $l\in\mathbb{Z}$ we define the operators $\mathcal{C}_{n,l}$ by
\begin{equation}\label{suploc}
\widehat{\mathcal{C}_{n,l}g}(\xi):=\chi(|\xi|2^{-l}-n)\widehat{g}(\xi),
\end{equation}
where $\chi:\mathbb{R}\to[0,1]$ is a smooth function supported in $[-2,2]$ with the property that $\sum_{n\in\mathbb{Z}}\chi(x-n)=1$ for all $x\in\mathbb{R}$. Assume that $|t|\geq 1$, $j\geq\max(-k,0)$, and $l\leq k-6$.  Then
\begin{equation}\label{Linfty1.6*}
\Big\{\sum_{n\geq 4}\|e^{-it\Lambda_{wa}}\mathcal{C}_{n,l}f_{j,k}\|_{L^\infty}^2\Big\}^{1/2}\lesssim \langle t\rangle^{-1}2^{l/2}(1+\langle t\rangle 2^{k})^{\delta/8}\|Q_{jk}f\|_{H^{0,1}_\Omega}
\end{equation}
provided that $2^{j}+2^{-l}\lesssim \langle t\rangle(1+\langle t\rangle 2^{k})^{-\delta/8}$. Moreover, if $2^{j}+2^{-l}\lesssim\langle t\rangle^{1/2} 2^{-k/2}$ then
\begin{equation}\label{Linfty1.1*}
\sup_{n\geq 4}\|e^{-it\Lambda_{wa}}\mathcal{C}_{n,l}f_{\leq j,k}\|_{L^\infty}\lesssim 2^{k}2^l\langle t\rangle^{-1}\|\widehat{Q_{\leq jk}f}\|_{L^\infty}.
\end{equation}
Finally, if $2^{j}+2^{-l}\lesssim \langle t\rangle 2^{k^-}(1+\langle t\rangle 2^{2k^-})^{-\delta/8}$ then
\begin{equation}\label{Linfty3.6*}
\Big\{\sum_{n\geq 4}\|e^{-it\Lambda_{kg}}\mathcal{C}_{n,l}f_{j,k}\|_{L^\infty}^2\Big\}^{1/2}\lesssim 2^{5k^+}\langle t\rangle^{-1}2^{l/2}2^{-k^-}(1+\langle t\rangle 2^{2k^-})^{\delta/8}\|Q_{jk}f\|_{H^{0,1}_\Omega}.
\end{equation}
\end{lemma}

The super-localized bounds in (iv) are not being used in this paper. We include them here for future reference, as they can be proved in the same way as the bounds in (ii) and (iii).

\begin{proof} (i) The bound \eqref{Linfty3.4} follows from definitions, since every $\xi$ derivative corresponds to multiplication by $x$ in the physical space. Similarly,
\begin{equation*}
\|\widehat{f_{j,k}}\|_{L^\infty}\lesssim \|\widehat{Q_{jk}f}\ast \widehat{\varphi_{\leq j+4}}\|_{L^\infty}\lesssim 2^{3j/2}\|\widehat{Q_{jk}f}\|_{L^2},
\end{equation*}
which gives the first inequality in \eqref{Linfty3.3}. A similar argument also gives \eqref{Linfty3.31}. 

Using the Sobolev embedding along the spheres $\mathbb{S}^2$, for any $g\in H^{0,1}_\Omega$ and $p\in[2,\infty)$ we have
\begin{equation}\label{triv8}
\big\Vert \widehat{g}(r\theta)\,\big\Vert_{L^2(r^2dr)L^p_\theta}
\lesssim_p \sum_{m_1+m_2+m_3\leq 1}\Vert \Omega_{23}^{m_1}\Omega_{31}^{m_2}\Omega_{12}^{m_3}\widehat{g}\Vert_{L^2}\lesssim_p \Vert \widehat{g}\Vert_{H^{0,1}_\Omega}.
\end{equation}
This gives \eqref{Linfty3.34}. Moreover, for $\xi\in\mathbb{R}^3$ with $|\xi|\approx 2^k$ we estimate
\begin{equation*}
\begin{split}
|\widehat{f_{j,k}}(\xi)|&\lesssim \int_{\mathbb{R}^3}|\widehat{Q_{jk}f}(r\theta)||\widehat{\varphi_{\leq j+4}}(\xi-r\theta)| r^2drd\theta\\
&\lesssim \|\widehat{Q_{jk}f}(r\theta)\|_{L^2(r^2dr)L^p_\theta}\|2^{3j}(1+2^j|\xi-r\theta|)^{-8}\|_{L^2(r^2dr)L^{p'}_\theta}\\
&\lesssim_p \|\widehat{Q_{jk}f}\|_{H^{0,1}_\Omega}\cdot 2^{3j}2^{-j/2}2^k2^{-2(j+k)/p'}.
\end{split}
\end{equation*}
The second bound in \eqref{Linfty3.3} follows. The proof of \eqref{Linfty3.33} is similar.

To prove \eqref{consu2} we use \eqref{Linfty3.3} to estimate
\begin{equation*}
\sum_{j\geq -k^-}\|\widehat{f_{j,k}}\|_{L^\infty}\lesssim 2^{-3k/2}\sum_{j\geq -k^-}2^{(j+k)/2}2^{\delta (j+k)/8}\|Q_{jk}f\|_{H^{0,1}_\Omega}\lesssim 2^{-3k/2}A^{(1-\delta)/2}B^{(1+\delta)/2},
\end{equation*}
and the desired conclusion follows.

We prove the remaining bounds \eqref{Linfty1}--\eqref{Linfty3.6*} in several steps.

{\bf{Step 1: proof of \eqref{Linfty1.6*} and \eqref{Linfty1.1*}.}} Let
\begin{equation}\label{bar1}
f_{j,k}^\ast:=Q_{jk}f,\quad f_{j,k;n}:=\mathcal{C}_{n,l}f_{j,k},\quad\widehat{g_{j,k;n}}(\xi):=\widehat{f_{j,k}^\ast}(\xi)\varphi_{\leq 4}(2^{-l}|\xi|-n).
\end{equation}
By orthogonality,
\begin{equation*}
\Big\{\sum_{n\geq 4}\|g_{j,k;n}\|_{H^{0,1}_\Omega}^2\Big\}^{1/2}\lesssim \|f^\ast_{j,k}\|_{H^{0,1}_\Omega}.
\end{equation*}
For \eqref{Linfty1.6*} it suffices to prove that, for any $n\geq 4$ and $x\in\mathbb{R}^3$,
\begin{equation}\label{bar2}
\Big|\int_{\mathbb{R}^3}e^{-it|\xi|}e^{ix\cdot\xi}\widehat{g_{j,k;n}}(\xi)\varphi_{[k-2,k+2]}(\xi)\chi(|\xi|2^{-l}-n)\,d\xi\Big|\lesssim \langle t\rangle^{-1}2^{l/2}(\langle t\rangle 2^{k})^{\delta/8}\|g_{j,k;n}\|_{H^{0,1}_\Omega}.
\end{equation}

This follows easily if $2^k\langle t\rangle\lesssim 1$. Recall that $2^{j}+2^{-l}\lesssim \langle t\rangle(1+\langle t\rangle 2^{k})^{-\delta/8}$ and $k\geq l+6$. The bounds \eqref{bar2} also follow directly from Lemma \ref{tech5} (integration by parts in $\xi$) if $|x|\notin[2^{-40}\langle t\rangle,2^{40}\langle t\rangle]$. 

It remains to prove \eqref{bar2} when
\begin{equation}\label{bar4}
2^k\langle t\rangle\geq 2^{50},\qquad |x|\in[2^{-50}\langle t\rangle,2^{50}\langle t\rangle].
\end{equation}
By rotation invariance we may assume $x=(x_1,0,0)$. We decompose $e^{-it\Lambda_{wa}}f_{j,k;n}(x)=\sum_{b,c\geq 0}J_{b,c}$, where
\begin{equation}\label{tric6}
\begin{split}
&J_{b,c}:=C\int_{\mathbb{R}^3}\widehat{g_{j,k;n}}(\xi)\varphi_{[k-2,k+2]}(\xi)\chi(|\xi|2^{-l}-n)e^{ix_1\xi_1-it|\xi|}\psi_{b,c}(\xi)\,d\xi,\\
&\psi_{b,c}(\xi):=\varphi_b^{[0,\infty)}(\xi_2/2^{\lambda})\varphi_c^{[0,\infty)}(\xi_3/2^{\lambda}),\qquad 2^{\lambda}:=\langle t\rangle^{-1/2}2^{k/2}.
\end{split}
\end{equation}

We estimate first $|J_{0,0}|$. For any $p\in[2,\infty)$, using also \eqref{triv8} we have
\begin{equation}\label{tric9.2}
|J_{0,0}|\lesssim \|\widehat{g_{j,k;n}}(r\theta)\|_{L^2(r^2dr)L^p_\theta}(2^{\lambda-k})^{2/p'}\cdot 2^{k}2^{l/2}\lesssim_p \|g_{j,k;n}\|_{H^{0,1}_\Omega}\cdot \langle t\rangle^{-1}2^{l/2}(\langle t\rangle 2^{k})^{1/p}.
\end{equation} 
This is consistent with the desired bound \eqref{bar2}, by taking $p$ large enough.

To estimate $|J_{b,c}|$ when $(b,c)\neq (0,0)$ we may assume without loss of generality that $b\geq c$. It suffices to show that if $b\geq\max(c,1)$ then
\begin{equation}\label{tric9}
|J_{b,c}|\lesssim \langle t\rangle^{-1}2^{l/2}(\langle t\rangle 2^{k})^{\delta/40}\|g_{j,k;n}\|_{H^{0,1}_\Omega}.
\end{equation}

We integrate by parts in the integral in \eqref{tric6}, up to three times, using the rotation vector-field $\Omega_{12}=\xi_1\partial_{\xi_2}-\xi_2\partial_{\xi_1}$. Since $\Omega_{12}\{x_1\xi_1-t|\xi|\}=-\xi_2x_1$, every integration by parts gains a factor of $|t|2^{\lambda+b}\approx \langle t\rangle^{1/2}2^{k/2+b}$ and loses a factor $\lesssim \langle t\rangle^{1/2}2^{k/2}$. If $\Omega_{12}$ hits the function $\widehat{g_{j,k;n}}$ then we stop integrating by parts and bound the integral by estimating $\Omega_{12}\widehat{g_{j,k;n}}$ in $L^2$. As in \eqref{tric9.2} it follows that
\begin{equation*}
|J_{b,c}|\lesssim \|\widehat{g_{j,k;n}}(r\theta)\|_{L^2(r^2dr)L^p_\theta}(2^{\lambda-k})^{2/p'}2^{k}2^{l/2}2^{-b}+\|\Omega_{12}\widehat{g_{j,k;n}}\|_{L^2}(2^{\lambda+b}2^{l/2})(\langle t\rangle^{1/2}2^{k/2+b})^{-1},
\end{equation*}
which gives the desired bound \eqref{tric9}. This completes the proof of the main bounds \eqref{bar2}.

The proof of \eqref{Linfty1.1*} is easier. We define $f_{\leq j,k;n}:=\mathcal{C}_{n,l}f_{\leq j,k}$ and $f_{\leq j,k}^\ast:=Q_{\leq jk}f$. For \eqref{Linfty1.1*} it suffices to prove that, for any $n\geq 4$ and $x\in\mathbb{R}^3$,
\begin{equation}\label{bar7}
\Big|\int_{\mathbb{R}^3}e^{-it|\xi|}e^{ix\cdot\xi}\widehat{f^\ast_{\leq j,k}}(\xi)\varphi_{[k-2,k+2]}(\xi)\chi(|\xi|2^{-l}-n)\,d\xi\Big|\lesssim 2^{k}2^l\langle t\rangle^{-1}\|\widehat{f^\ast_{\leq j,k}}\|_{L^\infty}.
\end{equation}
As before, we may assume $x=(x_1,0,0)$ and decompose $e^{-it\Lambda_{wa}}f_{\leq j,k;n}(x)=\sum_{b,c\geq 0}J'_{b,c}$, where
\begin{equation}\label{bar7.1}
\begin{split}
&J'_{b,c}:=C\int_{\mathbb{R}^3}\widehat{f^\ast_{\leq j,k}}(\xi)\varphi_{[k-2,k+2]}(\xi)\chi(|\xi|2^{-l}-n)e^{ix_1\xi_1-it|\xi|}\psi_{b,c}(\xi)\,d\xi,\\
&\psi_{b,c}(\xi):=\varphi_b^{[0,\infty)}(\xi_2/2^{\lambda})\varphi_c^{[0,\infty)}(\xi_3/2^{\lambda}),\qquad 2^{\lambda}:=\langle t\rangle^{-1/2}2^{k/2}.
\end{split}
\end{equation}
Using polar coordinates, it is easy to see that $|J'_{0,0}|\lesssim 2^{k}2^l\langle t\rangle^{-1}\|\widehat{f^\ast_{\leq j,k}}\|_{L^\infty}$. Then we integrate by parts in $\xi_2$ or $\xi_3$ (using the assumption $2^j+2^{-l}\lesssim \langle t\rangle 2^{\lambda}2^{-k}$) to show that 
\begin{equation*}
|J'_{b,c}|\lesssim 2^{-\max(b,c)}2^{k}2^l\langle t\rangle^{-1}\|\widehat{f^\ast_{\leq j,k}}\|_{L^\infty}
\end{equation*}
for any $b,c\geq 0$. The desired conclusion \eqref{bar7} follows.

{\bf{Step 2: proof of \eqref{Linfty1} and \eqref{Linfty1.5}.}} We start with \eqref{Linfty1.5}.  By rotation invariance we may assume $x=(x_1,0,0)$, $|x_1|\approx\langle t\rangle$. We may also assume that $2^k\langle t\rangle\geq 2^{40}$. As before we decompose $e^{-it\Lambda_{wa}}f_{j,k}(x)=\sum_{b,c\geq 0}J''_{b,c}$, where
\begin{equation}\label{tris6}
\begin{split}
&J''_{b,c}:=\int_{\mathbb{R}^3}\widehat{f_{j,k}}(\xi)\varphi_{[k-4,k+4]}(\xi)e^{ix_1\xi_1-it|\xi|}\psi_{b,c}(\xi)\,d\xi,\\
&\psi_{b,c}(\xi):=\varphi_b^{[0,\infty)}(\xi_2/2^{\lambda})\varphi_c^{[0,\infty)}(\xi_3/2^{\lambda}),\qquad 2^{\lambda}:=\langle t\rangle^{-1/2}2^{k/2}.
\end{split}
\end{equation}
 This is similar to the decomposition \eqref{tric6} with $l=k-6$, once we notice that super-localization is not important if $2^l\approx 2^k$. As in \eqref{tric9.2} and \eqref{tric9}, we have
\begin{equation*}
|J''_{0,0}|\lesssim_p \|f_{j,k}\|_{H^{0,1}_\Omega}\cdot \langle t\rangle^{-1}2^{k/2}(\langle t\rangle 2^{k})^{\delta/10},
\end{equation*} 
and, if $b\geq\max(c,1)$,
\begin{equation*}
|J''_{b,c}|\lesssim \langle t\rangle^{-1}2^{k/2}(\langle t\rangle 2^{k})^{\delta/10}\|f^\ast_{j,k}\|_{H^{0,1}_\Omega}.
\end{equation*}
The proof of this second bound uses integration by parts with the rotation vector-field $\Omega_{12}=\xi_1\partial_{\xi_2}-\xi_2\partial_{\xi_1}$, and relies on the assumption $|x_1|\approx\langle t\rangle$. The desired conclusion \eqref{Linfty1.5} follows from these two bounds.

The bounds \eqref{Linfty1} follow by the same argument, using the decomposition \eqref{tris6}, but using \eqref{Linfty3.33} instead of \eqref{Linfty3.34} in the estimate of $|J_{0,0}|$. Also, integration by parts in $\xi_2$ or $\xi_3$ is used to bound $|J_{b,c}|$ when $2^{\lambda+\max(b,c)}\gtrsim 2^{j+k}\langle t\rangle^{-1}$. 

{\bf{Step 3: proof of \eqref{Linfty1.6} and \eqref{Linfty1.1}.}} The bounds \eqref{Linfty1.1} follow directly from \eqref{Linfty1.1*} by taking $2^l\approx 2^k$. To prove \eqref{Linfty1.6} we may assume that $x=(x_1,0,0)$ and $\langle t\rangle 2^{k}\geq 2^{40}$. If $|x_1|\in [2^{-10}|t|,2^{10}|t|]$ then the desired bounds follow from \eqref{Linfty1.5}. On the other hand, if $|x_1|\leq 2^{-10}|t|$ or $|x_1|\geq 2^{10}|t|$ then we write
\begin{equation}\label{tris2.2}
[e^{-it\Lambda_{wa}}f_{j,k}](x)=C\int_{\mathbb{R}^3\times\mathbb{R}^3}Q_{jk}f(y)e^{-iy\cdot\xi}e^{ix_1\xi_1}e^{-it|\xi|}\varphi_{[k-2,k+2]}(\xi)\,d\xi dy.
\end{equation}
Here we use the fact that $2^j\leq\langle t\rangle 2^{-10}$ and integrate by parts in $\xi$ sufficiently many times (using Lemma \ref{tech5}) to see that 
\begin{equation*}
|e^{-it\Lambda_{wa}}f_{j,k}(x)|\lesssim (\langle t\rangle 2^{k})^{-4}2^{3k}2^{3j/2}\|Q_{jk}f\|_{L^2}\lesssim (\langle t\rangle 2^{k})^{-4}2^{3k}\langle t\rangle^{3/2}\|Q_{jk}f\|_{L^2},
\end{equation*}
which is better than what we need.  

{\bf{Step 4: proof of \eqref{Linfty3.6*}.}} This is similar to the proof of \eqref{Linfty1.6*}. It suffices to show that for any $n\geq 4$ and $x\in\mathbb{R}^3$,
\begin{equation}\label{lbar2}
\begin{split}
\Big|\int_{\mathbb{R}^3}e^{-it\langle\xi\rangle}e^{ix\cdot\xi}&\widehat{g_{j,k;n}}(\xi)\varphi_{[k-2,k+2]}(\xi)\chi(|\xi|2^{-l}-n)\,d\xi\Big|\\
&\lesssim 2^{5k^+}\langle t\rangle^{-1}2^{l/2}2^{-k^-}(1+\langle t\rangle 2^{2k^-})^{\delta/8}\|g_{j,k;n}\|_{H^{0,1}_\Omega}.
\end{split}
\end{equation}
This follows easily if $2^{2k^-}\langle t\rangle\lesssim 1$. Recall that $2^{j}+2^{-l}\lesssim \langle t\rangle 2^{k^-}(1+\langle t\rangle 2^{2k^-})^{-\delta/8}$ and $k\geq l+6$. The bounds \eqref{lbar2} also follow directly from Lemma \ref{tech5} if $|x|\notin[2^{-40}2^{k^-}\langle t\rangle,2^{40}2^{k^-}\langle t\rangle]$. 

It remains to prove \eqref{lbar2} when
\begin{equation}\label{lbar4}
2^{2k^-}\langle t\rangle\geq 2^{50},\qquad |x|\in[2^{-50}2^{k^-}\langle t\rangle,2^{50}2^{k^-}\langle t\rangle].
\end{equation}
We may assume $x=(x_1,0,0)$ and decompose $e^{-it\Lambda_{kg}}f_{j,k;n}(x)=\sum_{b,c\geq 0}J'''_{b,c}$, where
\begin{equation}\label{ltric6}
\begin{split}
&J'''_{b,c}:=C\int_{\mathbb{R}^3}\widehat{g_{j,k;n}}(\xi)\varphi_{[k-2,k+2]}(\xi)\chi(|\xi|2^{-l}-n)e^{ix_1\xi_1-it\langle\xi\rangle}\psi'_{b,c}(\xi)\,d\xi,\\
&\psi'_{b,c}(\xi):=\varphi_b^{[0,\infty)}(\xi_2/2^{\lambda'})\varphi_c^{[0,\infty)}(\xi_3/2^{\lambda'}),\qquad 2^{\lambda'}:=\langle t\rangle^{-1/2}2^{k^+}.
\end{split}
\end{equation}
As in the proof of \eqref{Linfty1.6*}, we estimate first $|J'''_{0,0}|$, using \eqref{triv8}. Thus, for any $p\in[2,\infty)$,
\begin{equation*}
|J'''_{0,0}|\lesssim \|\widehat{g_{j,k;n}}(r\theta)\|_{L^2(r^2dr)L^p_\theta}(2^{\lambda'-k})^{2/p'}2^{k}2^{l/2}\lesssim_p \|g_{j,k;n}\|_{H^{0,1}_\Omega}\cdot 2^{5k^+}2^{-k^-}\langle t\rangle^{-1}2^{l/2}(\langle t\rangle 2^{2k^-})^{1/p}.
\end{equation*} 
Moreover, if $b\geq\max(c,1)$ then we show that
\begin{equation}\label{ltric9}
|J'''_{b,c}|\lesssim 2^{5k^+}\langle t\rangle^{-1}2^{l/2}2^{-k^-}(\langle t\rangle 2^{2k^-})^{\delta/10}\|g_{j,k;n}\|_{H^{0,1}_\Omega}.
\end{equation}
These two bounds clearly suffice to prove \eqref{lbar2}.

To prove \eqref{ltric9} we integrate by parts in the integral in \eqref{ltric6}, up to three times, using the rotation vector-field $\Omega_{12}=\xi_1\partial_{\xi_2}-\xi_2\partial_{\xi_1}$. Since $\Omega_{12}\{x_1\xi_1-t\langle\xi\rangle\}=-\xi_2x_1$, every integration by parts gains a factor of $2^{k^-}|t|2^{\lambda'+b}\approx \langle t\rangle^{1/2}2^{k+b}$ (see \eqref{lbar4}) and loses a factor $\lesssim \langle t\rangle^{1/2}2^{k}$. If $\Omega_{12}$ hits the function $\widehat{g_{j,k;n}}$ then we stop integrating by parts and bound the integral by estimating $\Omega_{12}\widehat{g_{j,k;n}}$ in $L^2$. As before it follows that
\begin{equation*}
|J'''_{b,c}|\lesssim \|\widehat{g_{j,k;n}}(r\theta)\|_{L^2(r^2dr)L^p_\theta}(2^{\lambda'-k})^{2/p'}2^{k}2^{l/2}2^{-b}+\|\Omega_{12}\widehat{g_{j,k;n}}\|_{L^2}(2^{\lambda'+b}2^{l/2})(\langle t\rangle^{1/2}2^{k+b})^{-1},
\end{equation*}
which gives the desired bound \eqref{ltric9}. This completes the proof of the main bounds \eqref{lbar2}.

{\bf{Step 5: proof of \eqref{Linfty3}--\eqref{Linfty3.1}.}} Clearly $\|e^{-it\Lambda_{kg}}f_{j,k}\|_{L^\infty}\lesssim \|\widehat{f_{j,k}}\|_{L^1}\lesssim 2^{3k/2}\|f_{j,k}\|_{L^2}$. Moreover, the standard dispersive bounds
\begin{equation*}
\|e^{-it\Lambda_{kg}}P_{\leq k}\|_{L^1\to L^\infty}\lesssim (1+|t|)^{-3/2}2^{3k^+}
\end{equation*}
can then be used to prove \eqref{Linfty3}, i.e.
\begin{equation*}
\|e^{-it\Lambda_{kg}}f_{j,k}\|_{L^\infty}\lesssim (1+|t|)^{-3/2}2^{3k^+}\|Q_{jk}f\|_{L^1}\lesssim (1+|t|)^{-3/2}2^{3k^+}2^{3j/2}\|Q_{jk}f\|_{L^2}.
\end{equation*}

To prove \eqref{Linfty3.6} we consider first the harder case $2^j\geq \langle t\rangle^{1/2}$. By rotation invariance we may assume $x=(x_1,0,0)$, $x_1/t>0$. We may also assume that $2^{j+k}\geq 2^{3k^++10}$ (otherwise the desired conclusion follows from \eqref{Linfty3}) and $\langle t\rangle2^{-3k^+}\gg 1$. If $|x_1|\leq 2^{-100}|t|2^{k^-}$ or $|x_1|\geq 2^{100}|t|2^{k^-}$ then we write
\begin{equation}\label{triv2.2}
[e^{-it\Lambda_{kg}}f_{j,k}](x)=C\int_{\mathbb{R}^3\times\mathbb{R}^3}Q_{jk}f(y)e^{-iy\cdot\xi}e^{ix_1\xi_1}e^{-it\sqrt{|\xi|^2+1}}\varphi_{[k-2,k+2]}(\xi)\,d\xi dy.
\end{equation}
We integrate by parts in $\xi$ sufficiently many times (using Lemma \ref{tech5} and recalling that $|y|\leq 2^{j+1}\leq 2^{k^--19}\langle t\rangle$) to see that 
\begin{equation*}
|e^{-it\Lambda_{kg}}f_{j,k}(x)|\lesssim (\langle t\rangle 2^{2k^-})^{-4}2^{3k}2^{3j/2}\|Q_{jk}f\|_{L^2}.
\end{equation*}
This is better than what we need. 

It remains to consider the main case $|x_1|\approx |t|2^{k^-}$. Let $\rho\in(0,\infty)$ denote the unique number with the property that $t\rho/\sqrt{\rho^2+1}=x_1$, such that $(\rho,0,0)$ is the stationary point of the phase $\xi\to x_1\xi_1-t\sqrt{|\xi|^2+1}$ and $\rho\gtrsim 2^{k^-}$. Using integration by parts (Lemma \ref{tech5}), we may assume that $\xi_1,\xi_2,\xi_3$ are restricted to $|\xi_2|,|\xi_3|\leq 2^{k-10}$ and $\xi_1\in [2^{k-10},2^{k+10}]$ (for the other contributions we can use the formula \eqref{triv2.2} and get stronger bounds as before). Then we let
\begin{equation}\label{triv6}
\begin{split}
&J_{a,b,c}:=\int_{\mathbb{R}^3}\widehat{f_{j,k}}(\xi)\varphi_{[k-4,k+4]}(\xi)\mathbf{1}_{+}(\xi_1)\varphi_{\leq k-9}(\xi_2)\varphi_{\leq k-9}(\xi_3) e^{ix_1\xi_1-it\sqrt{|\xi|^2+1}}\psi_{a,b,c}(\xi)\,d\xi,\\
&\psi_{a,b,c}(\xi):=\varphi_a^{[0,\infty)}((\xi_1-\rho)/2^{\lambda_1})\varphi_b^{[0,\infty)}(\xi_2/2^{\lambda_2})\varphi_c^{[0,\infty)}(\xi_3/2^{\lambda_2}),
\end{split}
\end{equation} 
where, for some sufficiently large constant $C$,
\begin{equation}\label{triv7}
2^{\lambda_1}:=2^j\langle t\rangle^{-1}2^{3k^++C}(\langle t\rangle 2^{2k^-})^{\delta/20},\qquad 2^{\lambda_2}:=\langle t\rangle^{-1/2}2^{k^+}.
\end{equation} 
Compared to the earlier decompositions, such as \eqref{tric6}, we notice that we insert an additional decomposition in the variable $\xi_1$ around the stationary point $(\rho,0,0)$.

Recall that $2^j\geq\langle t\rangle^{1/2}$. We estimate first $|J_{0,0,0}|$, using \eqref{Linfty3.34}, for any $p\in[2,\infty)$,
\begin{equation}\label{triv9.2}
\begin{split}
|J_{0,0,0}|&\lesssim_p \|\widehat{f_{j,k}}(r\theta)\|_{L^2(r^2dr)L^p_\theta}(2^{\lambda_2-k})^{2/p'}2^k2^{\lambda_1/2}\\
&\lesssim_p \|f_{j,k}\|_{H^{0,1}_\Omega}\langle t\rangle^{-3/2}2^{j/2}2^{-k^-}2^{4k^+}(\langle t\rangle 2^{2k^-})^{1/p+\delta/20}.
\end{split}
\end{equation} 
This is consistent with the desired bound \eqref{Linfty3.6}, by taking $p$ large enough.

To estimate $|J_{a,b,c}|$ when $(a,b,c)\neq (0,0,0)$ we may assume without loss of generality that $b\geq c$. If $2^{\lambda_2+b}\geq 2^j\langle t\rangle^{-1}2^{k^+}(\langle t\rangle 2^{2k^-})^{\delta/40}$ then we integrate by parts in $\xi_2$ many times, using Lemma \ref{tech5}, to show that
\begin{equation*}
|J_{a,b,c}|\lesssim \|f_{j,k}\|_{L^2}(\langle t\rangle 2^{2k^-})^{-4}2^{3k/2},
\end{equation*}
which is better than what we need. This bound also holds, using integration by parts in $\xi_1$, if $2^{\lambda_2+b}\leq 2^j\langle t\rangle^{-1}2^{k^+}(\langle t\rangle 2^{2k^-})^{\delta/40}$ and $a\geq 1$. It remains to prove that
\begin{equation}\label{triv9}
|J_{0,b,c}|\lesssim 2^{5k^+}\langle t\rangle^{-3/2}2^{j/2-k^-}(\langle t\rangle 2^{2k^-})^{\delta/10}\|Q_{jk}f\|_{H^{0,1}_\Omega}
\end{equation}
provided that
\begin{equation}\label{triv9.1}
b\geq \max(c,1)\qquad\text{ and }\qquad 2^{\lambda_2+b}\leq 2^j\langle t\rangle^{-1}2^{k^+}(\langle t\rangle 2^{2k^-})^{\delta/40}.
\end{equation}

To prove \eqref{triv9} we integrate by parts in \eqref{triv6}, up to three times, using the rotation vector-field $\Omega_{12}=\xi_1\partial_{\xi_2}-\xi_2\partial_{\xi_1}$. Since $\Omega_{12}\{x_1\xi_1-t\sqrt{|\xi|^2+1}\}=-\xi_2x_1$, every integration by parts gains a factor of $|t|2^{k^-}2^{\lambda_2+b}\approx \langle t\rangle^{1/2}2^{k+b}$ and loses a factor $\lesssim \langle t\rangle^{1/2}2^{k}$. If $\Omega_{12}$ hits the function $\widehat{f_{j,k}}$ then we stop integrating by parts and bound the integral by estimating $\Omega_{12}\widehat{f_{j,k}}$ in $L^2$. As in \eqref{triv9.2} it follows that
\begin{equation*}
|J_{0,b,c}|\lesssim_p \|\widehat{f_{j,k}}(r\theta)\|_{L^2(r^2dr)L^p_\theta}(2^{\lambda_2-k})^{2/p'}2^k2^{\lambda_1/2}2^{-b}+\|\Omega_{12}\widehat{f_{j,k}}\|_{L^2}2^{\lambda_2}2^{\lambda_1/2}(\langle t\rangle^{1/2}2^{k})^{-1},
\end{equation*}
which gives the desired bound \eqref{triv9}. This completes the proof of \eqref{Linfty3.6} when $2^j\geq \langle t\rangle^{1/2}$.

The bound \eqref{Linfty3.1} follows by a similar argument. We decompose the integral dyadically around the critical point $(\rho,0,0)$, as in \eqref{triv6} with $2^{\lambda_1}=\langle t\rangle^{-1/2}2^{3k^++C}$ and $2^{\lambda_2}=\langle t\rangle^{-1/2}2^{k^+}$, and integrate by parts four times either in $\xi_1$, or in $\xi_2$, or in $\xi_3$. 

The bound \eqref{Linfty3.6} when $2^j\leq\langle t\rangle^{1/2}$ follows from \eqref{Linfty3.1} using also \eqref{Linfty3.3}.
\end{proof}

We prove now a Hardy-type estimate involving localization in frequency and space.

\begin{lemma}\label{hyt1}
For $f\in L^2(\mathbb{R}^3)$ and $k\in\mathbb{Z}$ let
\begin{equation}\label{hyt2}
\begin{split}
&A_k:=\|P_kf\|_{L^2}+\sum_{l=1}^3\|\varphi_k(\xi)(\partial_{\xi_l}\widehat{f})(\xi)\|_{L^2_\xi},\qquad B_k:=\Big[\sum_{j\geq\max(-k,0)}2^{2j}\|Q_{jk}f\|_{L^2}^2\Big]^{1/2}.
\end{split}
\end{equation} 
Then, for any $k\in\mathbb{Z}$,
\begin{equation}\label{hyt3}
A_k\lesssim\sum_{|k'-k|\leq 4}B_{k'}
\end{equation}
and
\begin{equation}\label{hyt3.1}
B_k\lesssim
\begin{cases}
\sum_{|k'-k|\leq 4}A_{k'}\qquad&\text{ if } k\geq 0,\\
\sum_{k'\in\mathbb{Z}}A_{k'}2^{-|k-k'|/2}\min(1,2^{k'-k})\qquad&\text{ if } k\leq 0.
\end{cases}
\end{equation}
\end{lemma}

\begin{proof} Clearly, by almost orthogonality,
\begin{equation}\label{hyt4}
\begin{split}
B_k&\approx 2^{\max(-k,0)}\|P_kf\|_{L^2}+\||x|\cdot P_kf\|_{L^2}\\
&\approx 2^{\max(-k,0)}\|P_kf\|_{L^2}+\sum_{l=1}^3\|\partial_{\xi_l}(\varphi_k(\xi)\widehat{f}(\xi))\|_{L^2_\xi}.
\end{split}
\end{equation}
The bound \eqref{hyt3} follows. The bound in \eqref{hyt3.1} also follows when $k\geq 0$. On the other hand, if $k\leq 0$ then it suffices to prove that
\begin{equation}\label{hyt5}
2^{-k}\|P_kf\|_{L^2}\lesssim \sum_{k'\in\mathbb{Z}}A_{k'}2^{-|k-k'|/2}\min(1,2^{k'-k}).
\end{equation}

For this we let $f_l:=x_l f$, $l\in\{1,2,3\}$, so
\begin{equation*}
f=\frac{1}{|x|^2+1}f+\sum_{l=1}^3\frac{x_l}{|x|^2+1}f_l
\end{equation*}
and, for any $k'\in\mathbb{Z}$,
\begin{equation*}
\|P_{k'}f\|_{L^2}+\sum_{l=1}^3\|P_{k'}f_l\|_{L^2}\lesssim A_{k'}.
\end{equation*}

Since $|\mathcal{F}\{(x^2+1)^{-1}\}(\xi)|\lesssim |\xi|^{-2}$ and $|\mathcal{F}\{x_l(x^2+1)^{-1}\}(\xi)|\lesssim |\xi|^{-2}$ for $l\in\{1,2,3\}$, for \eqref{hyt5} it suffices to prove that
\begin{equation}\label{hyt7}
2^{-k}\|\varphi_k(\xi)(g\ast K)(\xi)\|_{L^2}\lesssim \sum_{k'\in\mathbb{Z}}A_{k'}2^{-|k-k'|/2}\min(1,2^{k'-k}),
\end{equation}
provided that $\|\varphi_{k'}\cdot g\|_{L^2}\lesssim A_{k'}$ and $K(\eta)=|\eta|^{-2}$. With $g_{k'}=\varphi_{k'}\cdot g$ we estimate 
\begin{equation*}
\|\varphi_k(\xi)(g_{k'}\ast K)(\xi)\|_{L^2}\lesssim \|g_{k'}\|_{L^2}\|K\cdot\varphi_{\leq k+10}\|_{L^1}\lesssim 2^k\|g_{k'}\|_{L^2}\quad\text{ if }|k-k'|\leq 6;
\end{equation*} 
\begin{equation*}
\|\varphi_k(\xi)(g_{k'}\ast K)(\xi)\|_{L^2}\lesssim 2^{3k/2}\|g_{k'}\|_{L^2}\|K\cdot\varphi_{[k'-4,k'+4]}\|_{L^2}\lesssim 2^{3k/2}2^{-k'/2}\|g_{k'}\|_{L^2} \quad\text{ if }k'\geq k+6;
\end{equation*}
\begin{equation*}
\|\varphi_k(\xi)(g_{k'}\ast K)(\xi)\|_{L^2}\lesssim \|g_{k'}\|_{L^1}\|K\cdot\varphi_{[k-4,k+4]}\|_{L^2}\lesssim 2^{3k'/2}2^{-k/2}\|g_{k'}\|_{L^2} \quad\text{ if }k'\leq k-6.
\end{equation*}
The desired bound \eqref{hyt7} follows, which completes the proof of the lemma.
\end{proof}

\section{Elliptic estimates}\label{dtv}

In this section we prove several bounds on the functions $V^{wa}_{\mathcal{L}}$, $V^{kg}_{\mathcal{L}}$, $\mathcal{N}^{wa}_{\mathcal{L}}$ and $\mathcal{N}^{kg}_{\mathcal{L}}$ at fixed times $t\in[0,T]$. These bounds are used in the energy estimates and the normal form arguments in the next sections.

\subsection{Bounds on the profiles $V^{wa}_\mathcal{L}$ and $V^{kg}_{\mathcal{L}}$} Recall the definitions \eqref{variables4L} and the bootstrap assumptions \eqref{bootstrap2.1}--\eqref{bootstrap2.4}. For $\mu\in\{(wa,+),(wa,-), (kg,+), (kg,-)\}$, $\mathcal{L}\in\mathcal{V}_{n}$, $n\in\{0,\ldots,N_1\}$, $t\in[0,T]$, $(j,k)\in\mathcal{J}$, and $J\geq\max(-k,0)$ we define the localized profiles 
\begin{equation}\label{abc20.5}
V_{j,k;\mathcal{L}}^\mu(t):=P_{[k-2,k+2]}Q_{jk}V_{\mathcal{L}}^\mu(t),\qquad V_{\leq J,k;\mathcal{L}}^\mu(t):=\sum_{j\leq J}V_{j,k;\mathcal{L}}^\mu(t),\qquad V_{>J,k;\mathcal{L}}^\mu(t):=\sum_{j>J}V_{j,k;\mathcal{L}}^\mu(t).
\end{equation}
For simplicity of notation, we write sometimes $V_{j,k}^\mu$, $V_{\leq J,k}^\mu$, and $V_{>J,k}^\mu$ to denote the corresponding functions $V_{j,k;\mathcal{L}}^\mu$, $V_{\leq J,k;\mathcal{L}}^\mu$ and $V_{>J,k;\mathcal{L}}^\mu$ when $\mathcal{L}=Id$.

\begin{lemma}\label{dtv6}
Assume that $(u,v)$ is a solution to \eqref{on1} on some time interval $[0,T]$, $T\geq 1$, satisfying the bounds \eqref{bootstrap2.1}--\eqref{bootstrap2.4} in Proposition \ref{bootstrap}. Assume that $\mathcal{L}\in\mathcal{V}_{n}$, $n\in\{0,\ldots,N_1\}$. 

(i) For any $t\in[0,T]$ we have
\begin{equation}\label{abc21}
\||\nabla|^{-1/2}V^{wa}_{\mathcal{L}}(t)\|_{H^{N(n)}}+\|V^{kg}_{\mathcal{L}}(t)\|_{H^{N(n)}}\lesssim \varep_1\langle t\rangle^{H(n)\delta}.
\end{equation}
Moreover, if $n\leq N_1-1$, $k\in\mathbb{Z}$, and $l\in\{1,2,3\}$ then
\begin{equation}\label{abc21.5}
2^{k/2}\|\varphi_k(\xi)(\partial_{\xi_l}\widehat{V^{wa}_{\mathcal{L}}})(\xi,t)\|_{L^2_\xi}+2^{k^+}\|\varphi_k(\xi)(\partial_{\xi_l}\widehat{V^{kg}_{\mathcal{L}}})(\xi,t)\|_{L^2_\xi}\lesssim\varep_1Y(k,t;n),
\end{equation}
where
\begin{equation}\label{abc26.4}
Y(k,t;n):=\langle t\rangle^{H(n+1)\delta}2^{-N(n+1)k^+}.
\end{equation}
As a consequence, if $n\leq N_1-1$ and $(k,j)\in\mathcal{J}$ then
\begin{equation}\label{abc22}
2^j2^{k/2}\|Q_{jk}V^{wa}_{\mathcal{L}}(t)\|_{L^2}+2^j2^{k^+}\|Q_{jk}V^{kg}_{\mathcal{L}}(t)\|_{L^2}\lesssim\varep_1Y(k,t;n).
\end{equation}
In particular, if $n\leq N_1-1$ and $k\in\mathbb{Z}$ then
\begin{equation}\label{abc22.1}
2^{k/2}\|P_kV^{wa}_{\mathcal{L}}(t)\|_{L^2}+2^{k^+}\|P_kV^{kg}_{\mathcal{L}}(t)\|_{L^2}\lesssim\varep_12^{k^-}\langle t\rangle^{H(n+1)\delta}2^{-N(n+1)k^+}.
\end{equation}

(ii) As a consequence, if $n\leq N_1-1$ then for any $t\in[0,T]$ and $k\in\mathbb{Z}$
\begin{equation}\label{abc23.2}
\sum_{j\geq -k^-}\|e^{-it\Lambda_{wa}}V^{wa,+}_{j,k;\mathcal{L}}(t)\|_{L^\infty}\lesssim\varep_1Y(k,t;n)\langle t\rangle^{\delta/2}2^{2k^+}2^{k^-}\min(\langle t\rangle^{-1},2^{k^-}),
\end{equation}
and
\begin{equation}\label{abc23}
\sum_{j\geq -k^-}\|e^{-it\Lambda_{kg}}V^{kg,+}_{j,k;\mathcal{L}}(t)\|_{L^\infty}\lesssim\varep_1Y(k,t;n)2^{2k^+}2^{k^-/2}\min(\langle t\rangle^{-1},2^{2k^-}).
\end{equation}
Moreover, if $n\leq N_1-2$ and $2^{2k^--20}\langle t\rangle\geq 1$ then 
\begin{equation}\label{abc23.1}
\sum_{2^j\in[2^{-k^-},2^{k^--20}\langle t\rangle]}\|e^{-it\Lambda_{kg}}V^{kg,+}_{j,k;\mathcal{L}}(t)\|_{L^\infty}\lesssim \varep_1Y(k,t;n+1)\langle t\rangle^{-3/2+\delta/2}2^{-k^-/2}2^{6k^+}.
\end{equation}

(iii) In the case $n=0$ (so $\mathcal{L}=Id$), we have the stronger bounds
\begin{equation}\label{vcx1.1}
\begin{split}
\sum_{j\geq -k^-}2^j\|Q_{jk}V^{wa}(t)\|_{L^2}&\lesssim \varep_12^{-k^-/2-\kappa k^-}2^{-N_0k^++d'k^+},\\
\|\widehat{P_kV^{kg}}(t)\|_{L^\infty}&\lesssim \varep_12^{-k^-/2+\kappa k^-}2^{-N_0k^++d'k^+},\\
\|P_kV^{kg}(t)\|_{L^2}&\lesssim \varep_12^{k^-+\kappa k^-}2^{-(N(0)-2)k^+},
\end{split}
\end{equation}
for any $k\in\mathbb{Z}$ and $t\in[0,T]$. As a consequence 
\begin{equation}\label{vcx3.1}
\sum_{j\geq -k^-}\|e^{-it\Lambda_{wa}}V^{wa,+}_{j,k}(t)\|_{L^\infty}\lesssim \varep_12^{k^-(1-\kappa)}\min\{\langle t\rangle^{-1},2^{k^-}\}2^{-N_0k^++(d'+3)k^+}.
\end{equation}
Moreover, if $\langle t\rangle\geq 2^{-2k^-+20}$ and $2^J\in [2^{-k^-},2^{k^--20}\langle t\rangle]$ then
\begin{equation}\label{vcx5}
\|e^{-it\Lambda_{kg}}V^{kg,+}_{\leq J,k}(t)\|_{L^\infty}\lesssim \varep_1\langle t\rangle^{-3/2}2^{-k^-/2+\kappa k^-/20}2^{-N_0k^++(d'+6)k^+}.
\end{equation}
\end{lemma} 

\begin{proof} All the bounds in the lemma follow easily from the bootstrap assumptions \eqref{bootstrap2.1}--\eqref{bootstrap2.4}, and Lemmas \ref{LinEstLem} and \ref{hyt1}. Indeed, the bounds \eqref{abc21}, \eqref{abc21.5}, and \eqref{abc22} follow directly from the bootstrap assumptions \eqref{bootstrap2.1}--\eqref{bootstrap2.2} and the bounds \eqref{hyt3.1}. The bounds \eqref{abc22.1} follow from \eqref{abc22} by summation over $j\geq -k^-$. The bounds \eqref{vcx1.1} follow from \eqref{bootstrap2.4} and Definition \ref{MainZDef}. 

The dispersive bounds \eqref{abc23.2} and \eqref{vcx3.1} follow from \eqref{Linfty1} and the bounds \eqref{abc22} and \eqref{vcx1.1} respectively. The bounds \eqref{abc23} follow from \eqref{Linfty3} and \eqref{abc22}, by summation over $j$. The bounds \eqref{abc23.1} follow from \eqref{Linfty3.6}, once we notice that, as a consequence of \eqref{abc22}, for $|\alpha|\leq 1$ 
\begin{equation}\label{vcx5.1}
\|Q_{jk}\Omega^\alpha V_{\mathcal{L}}^{kg}\|_{L^2}\lesssim \varep_1Y(k,t;n+|\alpha|)2^{-j}2^{-k^+}.
\end{equation}
The bounds \eqref{vcx5} follow directly from \eqref{abc21} if $2^{k}\gtrsim\langle t\rangle^{1/(3d)}$ and from \eqref{vcx1.1} if $2^k\lesssim \langle t\rangle^{-1/2+\kappa/8}$. They also follow from \eqref{Linfty3.1} and \eqref{vcx1.1} if $2^J\leq\langle t\rangle^{1/2}$. Finally, if $2^{k}\in[\langle t\rangle^{-1/2+\kappa/8},\langle t\rangle^{1/(3d)}]$ and $2^J\geq\langle t\rangle^{1/2}$ then we use \eqref{Linfty3.6} and \eqref{vcx5.1} to estimate the remaining contribution by
\begin{equation*}
\begin{split}
\sum_{2^j\in[\langle t\rangle^{1/2},2^J]}&\|e^{-it\Lambda_{kg}}V^{kg,+}_{j,k}(t)\|_{L^\infty}\lesssim \sum_{2^j\in[\langle t\rangle^{1/2},2^J]}2^{5k^+}\langle t\rangle^{-3/2}2^{j/2-k^-}\langle t\rangle^{\delta}\cdot\varep_1Y(k,t;1)2^{-j}2^{-k^+}\\
&\lesssim \varep_1\langle t\rangle^{-3/2}2^{-k^-}2^{-N_0k^++2dk^++5k^+}\langle t\rangle^{\delta+H(2)\delta-1/4}.
\end{split}
\end{equation*}
This suffices to complete the proof of \eqref{vcx5} in the range $2^{k}\in[\langle t\rangle^{-1/2+\kappa/8},\langle t\rangle^{1/(3d)}]$.

We remark that the bounds \eqref{vcx3.1} and \eqref{vcx5} are the only dispersive bounds that provide sharp rates of decay $\langle t\rangle^{-1}$ and $\langle t\rangle^{-3/2}$ for significant parts of the normalized solutions $U^{wa}(t)$ and $U^{kg}(t)$ respectively. All the other pointwise bounds involve small $\langle t\rangle^{C\delta}$ losses.
\end{proof}

\subsection{The nonlinearities $\mathcal{N}_\mathcal{L}^{wa}$} We prove now several bounds on the nonlinearities $\mathcal{N}_\mathcal{L}^{wa}$. 

\begin{lemma}\label{dtv7}
Assume that $\mathcal{N}_{\mathcal{L}}^{wa}$ is as in \eqref{on6.1L}, $\mathcal{L}\in\mathcal{V}_{n}$, $n\in\{0,\ldots,N_1\}$, $t\in[0,T]$, and $k\in\mathbb{Z}$. 

(i) Then
\begin{equation}\label{abc30}
\|P_k\mathcal{N}_{\mathcal{L}}^{wa}(t)\|_{L^2}+\|P_k\partial_tV_{\mathcal{L}}^{wa}(t)\|_{L^2}\lesssim\varep_1^22^{k^-/2}\langle t\rangle^{H_{wa}''(n)\delta}\min(2^{k^-},\langle t\rangle^{-1})2^{-N(n)k^++5k^+},
\end{equation}
where
\begin{equation}\label{abc30.1}
H_{wa}''(0):=5,\qquad H_{wa}''(n):=H(n)+160\text{ for }n\in\{1,\ldots,N_1\}.
\end{equation}

(ii) Moreover for $l\in\{1,2,3\}$ and $n\leq N_1-1$,
\begin{equation}\label{abc30.3}
\|P_k(x_l\mathcal{N}_{\mathcal{L}}^{wa})(t)\|_{L^2}\lesssim\varep_1^22^{k^-/2}\langle t\rangle^{H_{wa}''(n)\delta}2^{-N(n)k^++5k^+}.
\end{equation}
\end{lemma}

\begin{proof} (i) For $k\in\mathbb{Z}$ let
\begin{equation}\label{cxz6}
\mathcal{X}_k:=\big\{(k_1,k_2)\in\mathbb{Z}^2:\,|\max(k_1,k_2)-k|\leq 6\text{ or }(\max(k_1,k_2)\geq k+7\text{ and }|k_1-k_2|\leq 6)\big\}.
\end{equation}
Let $m$ denote generic multipliers that satisfy the bounds
\begin{equation}\label{abc9}
\|\mathcal{F}^{-1}(\varphi_k\cdot D_{\xi}^\alpha m)\|_{L^1_x}\lesssim_{\alpha} 2^{-|\alpha|k^-},\qquad\text{ for any }\,k\in\mathbb{Z}\,\text{ and }\,\alpha\in\mathbb{Z}_+^3.
\end{equation}
For $m_1,m_2$ as in \eqref{abc9}, let $I$ denote a bilinear operator of the form
\begin{equation}\label{abc36}
\widehat{I[f,g]}(\xi):=\int_{\mathbb{R}^3}m_1(\xi-\eta)m_2(\eta)\widehat{f}(\xi-\eta)\widehat{g}(\eta)\,d\eta.
\end{equation}

Notice that for $\beta\in\{0,1,2,3\}$ and $\mathcal{L}^\ast\in\{\Gamma_1,\Gamma_2,\Gamma_3,\Omega_{23},\Omega_{31},\Omega_{12}\}$ we have
\begin{equation}\label{abc4}
[\mathcal{L}^\ast,\partial_\beta]=\sum_{\gamma\in\{0,1,2,3\}}c_{\mathcal{L}^\ast,\beta}^\gamma\partial_{\gamma},
\end{equation}
for suitable coefficients $c_{\mathcal{L}^\ast,\beta}^\gamma\in\mathbb{R}$. Clearly $\|P_k\mathcal{N}_{\mathcal{L}}^{wa}(t)\|_{L^2}\approx\|P_k\partial_tV_{\mathcal{L}}^{wa}(t)\|_{L^2}$. Recall also the identities \eqref{on5} and the definitions $\mathcal{N}^{wa}_{\mathcal{L}}=\mathcal{L}[A^{\alpha\beta}\partial_\alpha v\partial_\beta v+Dv^2]$. For \eqref{abc30} it suffices to prove that, with $I$ defined as in \eqref{abc9}--\eqref{abc36},
\begin{equation}\label{abc37}
\begin{split}
&\sum_{(k_1,k_2)\in\mathcal{X}_k}I^{wa}_{k,k_1,k_2}(t)\lesssim\varep_1^22^{k^-/2}\langle t\rangle^{H_{wa}''(n)\delta}\min(2^{k^-},\langle t\rangle^{-1})2^{-N(n)k^++5k^+},\\
&I^{wa}_{k,k_1,k_2}(t):=\|P_kI[P_{k_1}U^{kg,\iota_1}_{\mathcal{L}_1},P_{k_2}U^{kg,\iota_2}_{\mathcal{L}_2}](t)\|_{L^2},
\end{split}
\end{equation}
for any $\iota_1,\iota_2\in\{+,-\}$, $\mathcal{L}_1\in\mathcal{V}_{n_1}$, $\mathcal{L}_2\in\mathcal{V}_{n_2}$, $n_1+n_2\leq n$. Without loss of generality, we may assume that $n_1\leq n_2$. To prove \eqref{abc37} we consider several cases.
\medskip

{\bf{Case 1:}} Assume first that $(n_1,n_2)\neq (0,0)$ and $k\geq -10$. Let
\begin{equation}\label{abc50}
S_1:=\sum_{(k_1,k_2)\in\mathcal{X}_k,\,k_1\leq k_2+10}I^{wa}_{k,k_1,k_2}(t),\qquad S_2:=\sum_{(k_1,k_2)\in\mathcal{X}_k,\,k_2\leq k_1-10}I^{wa}_{k,k_1,k_2}(t).
\end{equation}
Using \eqref{abc21} and \eqref{abc23} we estimate
\begin{equation}\label{abc50.1}
\begin{split}
S_1&\lesssim \sum_{(k_1,k_2)\in\mathcal{X}_k,\,k_1\leq k_2+10}\varep_1^2\langle t\rangle^{H(n_2)\delta}2^{-N(n_2)k_2^+}\langle t\rangle^{H(n_1+1)\delta-1}2^{-N(n_1+1)k_1^++2k_1^+}2^{k_1^-/2}\\
&\lesssim \varep_1^2\langle t\rangle^{H(n_2)\delta+H(n_1+1)\delta-1}2^{-N(n_2)k^+}.
\end{split}
\end{equation}
Similarly, since $N(n)\leq N(n_1+1)$,
\begin{equation}\label{abc50.2}
S_2\lesssim \varep_1^2\langle t\rangle^{\delta/2+H(n_2)\delta+H(n_1+1)\delta-1}2^{-N(n)k^++2k^+}.
\end{equation}
Notice that, as a consequence of \eqref{fvc1}, if $n_1,n_2,n_1+n_2\in[0,N_1-1]\cap\mathbb{Z}$, and $n_2\neq 0$ then
\begin{equation}\label{abc50.3}
H(n_1+1)+H(n_2)\leq H(n_1+n_2)+10.
\end{equation}
Thus $S_1+S_2\lesssim \varep_1^2\langle t\rangle^{H_{wa}''(n)\delta-1}2^{-N(n)k^++4k^+}$, as desired.
\medskip

{\bf{Case 2:}} Assume now that $(n_1,n_2)\neq (0,0)$ and $k\leq -10$.{\footnote{One should think of $(n_1,n_2,n)=(0,N_1,N_1)$ as the worst case. In this case the only available bounds for the profiles $U^{kg,\iota_2}_{\mathcal{L}_2}$ are the $L^2$ bounds in \eqref{abc21}.}} Notice that, as a consequence of \eqref{abc21} and \eqref{abc22.1},
\begin{equation}\label{ku1}
\begin{split}
I^{wa}_{k,k_1,k_2}(t)&\lesssim 2^{3\min(k,k_1,k_2)/2}\|P_{k_1}U^{kg,\iota_1}_{\mathcal{L}_1}\|_{L^2}\|P_{k_2}U^{kg,\iota_2}_{\mathcal{L}_2}\|_{L^2}\\
&\lesssim \varep_1^22^{-4(k_1^++k_2^+)}2^{3\min(k,k_1,k_2)/2}2^{k_1}\langle t\rangle^{H(n_1+1)\delta+H(n_2)\delta}.
\end{split}
\end{equation}
Also, as a consequence of \eqref{abc21} and \eqref{abc23}, 
\begin{equation}\label{ku2}
\begin{split}
I^{wa}_{k,k_1,k_2}(t)&\lesssim \|P_{k_1}U^{kg,\iota_1}_{\mathcal{L}_1}\|_{L^\infty}\|P_{k_2}U^{kg,\iota_2}_{\mathcal{L}_2}\|_{L^2}\\
&\lesssim \varep_1^22^{-4(k_1^++k_2^+)}2^{k_1/2}\langle t\rangle^{H(n_1+1)\delta+H(n_2)\delta-1}.
\end{split}
\end{equation}

Recalling \eqref{abc50.3}, the bounds \eqref{ku1} already suffice to prove \eqref{abc37} if $2^k\lesssim \langle t\rangle^{-1+145\delta}$. On the other hand, if $2^k\geq \langle t\rangle^{-1+145\delta}$ then the contribution of the pairs $(k_1,k_2)\in\mathcal{X}_k$ for which $2^{k_1^-}\lesssim 2^k\langle t\rangle^{290\delta}$ can be bounded using \eqref{ku2}. Also, the contribution of the pairs $(k_1,k_2)\in\mathcal{X}_k$ for which $2^{k_1^-}\lesssim 2^{-k}\langle t\rangle^{-1+145\delta}$ can be bounded using again \eqref{ku1}. After these reductions, for \eqref{abc37} it remains to prove that
\begin{equation}\label{ku3}
\sum_{(k_1,k_2)\in\mathcal{X}_k,\,2^{k_1^--20}\geq \max (2^{k},2^{-k}\langle t\rangle^{-1+145\delta})}I^{wa}_{k,k_1,k_2}(t)\lesssim\varep_1^22^{k/2}\langle t\rangle^{H''_{wa}(n)\delta-1}.
\end{equation}

We set $2^J=2^{k_1^--30}\langle t\rangle$ and decompose $P_{k_1}V^{kg,\iota_1}_{\mathcal{L}_1}=V_{\leq J,k_1;\mathcal{L}_1}^{kg,\iota_1}+V_{> J,k_1;\mathcal{L}_1}^{kg,\iota_1}$, as in \eqref{abc20.5}. We have
\begin{equation}\label{ku4}
\begin{split}
\|e^{-it\Lambda_{kg,\iota_1}}V^{kg,\iota_1}_{\leq J,k_1;\mathcal{L}_1}(t)\|_{L^\infty}&\lesssim \varep_12^{-k_1^-/2}\langle t\rangle^{-3/2+H(n_1+2)\delta+\delta}2^{-N(n_1+2)k_1^++6k_1^+},\\
\|V^{kg,\iota_1}_{>J,k_1;\mathcal{L}_1}(t)\|_{L^2}&\lesssim \varep_12^{-k_1^-}\langle t\rangle^{-1+H(n_1+1)\delta}2^{-N(n_1+1)k_1^+},
\end{split}
\end{equation}
see \eqref{abc23.1} and \eqref{abc22}. Therefore, using also \eqref{abc21}, for $(k_1,k_2)\in\mathcal{X}_k$ as in \eqref{ku3},
\begin{equation*}
\begin{split}
\|P_kI[e^{-it\Lambda_{kg,\iota_1}}V^{kg,\iota_1}_{\leq J,k_1;\mathcal{L}_1}(t),P_{k_2}U^{kg,\iota_2}_{\mathcal{L}_2}(t)]\|_{L^2}&\lesssim \|e^{-it\Lambda_{kg,\iota_1}}V^{kg,\iota_1}_{\leq J,k_1;\mathcal{L}_1}(t)\|_{L^\infty}\|P_{k_2}U^{kg,\iota_2}_{\mathcal{L}_2}(t)\|_{L^2}\\
&\lesssim\varep_1^22^{-k_1^-/2}\langle t\rangle^{-3/2+H(n_1+2)\delta+\delta+H(n_2)\delta}2^{-4k_1^+}
\end{split} 
\end{equation*}
and
\begin{equation*}
\begin{split}
\|P_kI[e^{-it\Lambda_{kg,\iota_1}}V^{kg,\iota_1}_{>J,k_1;\mathcal{L}_1}(t),P_{k_2}U^{kg,\iota_2}_{\mathcal{L}_2}(t)]\|_{L^2}&\lesssim 2^{3k/2}\|V^{kg,\iota_1}_{>J,k_1;\mathcal{L}_1}(t)\|_{L^2}\|P_{k_2}U^{kg,\iota_2}_{\mathcal{L}_2}(t)\|_{L^2}\\
&\lesssim\varep_1^22^{3k/2}2^{-k_1^-}\langle t\rangle^{-1+H(n_1+1)\delta+H(n_2)\delta}2^{-4k_1^+}.
\end{split} 
\end{equation*}
Therefore, using \eqref{abc50.3} and the identity $H(n_1+2)=H(n_1+1)+200$, we have
\begin{equation*}
I^{wa}_{k,k_1,k_2}(t)\lesssim \varep_1^22^{k/2}\langle t\rangle^{-1+H(n)\delta+10\delta}2^{-4k_1^+}[2^{-k/2}2^{-k_1^-/2}\langle t\rangle^{-1/2+205\delta}+2^{k-k_1^-}],
\end{equation*}
which suffices to prove the bounds \eqref{ku3}.
\medskip

{\bf{Case 3.}} Finally, assume that $(n_1,n_2,n)=(0,0,0)$. We estimate first, using symmetry, \eqref{abc21}, and \eqref{abc23}
\begin{equation}\label{cxz7}
\begin{split}
\sum_{(k_1,k_2)\in\mathcal{X}_k}I^{wa}_{k,k_1,k_2}(t)&\lesssim \sum_{(k_1,k_2)\in\mathcal{X}_k,\,k_2\leq k_1}\|P_{k_1}U^{kg}(t)\|_{L^2}\|P_{k_2}U^{kg}(t)\|_{L^\infty}\\
&\lesssim\sum_{(k_1,k_2)\in\mathcal{X}_k,\,k_2\leq k_1}\varep_1^2\langle t\rangle^{\delta}2^{-N(0)k_1^+}\cdot\langle t\rangle^{H(1)\delta-1}2^{k_2^-/2}\\
&\lesssim \varep_1^2\langle t\rangle^{(1+H(1))\delta-1}2^{-N(0)k^+}.
\end{split}
\end{equation}
This suffices to prove \eqref{abc37} when $2^k\gtrsim\langle t\rangle^{2\delta}$. Moreover, we also have, as in \eqref{ku1},
\begin{equation}\label{ku6}
\begin{split}
I^{wa}_{k,k_1,k_2}(t)&\lesssim 2^{3\min(k,k_1,k_2)/2}\|P_{k_1}U^{kg,\iota_1}\|_{L^2}\|P_{k_2}U^{kg,\iota_2}\|_{L^2}\\
&\lesssim \varep_1^22^{-N_0(k_1^++k_2^+)}2^{3\min(k,k_1,k_2)/2}\langle t\rangle^{2\delta}.
\end{split}
\end{equation}
This suffices to prove \eqref{abc37} when $2^k\lesssim\langle t\rangle^{-1+2\delta}$.

On the other hand, if $2^k\in[\langle t\rangle^{-1+2\delta},\langle t\rangle^{2\delta}]$ then we use first \eqref{abc22.1} and \eqref{abc23} to show that
\begin{equation*}
I^{wa}_{k,k_1,k_2}(t)\lesssim \varep_1^22^{-10(k_1^++k_2^+)}\langle t\rangle^{-1+2H(1)\delta}2^{\min(k_1^-,k_2^-)}2^{\max(k_1^-,k_2^-)/2}.
\end{equation*}
This suffices to control the contribution of the pairs $(k_1,k_2)$ for which $2^{\min(k_1^-,k_2^-)}\lesssim \langle t\rangle^{-1/2}$. For \eqref{abc37} it remains to prove that if $2^k\in[\langle t\rangle^{-1+2\delta},\langle t\rangle^{2\delta}]$ then
\begin{equation}\label{ku9}
\sum_{(k_1,k_2)\in\mathcal{X}_k,\,\langle t\rangle^{-1/2}2^{20}\leq 2^{k_2^-}\leq 2^{k_1^-}}I^{wa}_{k,k_1,k_2}(t)\lesssim\varep_1^22^{k^-/2}\langle t\rangle^{-1+5\delta}2^{-N(0)k^++4k^+}.
\end{equation}

We set $2^J=2^{k_1^--30}\langle t\rangle$ and decompose $P_{k_1}V^{kg,\iota_1}=V_{\leq J,k_1}^{kg,\iota_1}+V_{> J,k_1}^{kg,\iota_1}$, as in \eqref{abc20.5}.
We have
\begin{equation}\label{ku10}
\begin{split}
\|e^{-it\Lambda_{kg,\iota_1}}V^{kg,\iota_1}_{\leq J,k_1}(t)\|_{L^\infty}&\lesssim \varep_12^{-k_1^-/2+\kappa k_1^-/20}\langle t\rangle^{-3/2}2^{-N_0k_1^++(d'+6)k_1^+},\\
\|V^{kg,\iota_1}_{>J,k_1}(t)\|_{L^2}&\lesssim \varep_12^{-k_1^-}\langle t\rangle^{-1+H(1)\delta}2^{-N(1)k_1^+},
\end{split}
\end{equation}
see \eqref{vcx5} and \eqref{abc22}. Therefore, using also \eqref{vcx1.1}, for $(k_1,k_2)\in\mathcal{X}_k$ as in \eqref{ku9}
\begin{equation*}
\begin{split}
\|P_kI[e^{-it\Lambda_{kg,\iota_1}}V^{kg,\iota_1}_{\leq J,k_1}(t),P_{k_2}U^{kg,\iota_2}(t)]\|_{L^2}&\lesssim \|e^{-it\Lambda_{kg,\iota_1}}V^{kg,\iota_1}_{\leq J,k_1}(t)\|_{L^\infty}\|P_{k_2}U^{kg,\iota_2}(t)\|_{L^2}\\
&\lesssim\varep_1^22^{-k_1^-/2}2^{k_2^-}\langle t\rangle^{-3/2}2^{-4k_1^+}
\end{split} 
\end{equation*}
and
\begin{equation*}
\begin{split}
\|P_kI[e^{-it\Lambda_{kg,\iota_1}}V^{kg,\iota_1}_{>J,k_1}(t),P_{k_2}U^{kg,\iota_2}(t)]\|_{L^2}&\lesssim \|V^{kg,\iota_1}_{>J,k_1}(t)\|_{L^2}\|P_{k_2}U^{kg,\iota_2}(t)\|_{L^\infty}\\
&\lesssim\varep_1^22^{-k_1^-}2^{k_2^-/2}\langle t\rangle^{-2+2H(1)\delta}2^{-4k_1^+}.
\end{split} 
\end{equation*}
Therefore, if $(k_1,k_2)\in\mathcal{X}_k$ and $\langle t\rangle^{-1/2}2^{20}\leq 2^{k_2^-}\leq 2^{k_1^-}$ then  
\begin{equation}\label{ku15}
I^{wa}_{k,k_1,k_2}(t)\lesssim \varep_1^2\langle t\rangle^{-3/2}2^{k_2^-/4}2^{-4k_1^+},
\end{equation}
which suffices to prove the bounds \eqref{ku9}.

(ii) For \eqref{abc30.3} it suffices to prove that 
\begin{equation}\label{abc48.1}
\sum_{(k_1,k_2)\in\mathcal{X}_k}\|\varphi_k(\xi)(\partial_{\xi_l}\mathcal{F}\{I[P_{k_1}U^{kg,\iota_1}_{\mathcal{L}_1},P_{k_2}U^{kg,\iota_2}_{\mathcal{L}_2}]\})(\xi,t)\|_{L^2_\xi}\lesssim\varep_1^22^{k^-/2}\langle t\rangle^{H''_{wa}(n)\delta}2^{-N(n)k^++5k^+},
\end{equation}
for any $\iota_1,\iota_2\in\{+,-\}$, $l\in\{1,2,3\}$, $\mathcal{L}_1\in\mathcal{V}_{n_1}$, $\mathcal{L}_2\in\mathcal{V}_{n_2}$, $n_1+n_2\leq n$.

Without loss of generality we may assume that $n_1\leq n_2$. Recall that $U^{kg,\iota_1}_{\mathcal{L}_1}=e^{-it\Lambda_{kg,\iota_1}}V^{kg,\iota_1}_{\mathcal{L}_1}$ and $U^{kg,\iota_2}_{\mathcal{L}_2}=e^{-it\Lambda_{kg,\iota_2}}V^{kg,\iota_2}_{\mathcal{L}_2}$. We examine the formula \eqref{abc36}. The $\partial_{\xi_l}$ derivative can hit either the phase $e^{-it\Lambda_{kg,\iota_1}(\xi-\eta)}$, or the profile $\widehat{P_{k_1}V^{kg,\iota_1}_{\mathcal{L}_1}}(\xi-\eta)$, or the multiplier $m_1(\xi-\eta)$. In the first case, the $\partial_{\xi_l}$ derivative effectively corresponds to multiplying by a factor $\lesssim \langle t\rangle$, and changing $m_1$ in a way that still satisfies \eqref{abc9}. The corresponding bounds follow from \eqref{abc30}.

It remains to consider the case when the $\partial_{\xi_l}$ derivative hits the function $m_1(\xi-\eta)\varphi_{k_1}(\xi-\eta)\widehat{V^{kg,\iota_1}_{\mathcal{L}_1}}(\xi-\eta)$. It suffices to prove that 
\begin{equation}\label{abc48.2}
\sum_{(k_1,k_2)\in\mathcal{X}_k}\|P_kI^{wa}[U^{kg,\iota_1}_{\mathcal{L}_1,\ast l,k_1},P_{k_2}U^{kg,\iota_2}_{\mathcal{L}_2}](t)\|_{L^2}\lesssim\varep_1^22^{k^-/2}\langle t\rangle^{H''_{wa}(n)\delta}2^{-N(n)k^++5k^+},
\end{equation}
where $\widehat{U^{kg,\iota_1}_{\mathcal{L}_1,\ast l,k_1}}(\xi,t):=e^{-it\Lambda_{kg,\iota_1}(\xi)}\partial_{\xi_l}(\varphi_{k_1}\cdot m_1\cdot\widehat{V^{kg,\iota_1}_{\mathcal{L}_1}})(\xi,t)$. It follows from \eqref{abc21}--\eqref{abc21.5} that
\begin{equation}\label{abc48.3}
\begin{split}
&\|U^{kg,\iota_1}_{\mathcal{L}_1,\ast l,k_1}(t)\|_{L^2}\lesssim \varep_1\langle t\rangle^{H(n_1+1)\delta}2^{-N(n_1+1)k_1^+},\\
&\|P_{k_2}U^{kg,\iota_2}_{\mathcal{L}_2}(t)\|_{L^2}\lesssim \varep_1\langle t\rangle^{H(n_2)\delta}2^{-N(n_2)k_2^+}.
\end{split}
\end{equation}
Thus
\begin{equation}\label{ku12}
\begin{split}
\|P_kI^{wa}[U^{kg,\iota_1}_{\mathcal{L}_1,\ast l,k_1},&P_{k_2}U^{kg,\iota_2}_{\mathcal{L}_2}](t)\|_{L^2}\lesssim 2^{3\min(k_1,k_2,k)/2}\|U^{kg,\iota_1}_{\mathcal{L}_1,\ast l,k_1}(t)\|_{L^2}\|P_{k_2}U^{kg,\iota_2}_{\mathcal{L}_2}(t)\|_{L^2}\\
&\lesssim \varep_1^22^{3\min(k_1,k_2,k)/2}\langle t\rangle^{H(n_1+1)\delta+H(n_2)\delta}2^{-N(n_1+1)k_1^+}2^{-N(n_2)k_2^+}.
\end{split}
\end{equation}
This suffices to prove \eqref{abc48.2} if $(n_1,n_2)\neq (0,0)$, using \eqref{abc50.3} and the inequality $n_1+1\leq n$. 

On the other hand, if $(n_1,n_2)=(0,0)$ then we need to be more careful because of the loss of derivative and the slightly worse power of $\langle t\rangle$ in \eqref{ku12}. From the very beginning, in proving \eqref{abc48.1} we notice that we may assume that the sum is over pairs $(k_1,k_2)$ with $k_1\leq k_2$ (otherwise we make the change of variables $\eta\to\xi-\eta$ to move the $\xi$ derivative on the low frequency component). Using \eqref{abc23} and the $L^2$ bounds in the first line of \eqref{abc48.3}, we have
\begin{equation}\label{ku13}
\begin{split}
\|P_kI^{wa}[U^{kg,\iota_1}_{\ast l,k_1},&P_{k_2}U^{kg,\iota_2}](t)\|_{L^2}\lesssim \|U^{kg,\iota_1}_{\ast l,k_1}(t)\|_{L^2}\|P_{k_2}U^{kg,\iota_2}(t)\|_{L^\infty}\\
&\lesssim \varep_1^2\langle t\rangle^{2H(1)\delta-1}2^{-N(1)k_1^+}2^{-N(1)k_2^+}2^{2k_2^+}2^{k_2^-/2}.
\end{split}
\end{equation}
It is easy to see that the bounds \eqref{ku12}--\eqref{ku13} suffice to control the contribution of the pairs $(k_1,k_2)$ with $k_1\leq k_2$ in \eqref{abc48.2}. The desired bounds \eqref{abc48.1} thus follow when $(n_1,n_2)=(0,0)$. 
\end{proof}

\subsection{The nonlinearities $\mathcal{N}_\mathcal{L}^{kg}$} We prove now similar bounds for the nonlinearities $\mathcal{N}_\mathcal{L}^{kg}$. 

\begin{lemma}\label{dtv8}
Assume that $\mathcal{N}_{\mathcal{L}}^{kg}$ is as in \eqref{on6.1L}, $\mathcal{L}\in\mathcal{V}_{n}$, $n\in\{0,\ldots,N_1\}$, $t\in[0,T]$, and $k\in\mathbb{Z}$. 

(i) Then
\begin{equation}\label{abc31}
\|P_k\mathcal{N}_{\mathcal{L}}^{kg}(t)\|_{L^2}+\|P_k\partial_tV_{\mathcal{L}}^{kg}(t)\|_{L^2}\lesssim\varep_1^2\langle t\rangle^{H_{kg}''(n)\delta}\min(\langle t\rangle^{-1},2^{k^-})2^{-N(n+1)k^+-5k^+}.
\end{equation}
where $N(n+1):=N(n)-d$ if $n=N_1$ and
\begin{equation}\label{abc31.1}
H_{kg}''(0):=8,\qquad H_{kg}''(n):=H(n)+160\text{ for }n\in\{1,\ldots,N_1\}.
\end{equation}

(ii) Moreover, for $l\in\{1,2,3\}$ and $n\leq N_1-1$,
\begin{equation}\label{abc31.3}
\|P_k(x_l\mathcal{N}_{\mathcal{L}}^{kg})(t)\|_{L^2}\lesssim\varep_1^2\langle t\rangle^{H_{kg}''(n)\delta}2^{-N(n+1)k^+-5k^+}.
\end{equation}
\end{lemma}

\begin{proof} (i) With $I$ defined as in \eqref{abc9}--\eqref{abc36}, for \eqref{abc31} it suffices to prove that
\begin{equation}\label{abc60}
\begin{split}
&\sum_{(k_1,k_2)\in\mathcal{X}_k}I^{kg}_{k,k_1,k_2}(t)\lesssim\varep_1^2\langle t\rangle^{H_{kg}''(n)\delta}\min(\langle t\rangle^{-1},2^{k^-})2^{-N(n+1)k^+-5k^+},\\
&I^{kg}_{k,k_1,k_2}(t):=2^{k_1^+-k_2}\|P_kI[P_{k_1}U^{kg,\iota_1}_{\mathcal{L}_1},P_{k_2}U^{wa,\iota_2}_{\mathcal{L}_2}](t)\|_{L^2},
\end{split}
\end{equation}
for any $\iota_1,\iota_2\in\{+,-\}$, $\mathcal{L}_1\in\mathcal{V}_{n_1}$, $\mathcal{L}_2\in\mathcal{V}_{n_2}$, $n_1+n_2\leq n$. 

We estimate first, using just \eqref{abc21},
\begin{equation}\label{gb33}
\begin{split}
\sum_{(k_1,k_2)\in\mathcal{X}_k}I^{kg}_{k,k_1,k_2}(t)&\lesssim\sum_{(k_1,k_2)\in\mathcal{X}_k}2^{k_1^+-k_2}2^{3\min(k,k_1,k_2)/2}\|P_{k_1}U_{\mathcal{L}_1}^{kg}(t)\|_{L^2}\|P_{k_2}U_{\mathcal{L}_2}^{wa}(t)\|_{L^2}\\
&\lesssim\varep_1^2\langle t\rangle^{\delta(H(n_1)+H(n_2)+1)}2^{k^-}.
\end{split}
\end{equation}
Since $H(n_1)+H(n_2)+2\leq H_{kg}''(n)$, this suffices to prove the bounds \eqref{abc60} when $2^{k^-}\lesssim\langle t\rangle^{-1}$. 

On the other hand, if $2^{k^-}\geq \langle t\rangle^{-1}$ then we estimate, using also \eqref{abc23.2}--\eqref{abc23}, 
\begin{equation}\label{gb34}
\begin{split}
\sum_{(k_1,k_2)\in\mathcal{X}_k,\,k_1\geq k-20}I^{kg}_{k,k_1,k_2}(t)&\lesssim\sum_{(k_1,k_2)\in\mathcal{X}_k,\,k_1\geq k-20}2^{k_1^+-k_2}\|P_{k_1}U_{\mathcal{L}_1}^{kg}(t)\|_{L^2}\|P_{k_2}U_{\mathcal{L}_2}^{wa}(t)\|_{L^\infty}\\
&\lesssim\varep_1^2\langle t\rangle^{-1+\delta(H(n_1)+H(n_2+1)+2)}2^{-N(n_1)k^++2k^+}
\end{split}
\end{equation}
if $n_2\leq N_1-1$, and 
\begin{equation}\label{gb35}
\begin{split}
\sum_{(k_1,k_2)\in\mathcal{X}_k,\,k_2\geq k-20}I^{kg}_{k,k_1,k_2}(t)&\lesssim\sum_{(k_1,k_2)\in\mathcal{X}_k,\,k_2\geq k-20}2^{k_1^+-k_2}\|P_{k_1}U_{\mathcal{L}_1}^{kg}(t)\|_{L^\infty}\|P_{k_2}U_{\mathcal{L}_2}^{wa}(t)\|_{L^2}\\
&\lesssim\varep_1^2\langle t\rangle^{-1+\delta(H(n_1+1)+H(n_2)+2)}2^{-N(n_2)k^++2k^+}
\end{split}
\end{equation}
if $n_1\leq N_1-1$. The desired bounds \eqref{abc60} follow from \eqref{gb34}--\eqref{gb35} and \eqref{abc50.3}, unless $n_1=0$ or $n_2=0$. We consider separately these remaining cases.
\medskip

{\bf{Case 1.}} Assume first that $2^{k^-}\geq \langle t\rangle^{-1}$, $n=n_1\geq 1$, and $n_2=0$. The bound \eqref{gb34} still gives suitable control of the sum over $k_1\geq k-20$. Estimating as in \eqref{gb33} it is easy to see that
\begin{equation*}
\sum_{(k_1,k_2)\in\mathcal{X}_k,\,k_1\leq k-20,\,2^{k_1}\leq \langle t\rangle^{-2}}I^{kg}_{k,k_1,k_2}(t)\lesssim\varep_1^2\langle t\rangle^{-1}2^{-N(n_1)k^++2k^+}.
\end{equation*}
The contribution of the remaining pairs $(k_1,k_2)$ for which $\langle t\rangle^{-2}\leq 2^{k_1}\leq 2^{k-20}$ is also bounded as claimed since
\begin{equation}\label{gb35.1}
\begin{split}
I^{kg}_{k,k_1,k_2}(t)&\lesssim 2^{k_1^+-k_2}\|P_{k_1}U_{\mathcal{L}_1}^{kg}(t)\|_{L^2}\|P_{k_2}U^{wa}(t)\|_{L^\infty}\\
&\lesssim\varep_1^2\langle t\rangle^{-1+\delta(H(n_1)+H(1)+2)}2^{-N(1)k_2^++2k_2^+}2^{-k_1^+}.
\end{split}
\end{equation}

{\bf{Case 2.}} Assume now that $2^{k^-}\geq \langle t\rangle^{-1}$, $n_1=0$, and $n_2=n\geq 1$. The bound \eqref{gb35} still gives suitable control of the sum over $k_2\geq k-20$.  It remains to show that if $2^{k}\geq\langle t\rangle^{-1}$ then
\begin{equation}\label{abc60.1}
\sum_{(k_1,k_2)\in\mathcal{X}_k,\,k_2\leq k-10}2^{k_1^+-k_2}\|P_kI[P_{k_1}U^{kg,\iota_1},P_{k_2}U_{\mathcal{L}_2}^{wa,\iota_2}](t)\|_{L^2}\lesssim\varep_1^2\langle t\rangle^{H''(n_2)\delta-1}2^{-N(n_2)k^++5k^+}.
\end{equation}

To prove \eqref{abc60.1} we estimate first, when $k_2\leq k-10$ and $|k_1-k|\leq 4$,
\begin{equation}\label{gb35.2}
\begin{split}
\|P_kI[P_{k_1}U^{kg,\iota_1},P_{k_2}U_{\mathcal{L}_2}^{wa,\iota_2}](t)\|_{L^2}&\lesssim \|P_{k_1}U^{kg}(t)\|_{L^2}2^{3k_2/2}\|P_{k_2}U_{\mathcal{L}_2}^{wa}(t)\|_{L^2}\\
&\lesssim \varep_1^2\langle t\rangle^{H(1)\delta+H(n_2)\delta+2\delta}2^{k_1^-}2^{2k_2^-}2^{-N(0)k^+},
\end{split}
\end{equation}
using \eqref{abc21} and \eqref{abc22.1}. Therefore, since $N(0)-N(1)=40$,
\begin{equation*}
\begin{split}
\sum_{(k_1,k_2)\in\mathcal{X}_k,\,k_2\leq k-10,\,2^{k_1^-+k_2^-}\lesssim\langle t\rangle^{-1+145\delta}2^{40k^+}}2^{k_1^+-k_2}&\|P_kI[P_{k_1}U^{kg,\iota_1},P_{k_2}U_{\mathcal{L}_2}^{wa,\iota_2}](t)\|_{L^2}\\
&\lesssim\varep_1^2\langle t\rangle^{H''(n_2)\delta-1}2^{-N(1)k^++4k^+}.
\end{split}
\end{equation*}
For \eqref{abc60.1} it remains to prove that
\begin{equation}\label{abc83}
2^{-k_2}\|P_kI[P_{k_1}U^{kg,\iota_1},P_{k_2}U_{\mathcal{L}_2}^{wa,\iota_2}](t)\|_{L^2}\lesssim\varep_1^2\langle t\rangle^{H''(n_2)\delta-\delta-1}2^{-N(n_2)k^++4k^+}
\end{equation}
for any $k_1,k_2,k\in\mathbb{Z}$ such that 
\begin{equation}\label{abc84}
k_2\leq k-10,\qquad |k_1-k|\leq 4,\qquad 2^{k_1^-+k_2^-}\geq\langle t\rangle^{-1+145\delta}2^{40k^++80}.
\end{equation}

We set $2^J=2^{k_1^--30}\langle t\rangle$, decompose $P_{k_1}V^{kg,\iota_1}=V_{\leq J,k_1}^{kg,\iota_1}+V_{> J,k_1}^{kg,\iota_1}$, and recall the bounds \eqref{ku10}. Therefore, using also the $L^2$ bounds \eqref{abc21} 
\begin{equation*}
\begin{split}
2^{-k_2}\|P_kI[e^{-it\Lambda_{kg,\iota_1}}V_{\leq J,k_1}^{kg,\iota_1},&P_{k_2}U_{\mathcal{L}_2}^{wa,\iota_2})](t)\|_{L^2}\lesssim 2^{-k_2}\|e^{-it\Lambda_{kg,\iota_1}}V_{\leq J,k_1}^{kg,\iota_1}\|_{L^\infty}\|P_{k_2}U_{\mathcal{L}_2}^{wa,\iota_2}(t)\|_{L^2}\\
&\lesssim \varep_1^2\langle t\rangle^{-3/2+H(n_2)\delta}2^{-k_1^-/2}2^{-k_2^-/2}2^{-N(1)k^++16k^+},
\end{split} 
\end{equation*}
and
\begin{equation*}
\begin{split}
2^{-k_2}\|P_kI[e^{-it\Lambda_{kg,\iota_1}}V_{>J,k_1}^{kg,\iota_1},&P_{k_2}U_{\mathcal{L}_2}^{wa,\iota_2})](t)\|_{L^2}\lesssim 2^{-k_2}\|e^{-it\Lambda_{kg,\iota_1}}V_{> J,k_1}^{kg,\iota_1}\|_{L^2}\|P_{k_2}U_{\mathcal{L}_2}^{wa,\iota_2}(t)\|_{L^\infty}\\
&\lesssim \varep_1^2\langle t\rangle^{-1+10\delta+H(n_2)\delta}2^{-k_1^-}2^{k_2^-}2^{-N(1)k^+}.
\end{split} 
\end{equation*}
Since $2^{-k_1^-/2}2^{-k_2^-/2}\lesssim \langle t\rangle^{1/2-75\delta}2^{-20k^+}$ (see \eqref{abc84}), these bounds suffice to prove \eqref{abc83}.
\medskip

{\bf{Case 3.}} Finally, assume that $2^{k^-}\geq \langle t\rangle^{-1}$ and $n_1=n_2=n=0$. The bounds \eqref{gb34}--\eqref{gb35} are sufficient if $2^{k^+}\gtrsim \langle t\rangle^{2\delta}$. The contribution of the pairs $(k_1,k_2)$  in \eqref{abc60} for which $2^{\min(k_1,k_2)}2^{-30(k_1^++k_2^+)}\lesssim \langle t\rangle^{-1+5.9\delta}$ can be bounded as before using just $L^2$ estimates and recalling that $N(1)=N(0)-40$. For \eqref{abc60} it remains to show that 
\begin{equation}\label{abc60.2}
2^{k_1^+-k_2}\|P_kI[P_{k_1}U^{kg,\iota_1},P_{k_2}U^{wa,\iota_2}](t)\|_{L^2}\lesssim\varep_1^2\langle t\rangle^{7.9\delta-1}2^{-N(1)k^+-5k^+},
\end{equation}
provided that
\begin{equation}\label{ku30}
2^k\in[\langle t\rangle^{-1},\langle t\rangle^{2\delta}],\qquad 2^{\min(k_1,k_2)}2^{-30(k_1^++k_2^+)}\geq \langle t\rangle^{-1+5.9\delta}2^{40}.
\end{equation} 

We decompose 
\begin{equation}\label{on11.30}
P_{k_1}U^{kg,\iota_1}(t)=\sum_{j_1\geq -k_1^-}e^{-it\Lambda_{kg,\iota_1}}V_{j_1,k_1}^{kg,\iota_1}(t),\qquad P_{k_2}U^{wa,\iota_2}(t)=\sum_{j_2\geq -k_2^-}e^{-it\Lambda_{wa,\iota_2}}V_{j_2,k_2}^{wa,\iota_2}(t)
\end{equation}
as in \eqref{abc20.5}. Notice that
\begin{equation}\label{ku31}
\begin{split}
&\mathcal{F}\big\{P_kI[e^{-it\Lambda_{kg,\iota_1}}V_{j_1,k_1}^{kg,\iota_1},e^{-it\Lambda_{wa,\iota_2}}V_{j_2,k_2}^{wa,\iota_2}]\big\}(\xi,t)\\
&=C\varphi_k(\xi)\int_{\mathbb{R}^3}e^{-it\Lambda_{kg,\iota_1}(\xi-\eta)-it\Lambda_{wa,\iota_2}(\eta)}m_1(\xi-\eta)\widehat{V_{j_1,k_1}^{kg,\iota_1}}(\xi-\eta,t)\cdot m_2(\eta)\widehat{V_{j_2,k_2}^{wa,\iota_2}}(\eta,t)\,d\eta.
\end{split}
\end{equation} 
The main observation is that the absolute value of the $\eta$ gradient of the phase function $\eta\to t[\Lambda_{kg,\iota_1}(\xi-\eta)+\Lambda_{wa,\iota_2}(\eta)]$ is bounded from below by $c|t|2^{-2k_1^+}$ in the support of the integral. Recalling also \eqref{ku30}, we can thus use Lemma \ref{tech5} to show that the contribution of the pairs $(j_1,j_2)$ for which $2^{\max(j_1,j_2)}\leq\langle t\rangle^{1-\delta/2}2^{-2k_1^+}$ is negligible, i.e. 
\begin{equation}\label{ku40}
2^{k_1^+-k_2}\|P_kI^{kg}[e^{-it\Lambda_{kg,\iota_1}}V_{j_1,k_1}^{kg,\iota_1}(t),e^{-it\Lambda_{wa,\iota_2}}V_{j_2,k_2}^{wa,\iota_2}(t)]\|_{L^2}\lesssim\varep_1^2\langle t\rangle^{-2}.
\end{equation}

To deal with the remaining pairs $(j_1,j_2)$ we fix $J_1$ such that $2^{J_1}=\langle t\rangle^{200\delta}$ and estimate
\begin{equation}\label{ku41}
\begin{split}
2^{k_1^+-k_2}&\|P_kI[e^{-it\Lambda_{kg,\iota_1}}V_{> J_1,k_1}^{kg,\iota_1}(t),P_{k_2}U^{wa,\iota_2}(t)]\|_{L^2}\lesssim 2^{k_1^+-k_2}\|V_{>J_1,k_1}^{kg,\iota_1}(t)\|_{L^2}\|P_{k_2}U^{wa,\iota_2}(t)\|_{L^\infty}\\
&\lesssim \varep_1^22^{-J_1}\langle t\rangle^{-1+2H(1)\delta+\delta}2^{-N(1)k_1^++4k_1^+}2^{-N(1)k_2^++4k_2^+},
\end{split}
\end{equation}
using \eqref{abc22} and \eqref{abc23.2}. This is consistent with the desired bounds \eqref{abc60.2}, due to the choice of $J_1$. For the remaining contribution (which is nontrivial only when $2^{k_1}\gtrsim\langle t\rangle^{-200\delta}$) we fix $J_2$ such that $2^{J_2}=\langle t\rangle^{1-\delta/2}2^{-2k_1^+}$ and prove two bounds. We have
\begin{equation}\label{ku42}
\begin{split}
2^{k_1^+-k_2}&\|P_kI[e^{-it\Lambda_{kg,\iota_1}}V_{\leq J_1,k_1}^{kg,\iota_1}(t),e^{-it\Lambda_{wa,\iota_2}}V_{> J_2,k_2}^{wa,\iota_2}(t)]\|_{L^2}\\
&\lesssim 2^{k_1^+-k_2}2^{3k_2/2}\|V_{\leq J_1,k_1}^{kg,\iota_1}(t)\|_{L^2}\|V_{> J_2,k_2}^{wa,\iota_2}(t)\|_{L^2}\\
&\lesssim \varep_1^2\langle t\rangle^{H(1)\delta+\delta}2^{k_1^-+k_2^-}2^{-N(0)k_1^++2k_1^+}2^{-N(0)k_2^++2k_2^+},
\end{split}
\end{equation}
using just the $L^2$ bounds \eqref{abc21} and \eqref{abc22.1}. The bounds \eqref{ku40}--\eqref{ku42} suffice to prove \eqref{abc60.2} if \eqref{ku30} holds and, in addition, $2^{k_1^-+k_2^-}\lesssim \langle t\rangle^{-1-3.5\delta}$. On the other hand, we also have
\begin{equation}\label{ku43}
\begin{split}
&2^{k_1^+-k_2}\|P_kI[e^{-it\Lambda_{kg,\iota_1}}V_{\leq J_1,k_1}^{kg,\iota_1}(t),e^{-it\Lambda_{wa,\iota_2}}V_{> J_2,k_2}^{wa,\iota_2}(t)]\|_{L^2}\\
&\lesssim 2^{k_1^+-k_2}\|e^{-it\Lambda_{kg,\iota_1}}V_{\leq J_1,k_1}^{kg,\iota_1}(t)\|_{L^\infty}\|V_{> J_2,k_2}^{wa,\iota_2}(t)\|_{L^2}\\
&\lesssim \varep_1^22^{k_1^+-k_2}\langle t\rangle^{-3/2}2^{-k_1^-/2}2^{-N_0k_1^++(d'+6)k_1^+}\cdot 2^{-k_2/2}\langle t\rangle^{-1+\delta/2}2^{2k_1^+}\langle t\rangle^{H(1)\delta}2^{-N(1)k_2^+},
\end{split}
\end{equation}
using \eqref{abc22} and \eqref{vcx5}. If $2^{k_1^-+k_2^-}\gtrsim \langle t\rangle^{-1-3.5\delta}$ then the right-hand side of \eqref{ku43} is bounded by
\begin{equation*}
C\varep_1^2\langle t\rangle^{-2+12.5\delta}2^{-k_2^-}2^{-N(1)k_1^++d'k_1^+}2^{-N(1)k_2^+}.
\end{equation*}
Since $2^{-k_2^-}\lesssim \langle t\rangle^{1-5.9\delta}2^{-28(k_1^++k_2^+)}$ (see \eqref{ku30}) this is consistent with the desired bound \eqref{abc60.2}. The conclusion follows in the remaining case $2^{k_1^-+k_2^-}\gtrsim \langle t\rangle^{-1-3.5\delta}$.

(ii) With the same notation as before, it suffices to prove that
\begin{equation}\label{abc99.1}
\begin{split}
\sum_{(k_1,k_2)\in\mathcal{X}_k}2^{k_1^+-k_2}\|\varphi_k(\xi)(\partial_{\xi_l}\mathcal{F}\{I[P_{k_1}&U^{kg,\iota_1}_{\mathcal{L}_1},P_{k_2}U^{wa,\iota_2}_{\mathcal{L}_2}]\})(\xi,t)\|_{L^2_\xi}\\
&\lesssim\varep_1^2\langle t\rangle^{H_{kg}''(n)\delta}2^{-N(n+1)k^+-5k^+},
\end{split}
\end{equation}
for any $\iota_1,\iota_2\in\{+,-\}$, $l\in\{1,2,3\}$, $\mathcal{L}_1\in\mathcal{V}_{n_1}$, $\mathcal{L}_2\in\mathcal{V}_{n_2}$, $n_1+n_2\leq n$. 

We write $U^{kg,\iota_1}_{\mathcal{L}_1}=e^{-it\Lambda_{kg,\iota_1}}V^{kg,\iota_1}_{\mathcal{L}_1}$ and notice that the $\partial_{\xi_l}$ derivative can hit either the phase $e^{-it\Lambda_{kg,\iota_1}(\xi-\eta)}$, or the multiplier $m_1(\xi-\eta)$, or the profile $\widehat{P_{k_1}V^{kg,\iota_1}_{\mathcal{L}_1}}(\xi-\eta)$. In the first case the derivative effectively corresponds to multiplying by factors $\lesssim \langle t\rangle$ or $\lesssim 2^{-k_1^-}$, and changing the multiplier $m_1$, in a way that still satisfies \eqref{abc9}. The desired estimates follow from \eqref{abc60}.

It remains to consider the case when the $\partial_{\xi_l}$ derivative hits the function $m_1(\xi-\eta)\widehat{P_{k_1}V^{kg,\iota_1}_{\mathcal{L}_1}}(\xi-\eta)$. It suffices to prove that 
\begin{equation}\label{abc99.2}
\sum_{(k_1,k_2)\in\mathcal{X}_k}2^{k_1^+-k_2}\|P_kI[U^{kg,\iota_1}_{\mathcal{L}_1,\ast l,k_1},P_{k_2}U^{wa,\iota_2}_{\mathcal{L}_2}](t)\|_{L^2}\lesssim\varep_1^2\langle t\rangle^{H_{kg}''(n)\delta}2^{-N(n+1)k^+-5k^+},
\end{equation}
where, as in the proof of Lemma \ref{dtv7}, $\widehat{U^{kg,\iota_1}_{\mathcal{L}_1,\ast l,k_1}}(\xi,t)=e^{-it\Lambda_{kg,\iota_1}(\xi)}\partial_{\xi_l}\{\varphi_{k_1}\cdot m_1\cdot\widehat{V^{kg,\iota_1}_{\mathcal{L}_1}}\}(\xi,t)$. The estimates \eqref{abc99.2} follow from \eqref{abc48.3} and \eqref{abc21} if $n_2\geq 1$, using an estimate similar to \eqref{gb33}.

Assume now that 
\begin{equation}\label{abc99.4}
n_1=n\qquad\text{ and }\qquad n_2=0.
\end{equation}
We need to be slightly more careful than before. If $2^{k}\lesssim \langle t\rangle^{1/100}$ then we can just use the $L^\infty$ bounds \eqref{abc23.2} and the $L^2$ bounds \eqref{abc48.3} to prove \eqref{abc99.2}. On the other hand, if $2^{k-10}\geq \langle t\rangle^{1/100}$ then the contribution of the pairs $(k_1,k_2)$ with $k_2\geq k-10$ or $k_2\leq -6k$ can be estimated as before, using just $L^2$ bounds. 

To estimate the contribution of the remaining pairs, we need to avoid derivative loss. We go back to \eqref{abc99.1}. It remains to prove that if $2^{k-10}\geq \langle t\rangle^{1/100}$ then
\begin{equation}\label{abc99.6}
\begin{split}
\sum_{(k_1,k_2)\in\mathcal{X}_k,\,k_2\in[-6k,k-10]}2^{k_1^+-k_2}\|\varphi_k(\xi)(\partial_{\xi_l}\mathcal{F}\{I[P_{k_1}&U^{kg,\iota_1}_{\mathcal{L}_1},P_{k_2}U^{wa,\iota_2}]\})(\xi,t)\|_{L^2_\xi}\\
&\lesssim\varep_1^2\langle t\rangle^{H_{kg}''(n)\delta}2^{-N(n+1)k^+-5k^+}.
\end{split}
\end{equation}
We make the change of variables $\eta\to\xi-\eta$ in the integral \eqref{abc36}, to move the $\partial_{\xi_l}$ derivative to the low frequency factor. Using \eqref{abc21} and $L^2\times L^\infty$ estimates as before, for \eqref{abc99.6} it suffices to prove that
\begin{equation}\label{abc99.7}
\big\|\mathcal{F}^{-1}\{\partial_{\xi_l}[m_2(\xi)\widehat{P_{k_2}U^{wa}}(\xi,t)]\}\big\|_{L^\infty}\lesssim \varep_1\langle t\rangle^{2H(1)\delta}2^{-2k_2^+}2^{k_2^--2\delta k_2^-}.
\end{equation}
We remark that the loss of a factor of $\langle t\rangle ^{C\delta}$ is mitigated by the gain of derivative and the assumption $2^{k}\gtrsim\langle t\rangle^{1/100}$. 

It remains to prove \eqref{abc99.7}. The derivative $\partial_{\xi_l}$ can hit the symbol $m_2$, and this contribution is bounded easily using \eqref{abc23.2}. On the other hand, if $\partial_{\xi_l}$ hits the function $\widehat{P_{k_2}U^{wa}}(\xi,t)$ then we replace $U^{wa}(t)$ with $e^{-it|\nabla|}V^{wa}(t)$. Notice that
\begin{equation*}
\big\|\mathcal{F}^{-1}\{\partial_{\xi_l}[e^{-it|\xi|}\varphi_{k_2}(\xi)\widehat{V^{wa}}(\xi)]\}\big\|_{L^\infty}\lesssim \varep_1\langle t\rangle^{H(1)\delta+2\delta}2^{-4k_2^+}2^{k_2^--\delta k_2^-},
\end{equation*}
as a consequence of \eqref{abc21.5} and \eqref{abc23.2}. The desired bound \eqref{abc99.7} follows in this case as well. This completes the proof of the lemma. 
\end{proof}

\subsection{The bounds \eqref{bootstrap3.1}--\eqref{bootstrap3.4} at time $t=0$} We use now the initial-data assumptions \eqref{ml1} and elliptic estimates to take the first step towards proving Proposition \ref{bootstrap}. 

\begin{proposition}\label{IniDat}
The bounds \eqref{bootstrap3.1}--\eqref{bootstrap3.4} hold at time $t=0$.
\end{proposition}

\begin{proof}  {\bf{Step 1.}} We prove first the $Z$ norm bounds \eqref{bootstrap3.4}. Notice that $U^{wa}_{\mathcal{L}}(0) = V^{wa}_{\mathcal{L}}(0)$ and $U^{kg}_{\mathcal{L}}(0) = V^{kg}_{\mathcal{L}}(0)$. It follows from \eqref{ml1.11} and \eqref{hyt4} that
\begin{equation}\label{IniDat2}
\Big[\sum_{(k,j)\in\mathcal{J}}2^{2N(1)k^+}2^k2^{2j}\|Q_{jk}U^{wa}(0)\|_{L^2}^2\Big]^{1/2}\lesssim\varep_0.
\end{equation}
Moreover $2^{2j}2^{3k/2}\|Q_{jk}U^{wa}(0)\|_{L^2}\lesssim \varep_0$ for $j\geq 10|k|+10$ as a consequence of \eqref{ml1}. Using also \eqref{IniDat2} it follows that, for any $k\in\mathbb{Z}$,
\begin{equation}\label{IniDat5}
\sum_{j\geq\max(-k,0)}2^{j}\|Q_{jk}U^{wa}(0)\|_{L^2}\lesssim 2^{-N(1)k^+}2^{-k/2}(1+|k|). 
\end{equation}
This gives the desired $Z$ norm control for the wave component. 

Similarly, using the assumptions \eqref{ml1} we have
\begin{equation}\label{ku50}
\|U^{kg}_0\|_{H^{N(0)}}+\sum_{|\beta|\leq 3}\|x^{\beta}U^{kg}_0\|_{H^{N(3)}}\lesssim\varep_0.
\end{equation}
In particular $\|P_kU^{kg}_0\|_{L^1}\lesssim\varep_0$, which gives $\|\widehat{P_kU^{kg}_0}\|_{L^\infty}\lesssim\varep_0$ for any $k\in\mathbb{Z}$. This suffices for $k\leq 0$. On the other hand, if $k\geq 0$ then it follows from \eqref{ku50} that
\begin{equation*}
\|P_kU^{kg}_0\|_{L^2}\lesssim\varep_02^{-N(0) k^+},\qquad \|\,|x|^3P_kU^{kg}_0\|_{L^2}\lesssim \varep_02^{-N(3) k^+}.
\end{equation*}
Thus $\|P_kU^{kg}_0\|_{L^1}\lesssim\varep_02^{-(N(0)+N(3))/2 k^+}$, which gives the desired control.

{\bf{Step 2.}} We consider now the high order Sobolev bounds in \eqref{bootstrap3.1}. It suffices to show that
\begin{equation}\label{IniDat6}
\begin{split}
&\sum_{|\beta'|\leq |\beta|+\beta_0-1\leq n}\||\nabla|^{-1/2}(x^{\beta'}\partial^{\beta}\partial_0^{\beta_0}u)(0)\|_{H^{N(n)}}+\sum_{|\beta'|,|\beta|+\beta_0-1\leq n}\|(x^{\beta'}\partial^{\beta}\partial_0^{\beta_0}v)(0)\|_{H^{N(n)}}\lesssim\varep_0,
\end{split}
\end{equation}
for any $n\in[0,N_1]$, where $x^{\beta'}=x_1^{\beta'_1}x_2^{\beta'_2}x_3^{\beta'_3}$ and $\partial^{\beta}=\partial_1^{\beta_1}\partial_2^{\beta_2}\partial_3^{\beta_3}$.

We prove \eqref{IniDat6} by induction over $\beta_0$. Notice that the bounds follow directly from \eqref{ml1.11} if $\beta_0=0$ or if $\beta_0=1$, by passing to the Fourier space. If $\beta_0\geq 2$ then we use the identities $\partial_t^2u=\Delta u+\mathcal{N}^{wa}$ and $\partial_t^2v=\Delta v-v+\mathcal{N}^{kg}$. The contribution of the linear components is bounded as claimed, due to the induction hypothesis. For \eqref{IniDat6} it remains to prove that, for $n\in[0,N_1]$,
\begin{equation}\label{IniDat7}
\begin{split}
&\sum_{|\beta'|\leq |\beta|+\gamma\leq n,\,1\leq\gamma\leq\beta_0-1}\||\nabla|^{-1/2}(x^{\beta'}\partial^{\beta}\partial_0^{\gamma-1}\mathcal{N}^{wa})(0)\|_{H^{N(n)}}\lesssim \varep_0,\\
&\sum_{|\beta'|,|\beta|+\gamma\leq n,\,1\leq\gamma\leq\beta_0-1}\|(x^{\beta'}\partial^{\beta}\partial_0^{\gamma-1}\mathcal{N}^{kg})(0)\|_{H^{N(n)}}\lesssim\varep_0.
\end{split}
\end{equation} 

These bounds follow easily using the induction hypothesis \eqref{IniDat6} and the explicit formulas. For the first inequality, recall that $\mathcal{N}^{wa}=A^{\al\be}\partial_\al v\partial_\be v+Dv^2$. Therefore $x^{\beta'}\partial^{\beta}\partial_0^{\gamma-1}\mathcal{N}^{wa}$ can be written as a sum of terms of the form
\begin{equation*}
x^{\beta'} (\partial^{\beta_1}\partial_0^{\gamma_1}v)(\partial^{\beta_2}\partial_0^{\gamma_2}v),
\end{equation*}
where $|\beta'|\leq n$, $|\beta_1|+\gamma_1+|\beta_2|+\gamma_2\leq n+1$, and $\max(\gamma_1,\gamma_2)\leq \beta_0-1$. Such products can be easily bounded in $H^{N(n)}$, using the induction hypothesis and Littlewood--Paley decompositions as before, and placing the high frequency component in $L^2$ and the low frequency component in $L^\infty$. The contribution at low frequencies can be estimated using just $L^2$ bounds on both components. 

The proof of the inequality in the second line of \eqref{IniDat7} is similar. We use the formula $\mathcal{N}^{kg}=uB^{\alpha\beta}\partial_\al\partial_\be v+Euv$. Since $B^{00}=0$, one can distribute the $\partial^{\beta}\partial_0^{\gamma-1}$ derivatives and still get only terms with no more than $\beta_0-1$ time derivatives, of the form
\begin{equation*}
x^{\beta'} (\partial^{\beta_1}\partial_0^{\gamma_1}u)(\partial^{\beta_2}\partial_0^{\gamma_2}v),
\end{equation*}
where $|\beta'|\leq n$, $|\beta_1|+\gamma_1+|\beta_2|+\gamma_2\leq n+1$, $|\beta_1|+\gamma_1\leq n-1$, and $\max(\gamma_1,\gamma_2)\leq \beta_0-1$. Such products can be bounded in $H^{N(n)}$, using the induction hypothesis and Littlewood--Paley decompositions. This completes the proof of  \eqref{IniDat}.

{\bf{Step 3.}} We consider now the Sobolev bounds \eqref{bootstrap3.2} on the profiles $V^{wa}_\mathcal{L}$ and $V^{kg}_\mathcal{L}$. Since $V^{wa}_\mathcal{L}(0)=U^{wa}_\mathcal{L}(0)$ and $V^{kg}_\mathcal{L}(0)=U^{kg}_\mathcal{L}(0)$ it suffices to prove that
\begin{equation*}
\begin{split}
2^{k/2}\|\varphi_k(\xi)\partial_{\xi_l}&\mathcal{F}\{(\partial_t-i\Lambda_{wa})\mathcal{L}u\}(\xi,0)\|_{L^2}\\
&+2^{k^+}\|\varphi_k(\xi)\partial_{\xi_l}\mathcal{F}\{(\partial_t-i\Lambda_{kg})\mathcal{L}v\}(\xi,0)\|_{L^2}\lesssim \varep_02^{-N(n+1)k^+},
\end{split}
\end{equation*}
for $l\in\{1,2,3\}$, $k\in\mathbb{Z}$, $\mathcal{L}\in\mathcal{V}_n$, $n\in[0,N_1-1]$. These bounds follow again from \eqref{IniDat6}, after passing to the Fourier space. 
\end{proof}

\section{Energy estimates}\label{enerEst}

In this section we prove the energy bounds in \eqref{bootstrap3.1}. The vector-fields $\Gamma_j$ and $\Omega_{jk}$ commute with the wave operator, so, as a consequence of \eqref{on1}, for any $\mathcal{L}\in\mathcal{V}_{N_1}$,
\begin{equation}\label{abc2}
\begin{split}
-\square (\mathcal{L}u)&=\mathcal{L}(A^{\alpha\beta}\partial_\alpha v\partial_\beta v+Dv^2),\\
(-\square +1)(\mathcal{L}v)&=\mathcal{L}(uB^{\alpha\beta}\partial_\alpha\partial_\beta v+Euv).
\end{split}
\end{equation}

\subsection{The bound on $U^{wa}_{\mathcal{L}}$} We start by estimating the wave component.

\begin{proposition}\label{BootstrapEE4}
With the notation and hypothesis in Proposition \ref{bootstrap}, for any $t\in[0,T]$, $n\in[0,N_1]$ and $\mathcal{L}\in\mathcal{V}_n$
\begin{equation}\label{abc1}
\|\,|\nabla|^{-1/2}U^{wa}_{\mathcal{L}}(t)\|_{H^{N(n)}}\lesssim\varep_{0}\langle t\rangle^{H(n)\delta}.
\end{equation}
\end{proposition}

\begin{proof} 
With $P:=|\nabla|^{-1/2}\langle\nabla\rangle^{N(n)}\mathcal{L}$ we define the energy functional
\begin{equation}\label{abc3}
\mathcal{E}_{wa}^{\mathcal{L}}(t)=\int_{\mathbb{R}^3}\big[(\partial_0Pu(t))^2+\sum_{j=1}^3(\partial_jPu(t))^2\big]\,dx.
\end{equation}
Using the first equation in \eqref{abc2} we calculate
\begin{equation}\label{abc3.1}
\begin{split}
\frac{d}{dt}&\mathcal{E}_{wa}^{\mathcal{L}}(t)=2\int_{\mathbb{R}^3}P[A^{\al\be}\partial_\al v\partial_\be v+Dv^2](t)\cdot \partial_0Pu(t)\,dx\\
&=2\int_{\mathbb{R}^3}|\nabla|^{-1/2}\langle\nabla\rangle^{N(n)}\mathcal{L}[A^{\al\be}\partial_\al v\partial_\be v+Dv^2](t)\cdot |\nabla|^{-1/2}\langle\nabla\rangle^{N(n)}\partial_0\mathcal{L}u(t)\,dx\\
&=C\int_{\mathbb{R}^3}|\xi|^{-1}(1+|\xi|^2)^{N(n)}\mathcal{F}\{\mathcal{L}[A^{\al\be}\partial_\al v\partial_\be v+Dv^2]\}(\xi,t)\cdot\overline{\widehat{\partial_0\mathcal{L}u}(\xi,t)}\,d\xi.
\end{split}
\end{equation}

We decompose time integrals into dyadic pieces. More precisely, given $t\in[0,T]$, we fix a suitable decomposition of the function $\mathbf{1}_{[0,t]}$, i.e. we fix functions $q_0,\ldots,q_{L+1}:\mathbb{R}\to[0,1]$, $|L-\log_2(2+t)|\leq 2$, with the properties
\begin{equation}\label{nh2}
\begin{split}
&\mathrm{supp}\,q_0\subseteq [0,2], \quad \mathrm{supp}\,q_{L+1}\subseteq [t-2,t],\quad\mathrm{supp}\,q_m\subseteq [2^{m-1},2^{m+1}]\text{ for } m\in\{1,\ldots,L\},\\
&\sum_{m=0}^{L+1}q_m(s)=\mathbf{1}_{[0,t]}(s),\qquad q_m\in C^1(\mathbb{R})\text{ and }\int_0^t|q'_m(s)|\,ds\lesssim 1\text{ for }m\in \{1,\ldots,L\}.
\end{split}
\end{equation}

Let $I_m$ denote the support of $q_m$. In view of \eqref{abc3.1} and \eqref{abc4}, it suffices to prove that
\begin{equation*}
\Big|\int_{I_m}\int_{\mathbb{R}^3}q_m(s)|\xi|^{-1}(1+|\xi|^2)^{N(n)}\mathcal{F}\{\partial \mathcal{L}_1v\cdot\partial'\mathcal{L}_2v\}(\xi,s)\cdot\overline{\widehat{\partial_0\mathcal{L}u}(\xi,s)}\,d\xi ds\Big|\lesssim \varep_1^32^{2H(n)\delta m},
\end{equation*}
for any $t\in[0,T]$, $m\in\{0,\ldots,L+1\}$, $\partial,\partial'\in\{I,\partial_0,\partial_1,\partial_2,\partial_3\}$, $\mathcal{L}_1\in\mathcal{V}_{n_1}$, $\mathcal{L}_2\in\mathcal{V}_{n_2}$, $n_1+n_2\leq n$. We rewrite the functions $\mathcal{L}_1v$, $\mathcal{L}_2v$, $\mathcal{L}u$ in terms of the variables $U^{kg,\pm}_{\mathcal{L}_1}$, $U_{\mathcal{L}_2}^{kg,\pm}$, $U_{\mathcal{L}}^{wa,\pm}$, as in \eqref{on5}. It suffices to prove that, for any $\iota,\iota_1,\iota_2\in\{+,-\}$,
\begin{equation}\label{abc6.1}
\begin{split}
\Big|\int_{I_m}\int_{\mathbb{R}^3\times\mathbb{R}^3}&q_m(s)|\xi|^{-1}(1+|\xi|^2)^{N(n)}m_1(\xi-\eta)m_2(\eta)\\
&\times\widehat{U^{kg,\iota_1}_{\mathcal{L}_1}}(\xi-\eta,s)\widehat{U^{kg,\iota_2}_{\mathcal{L}_2}}(\eta,s)\overline{\widehat{U^{wa,\iota}_{\mathcal{L}}}(\xi,s)}\,d\xi d\eta ds\Big|\lesssim \varep_1^32^{2H(n)\delta m},
\end{split}
\end{equation}
where $m_1$ and $m_2$ are symbols satisfying \eqref{abc9}.

We further decompose dyadically in frequency. For any $k,k_1,k_2\in\mathbb{Z}$ let
\begin{equation}\label{abc11}
\begin{split}
I_{m;k,k_1,k_2}:=\int_{I_m}\int_{\mathbb{R}^3\times\mathbb{R}^3}&q_m(s)m_3(\xi)m_1(\xi-\eta)m_2(\eta)\\
&\times\widehat{P_{k_1}U^{kg,\iota_1}_{\mathcal{L}_1}}(\xi-\eta,s)\widehat{P_{k_2}U_{\mathcal{L}_2}^{kg,\iota_2}}(\eta,s)\overline{\widehat{P_kU^{wa,\iota}_{\mathcal{L}}}(\xi,s)}\,d\xi d\eta ds,
\end{split}
\end{equation}
where $m_3$ is also a multiplier as in \eqref{abc9}. For \eqref{abc6.1} it remains to prove that if $n\leq N_1$, $n_1+n_2\leq n$, $\mathcal{L}_1\in\mathcal{V}_{n_1},\mathcal{L}_2\in\mathcal{V}_{n_2}$ then
\begin{equation}\label{abc12}
\sum_{k,k_1,k_2\in\mathbb{Z}}2^{-k}2^{2N(n)k^+}|I_{m;k,k_1,k_2}|\lesssim \varep_1^32^{2H(n)\delta m}
\end{equation}
for any $\iota,\iota_1,\iota_2\in\{+,-\}$, $t\in[0,T]$, $m\in\{0,\ldots,L+1\}$. 

In certain cases we need to integrate by parts in time. For this we write
\begin{equation*}
\begin{split}
\widehat{P_{k_1}U^{kg,\iota_1}_{\mathcal{L}_1}}&(\xi-\eta,s)\widehat{P_{k_2}U^{kg,\iota_2}_{\mathcal{L}_2}}(\eta,s)\overline{\widehat{P_kU^{wa,\iota}_{\mathcal{L}}}(\xi,s)}\\
&=e^{-is\Lambda_{kg,\iota_1}(\xi-\eta)-is\Lambda_{kg,\iota_2}(\eta)+is\Lambda_{wa,\iota}(\xi)}\widehat{P_{k_1}V^{kg,\iota_1}_{\mathcal{L}_1}}(\xi-\eta,s)\widehat{P_{k_2}V^{kg,\iota_2}_{\mathcal{L}_2}}(\eta,s)\overline{\widehat{P_kV^{wa,\iota}_{\mathcal{L}}}(\xi,s)}.
\end{split}
\end{equation*}
Let $\sigma=(wa,\iota)$, $\mu=(kg,\iota_1)$, $\nu=(kg,\iota_2)$, and $\Phi(\xi,\eta)=\Phi_{\sigma\mu\nu}(\xi,\eta)=\Lambda_{\sigma}(\xi)-\Lambda_{\mu}(\xi-\eta)-\Lambda_{\nu}(\eta)$. We define the trilinear operators $\mathcal{Q}=\mathcal{Q}^{\sigma\mu\nu}_{s}$ by
\begin{equation}\label{OP1}
\begin{split}
\mathcal{Q}[f,g,h]:=\int_{\mathbb{R}^3\times\mathbb{R}^3}e^{is\Phi_{\sigma\mu\nu}(\xi,\eta)}&m_3(\xi)m_1(\xi-\eta)m_2(\eta)\cdot\widehat{g}(\xi-\eta)\widehat{h}(\eta)\overline{\widehat{f}(\xi)}\,d\xi d\eta.
\end{split}
\end{equation}
Clearly
\begin{equation}\label{OP2}
I_{m;k,k_1,k_2}=\int_{I_m}q_m(s)\mathcal{Q}[P_kV^{wa,\iota}_{\mathcal{L}}(s),P_{k_1}V^{kg,\iota_1}_{\mathcal{L}_1}(s),P_{k_2}V^{kg,\iota_2}_{\mathcal{L}_2}(s)]\,ds.
\end{equation} 
This formula and integration by parts in time show that if $m\in[1,L]$ then
\begin{equation}\label{mnb51.0}
I_{m;k,k_1,k_2}=i\int_{I_m}q'_m(s)II^{0}_{k,k_1,k_2}(s)+q_m(s)\big[II^{1}_{k,k_1,k_2}(s)+II^{2}_{k,k_1,k_2}(s)+II^{3}_{k,k_1,k_2}(s)\big]\,ds,
\end{equation}
where
\begin{equation}\label{mnb51.2}
\begin{split}
&II^{0}_{k,k_1,k_2}(s):=\mathcal{Q}^{\ast}[P_kV^{wa,\iota}_{\mathcal{L}}(s),P_{k_1}V^{kg,\iota_1}_{\mathcal{L}_1}(s),P_{k_2}V^{kg,\iota_2}_{\mathcal{L}_2}(s)],\\
&II^{1}_{k,k_1,k_2}(s):=\mathcal{Q}^{\ast}[P_k(\partial_sV^{wa,\iota}_{\mathcal{L}})(s),P_{k_1}V^{kg,\iota_1}_{\mathcal{L}_1}(s),P_{k_2}V^{kg,\iota_2}_{\mathcal{L}_2}(s)],\\
&II^{2}_{k,k_1,k_2}(s):=\mathcal{Q}^{\ast}[P_kV^{wa,\iota}_{\mathcal{L}}(s),P_{k_1}(\partial_sV^{kg,\iota_1}_{\mathcal{L}_1})(s),P_{k_2}V^{kg,\iota_2}_{\mathcal{L}_2}(s)],\\
&II^{3}_{k,k_1,k_2}(s):=\mathcal{Q}^{\ast}[P_kV^{wa,\iota}_{\mathcal{L}}(s),P_{k_1}V^{kg,\iota_1}_{\mathcal{L}_1}(s),P_{k_2}(\partial_sV^{kg,\iota_2}_{\mathcal{L}_2})(s)],
\end{split}
\end{equation} 
and
\begin{equation}\label{mnb50}
\mathcal{Q}^{\ast}[f,g,h]:=\int_{\mathbb{R}^3\times\mathbb{R}^3}\frac{e^{is\Phi_{\sigma\mu\nu}(\xi,\eta)}}{\Phi_{\sigma\mu\nu}(\xi,\eta)}m_3(\xi)m_1(\xi-\eta)m_2(\eta)\cdot\widehat{g}(\xi-\eta)\widehat{h}(\eta)\overline{\widehat{f}(\xi)}\,d\xi d\eta.
\end{equation}

Without loss of generality, in proving \eqref{abc12}, we may assume that $n_1\leq n_2$. We often use the basic bound
\begin{equation}\label{abc12.1}
|I_{m;k,k_1,k_2}|\lesssim\sup_{s\in I_m}|I_m|\,\|P_kU^{wa,\iota}_{\mathcal{L}}(s)\|_{L^p}\|P_{k_1}U^{kg,\iota_1}_{\mathcal{L}_1}(s)\|_{L^{p_1}}\|P_{k_2}U^{kg,\iota_2}_{\mathcal{L}_2}(s)\|_{L^{p_2}},
\end{equation}
for any choice of $(p,p_1,p_2)\in\{(2,2,\infty),(2,\infty,2),(\infty,2,2)\}$, which follows from Lemma \ref{L1easy} (i). We consider two cases.

{\bf{Case 1.}} We prove now the bounds \eqref{abc12} when
\begin{equation}\label{nbv3}
n_1\geq 1.
\end{equation}
In particular, $n_1,n_2\leq n-1\leq N_1-1$. We apply \eqref{abc12.1} to bound
\begin{equation}\label{wer10.1}
\begin{split}
2^{-k}2^{2N(n)k^+}|I_{m;k,k_1,k_2}|&\lesssim \varep_1^3|I_m|2^{(H(n)+H(n_1)+H(n_2))\delta m}2^{-k/2}2^{3\min(k,k_1,k_2)/2}\\
&\times 2^{N(n)k^+-N(n_1)k_1^+-N(n_2)k_2^+},
\end{split}
\end{equation}
using the $L^2$ bounds \eqref{abc21}. Since $H(n_1)+H(n_2)\leq H(n)-190$ and $N(n)\leq \min(N(n_1),N(n_2))-10$, this suffices to prove \eqref{abc12} if $|I_m|\lesssim 1$. It also suffices to control the contribution of triplets $(k,k_1,k_2)$ for which $\min(k,k_1,k_2)\leq -m+180\delta m$ when $|I_m|\approx 2^m$.  
It remains to prove that if $m\in[1/\kappa,L]$ then
\begin{equation}\label{mnb56.1}
\sum_{k,k_1,k_2\in\mathbb{Z},\,\min(k,k_1,k_2)\geq -m+180\delta m}2^{-k}2^{2N(n)k^+}|I_{m;k,k_1,k_2}|\lesssim \varep_1^32^{2H(n)\delta m}.
\end{equation}

For this we integrate by parts in time (the method of normal forms), using the identities \eqref{mnb51.0}--\eqref{mnb50}. Notice that
\begin{equation}\label{mnb60}
2^{-k}2^{2N(n)k^+}|\mathcal{Q}^{\ast}[P_kf,P_{k_1}g,P_{k_2}h]|\lesssim 2^{-k/2}2^{3\max(k_1^+,k_2^+)}2^{2N(n)k^+}\|P_kf\|_{L^2}\|P_{k_1}g\|_{L^2}\|P_{k_2}k\|_{L^2},
\end{equation}
using just the Cauchy-Schwarz in the Fourier space and the lower bound \eqref{pha3}. To apply this, we need the $L^2$ bounds
\begin{equation}\label{mnb60.1}
\begin{split}
\|P_kV^{wa,\iota}_{\mathcal{L}}(s)\|_{L^2}&\lesssim \varep_1 2^{k/2}2^{H(n)\delta m}2^{-N(n)k^+},\\
\|P_{k_1}V^{kg,\iota_1}_{\mathcal{L}_1}(s)\|_{L^2}&\lesssim \varep_12^{H(n_1)\delta m}2^{-N(n_1)k_1^+},\\
\|P_{k_2}V^{kg,\iota_2}_{\mathcal{L}_2}(s)\|_{L^2}&\lesssim \varep_12^{H(n_2)\delta m}2^{-N(n_2)k_2^+},
\end{split}
\end{equation}
and
\begin{equation}\label{mnb60.2}
\begin{split}
2^m\|P_k(\partial_sV^{wa,\iota}_{\mathcal{L}})(s)\|_{L^2}&\lesssim \varep_12^{k/2}2^{H''_{wa}(n)\delta m}2^{-N(n+1)k^+-5k^+},\\
2^m\|P_{k_1}(\partial_sV^{kg,\iota_1}_{\mathcal{L}_1})(s)\|_{L^2}&\lesssim \varep_12^{H''_{kg}(n_1)\delta m}2^{-N(n_1+1)k_1^+-5k_1^+},\\
2^m\|P_{k_2}(\partial_sV^{kg,\iota_2}_{\mathcal{L}_2})(s)\|_{L^2}&\lesssim \varep_12^{H''_{kg}(n_2)\delta m}2^{-N(n_2+1)k_2^+-5k_2^+},
\end{split}
\end{equation}
which follow from \eqref{abc21}, \eqref{abc30}, and \eqref{abc31}. Since $n_1,n_2,n=n_1+n_2\geq 1$, we also have
\begin{equation}\label{mnb60.3}
\begin{split}
&H''_{wa}(n)+H(n_1)+H(n_2)=H(n)+H(n_1)+H(n_2)+160\leq 2H(n)-30,\\
&H(n)+H''_{kg}(n_1)+H(n_2)=H(n)+H(n_1)+H''_{kg}(n_2)\leq 2H(n)-30.
\end{split}
\end{equation}
Using these estimates and the definitions \eqref{mnb51.2}, it follows that, for any $s\in I_m$,
\begin{equation*}
2^{-k}2^{2N(n)k^+}\Big\{|II^{0}_{k,k_1,k_2}(s)|+\sum_{l=1}^32^m|II^{l}_{k,k_1,k_2}(s)|\Big\}\lesssim\varep_1^32^{2H(n)\delta m-30\delta m}2^{-\max(k_1^+,k_2^+)}.
\end{equation*}
The desired bound \eqref{mnb56.1} follows, and this completes the proof in the case \eqref{nbv3}.

{\bf{Case 2.}} The bounds \eqref{abc12} in the case $n_1=0$ follow from Lemma \ref{Tat} below, see \eqref{tat1}.
\end{proof}

\subsection{The bound on $U^{kg}_{\mathcal{L}}$} We estimate now the Klein-Gordon components.

\begin{proposition}\label{BootstrapEE5}
With the notation and hypothesis in Proposition \ref{bootstrap}, for any $t\in[0,T]$, $n\in[0,N_1]$ and $\mathcal{L}\in\mathcal{V}_n$
\begin{equation}\label{rom1}
\|U^{kg}_{\mathcal{L}}(t)\|_{H^{N(n)}}\lesssim\varep_{0}\langle t\rangle^{H(n)\delta}.
\end{equation}
\end{proposition}

\begin{proof} With $P:=\langle\nabla\rangle^{N(n)}\mathcal{L}$ we define the energy functional
\begin{equation*}
\begin{split}
\mathcal{E}_{kg}^\mathcal{L}(t)=\int_{\mathbb{R}^3}\big[(\partial_0Pv(t))^2&+(Pv(t))^2+\sum_{j=1}^3(\partial_jPv(t))^2\\
&+u(t)\sum_{i,j=1}^3B^{ij}\partial_iPv(t)\partial_jPv(t)-Eu(t)(Pv(t))^2\big]\,dx.
\end{split}
\end{equation*}
Using the second equation in \eqref{abc2} we calculate
\begin{equation*}
\begin{split}
\frac{d}{dt}\mathcal{E}_{kg}^\mathcal{L}=\int_{\mathbb{R}^3}\big\{&2P(u B^{\al\be}\partial_\al\partial_\be v+Euv)\partial_0Pv+\partial_0u\sum_{i,j=1}^3B^{ij}\partial_iPv\cdot \partial_jPv-E\partial_0u(Pv)^2\\
&+2u\sum_{i,j=1}^3B^{ij}\partial_iPv\cdot \partial_0\partial_jPv-2EuPv\cdot\partial_0Pv\big\}\,dx.
\end{split}
\end{equation*}
Recall that $B^{00}=0$. Using integration by parts the energy identity can be rewritten as 
\begin{equation}\label{rom3}
\begin{split}
&\frac{d}{dt}\mathcal{E}_{kg}^\mathcal{L}(t)=\int_{\mathbb{R}^3}[I^\mathcal{L}(t)+II^\mathcal{L}(t)]\,dx\\
&I^\mathcal{L}:=2B^{\al\be}\partial_0Pv\{P(u \partial_\al\partial_\be v)-u\partial_\al\partial_\be Pv\}+2E\partial_0Pv\{P(uv)-uPv\},\\
&II^\mathcal{L}:=4uB^{j0}\partial_j\partial_0 Pv\partial_0Pv+\partial_0uB^{ij}\partial_iPv\partial_jPv-E\partial_0u(Pv)^2-2\partial_juB^{ij}\partial_iPv\partial_0Pv.
\end{split}
\end{equation}
Using \eqref{vcx3.1}, for any $s\in[0,T]$
\begin{equation}\label{mnb31}
\|\nabla u(s)\|_{L^\infty}+\|\partial_0u(s)\|_{L^\infty}\lesssim \sum_{k\in\mathbb{Z}}\|P_kU^{wa}(s)\|_{L^\infty}\lesssim \varep_1(1+s)^{-1}.
\end{equation}
Therefore, for any $s\in[0,T]$
\begin{equation*}
\Big|\int_{\mathbb{R}^3}II^\mathcal{L}(s)\,dx\Big|\lesssim \varep_1(1+s)^{-1}\|U_\mathcal{L}^{kg}(s)\|_{H^{N(n)}}^2\lesssim \varep_1^3(1+s)^{-1+2H(n)\delta}.
\end{equation*}
Since $\mathcal{E}_{kg}^\mathcal{L}(s)\approx \|U^{kg}_\mathcal{L}(s)\|_{H^{N(n)}}^2$ for any $s\in[0,T]$, it suffices to prove that, for any $t\in[0,T]$,
\begin{equation}\label{rom5}
\Big|\int_0^t\int_{\mathbb{R}^3}I^\mathcal{L}(x,s)\,dxds\Big|\lesssim \varep_1^3\langle t\rangle^{2H(n)\delta}.
\end{equation}

The commutation relations \eqref{abc4} show that
\begin{equation*}
P(u \partial_\al\partial_\be v)=\langle\nabla\rangle^{N(n)}\{\partial_\al\partial_\be \mathcal{L}v\cdot u\}+{\sum}_\ast c_{\mathcal{L}_1,\mathcal{L}_2,\rho,\sigma}\langle\nabla\rangle^{N(n)}\{\partial_\rho\partial_\sigma \mathcal{L}_1v\cdot\mathcal{L}_2u\},
\end{equation*}
where $c_{\mathcal{L}_1,\mathcal{L}_2,\rho,\sigma}$ are suitable coefficients, and the sum ${\sum}_\ast$ is taken over operators $\mathcal{L}_1,\mathcal{L}_2\in\mathcal{V}_n$ with $|\mathcal{L}_1|+|\mathcal{L}_2|\leq n$, $|\mathcal{L}_1|\leq n-1$, and indices $\rho,\sigma\in\{0,1,2,3\}$. Also,
\begin{equation*}
P(u v)=\langle\nabla\rangle^{N(n)}\{\mathcal{L}v\cdot u\}+{\sum}_{\ast\ast} c'_{\mathcal{L}_1,\mathcal{L}_2}\langle\nabla\rangle^{N(n)}\{\mathcal{L}_1v\cdot \mathcal{L}_2u\},
\end{equation*}
where $c'_{\mathcal{L}_1,\mathcal{L}_2}$ are suitable coefficients, and the sum ${\sum}_{\ast\ast}$ is taken over operators $\mathcal{L}_1,\mathcal{L}_2\in\mathcal{V}_n$ with $|\mathcal{L}_1|+|\mathcal{L}_2|\leq n$, $|\mathcal{L}_1|\leq n-1$.

We express the functions $\mathcal{L}_1v$, $\mathcal{L}v$, $\mathcal{L}_2u$ in terms of the variables $U^{kg,\pm}_{\mathcal{L}_1}$, $U_{\mathcal{L}}^{kg,\pm}$, $U_{\mathcal{L}_2}^{wa,\pm}$, as in \eqref{variables4}--\eqref{on5}. Let $m_1$ denote a multiplier as in \eqref{abc9} and define
\begin{equation}\label{rom7}
b_n(\xi,\eta):=|\eta|^{-1}\langle\xi-\eta\rangle m_1(\xi-\eta)\langle\xi\rangle^{N(n)}[\langle\xi\rangle^{N(n)}-\langle\xi-\eta\rangle^{N(n)}].
\end{equation} 
With $q_m$ defined as in \eqref{nh2}, for \eqref{rom5} it suffices to prove that 
\begin{equation}\label{rom8}
\Big|\int_{I_m}\int_{\mathbb{R}^3\times\mathbb{R}^3}q_m(s)b_n(\xi,\eta)\cdot \widehat{U^{kg,\iota_1}_{\mathcal{L}}}(\xi-\eta,s)\widehat{U^{wa,\iota_2}}(\eta,s)\overline{\widehat{U^{kg,\iota}_{\mathcal{L}}}(\xi,s)}\,d\xi d\eta ds\Big|\lesssim \varep_1^32^{2H(n)\delta m},
\end{equation}
and
\begin{equation}\label{rom9}
\begin{split}
\Big|\int_{I_m}\int_{\mathbb{R}^3\times\mathbb{R}^3}&q_m(s)|\eta|^{-1}(1+|\xi|^2)^{N(n)}\langle\xi-\eta\rangle m_1(\xi-\eta)\\
&\times\widehat{U^{kg,\iota_1}_{\mathcal{L}_1}}(\xi-\eta,s)\widehat{U^{wa,\iota_2}_{\mathcal{L}_2}}(\eta,s)\overline{\widehat{U^{kg,\iota}_{\mathcal{L}}}(\xi,s)}\,d\xi d\eta ds\Big|\lesssim \varep_1^32^{2H(n)\delta m},
\end{split}
\end{equation}
provided that $\iota,\iota_1,\iota_2\in\{+,-\}$, and $\mathcal{L}_1\in\mathcal{V}_{n_1},\mathcal{L}_2\in\mathcal{V}_{n_2}$, $n_1+n_2\leq n$, $n_1\leq n-1$.

{\bf{Step 1.}} We start by proving the bounds \eqref{rom8}. We decompose dyadically in frequency. For any $k,k_1,k_2\in\mathbb{Z}$ let
\begin{equation}\label{rom10}
J_{m;k,k_1,k_2}^{n,0}:=\int_{I_m}\int_{\mathbb{R}^3\times\mathbb{R}^3}q_m(s)b_n(\xi,\eta)\widehat{P_{k_1}U^{kg,\iota_1}_{\mathcal{L}}}(\xi-\eta,s)\widehat{P_{k_2}U^{wa,\iota_2}}(\eta,s)\overline{\widehat{P_kU^{kg,\iota}_{\mathcal{L}}}(\xi,s)}\,d\xi d\eta ds.
\end{equation}

We have the multiplier bounds
\begin{equation}\label{rom11}
\big\|\mathcal{F}^{-1}\{\varphi_k(\xi)\varphi_{k_1}(\xi-\eta)\varphi_{k_2}(\eta)b(\xi,\eta)\}\big\|_{L^1(\mathbb{R}^3\times\mathbb{R}^3)}\lesssim 2^{\min(k_1^+,k^+)}2^{(2N(n)-1)\overline{k}^+},
\end{equation}
where, as before, $\max(k,k_1,k_2)=\overline{k}$ and $\min(k,k_1,k_2)=\underline{k}$. This can be seen easily when $\overline{k}\leq 0$. On the other hand, if $\overline{k}\geq 0$ then \eqref{rom11} can be proved by analyzing the three cases $k=\underline{k}$, $k_1=\underline{k}$, and $k_2=\underline{k}$, and using the cancellation in the multiplier in the last case.

Using Lemma \ref{L1easy}, \eqref{vcx3.1}, and \eqref{abc23}, we have
\begin{equation}\label{mnb42}
|J^{n,0}_{m;k,k_1,k_2}|\lesssim \int_{I_m}2^{2N(n)\overline{k}^+}\|P_{k_1}U^{kg}_{\mathcal{L}}(s)\|_{L^2}\|P_kU^{kg}_{\mathcal{L}}(s)\|_{L^2}\cdot \varep_12^{-m}2^{-k_2^+}2^{k_2^-/2}\,ds
\end{equation}
if $k_2=\underline{k}$, and
\begin{equation}\label{mnb43}
|J^{n,0}_{m;k,k_1,k_2}|\lesssim |I_m|\varep_1^32^{-m+\kappa m/2}2^{\underline{k}^-/2}2^{-\overline{k}^+/2}
\end{equation}
if $k_2\geq \underline{k}+1$. Indeed, this follows by estimating the lowest frequency factor in $L^\infty$ and the other two factors in $L^2$, except for the case $n=N_1$, $k_2\geq \underline{k}+1$ when we estimate the high frequency wave component in $L^\infty$. The gain of $1/2$ high-order derivative in \eqref{mnb43} is due to the gain of derivative in \eqref{rom11}. It follows from \eqref{mnb42} that
\begin{equation*}
\sum_{k,k_1,k_2\in\mathbb{Z},\,k_2=\underline{k}}|J^{n,0}_{m;k,k_1,k_2}|\lesssim \varep_1^32^{2\delta m}.
\end{equation*} 
Moreover, the sum of $|J_{m;k,k_1,k_2}^{n,0}|$ over $k,k_1,k_2$ with $\underline{k}\leq -2\kappa m$ or $\overline{k}\geq 2\kappa m$ is also suitably bounded due to \eqref{mnb43}. For \eqref{rom8} it remains to prove that
\begin{equation}\label{rom12}
\sup_{k,k_1,k_2\in[-2\kappa m,2\kappa m]}\big|J^{n,0}_{m;k,k_1,k_2}\big|\lesssim \varep_{1}^3
\end{equation}
provided that $t\in[0,T]$, $\iota,\iota_1,\iota_2\in\{+,-\}$, and $m\in\{1/\kappa,\ldots,L\}$.

This follows easily by integration by parts in time, as in the proof of Proposition \ref{BootstrapEE4}. This procedure gains a factor of $2^{-m}$ and losses at most a factor of $2^{10\kappa m}$ when applying Lemma \ref{pha2} (ii), in the range of frequencies as in \eqref{rom12}.
\medskip

{\bf{Step 2.}} We prove now the bounds \eqref{rom9}. Notice that the case $n_1=0$ follows from Lemma \ref{Tat} below, after making changes of variables. Recall that $n_1\leq n-1$, so we may assume that
\begin{equation}\label{sta1}
n_1,n_2\in[1,N_1-1],\qquad n_1+n_2=n.
\end{equation}
We define 
\begin{equation*}
\begin{split}
J^{n,1}_{m;k,k_1,k_2}:=\int_{I_m}\int_{\mathbb{R}^3\times\mathbb{R}^3}&q_m(s)|\eta|^{-1}(1+|\xi|^2)^{N(n)}\langle\xi-\eta\rangle m_1(\xi-\eta)\\
&\times\widehat{P_{k_1}U^{kg,\iota_1}_{\mathcal{L}_1}}(\xi-\eta,s)\widehat{P_{k_2}U^{wa,\iota_2}_{\mathcal{L}_2}}(\eta,s)\overline{\widehat{P_kU^{kg,\iota}_{\mathcal{L}}}(\xi,s)}\,d\xi d\eta ds,
\end{split}
\end{equation*}
and we have to prove that
\begin{equation}\label{rom9.2}
\sum_{k,k_1,k_2\in\mathbb{Z}}|J_{m;k,k_1,k_2}^{n,1}|\lesssim \varep_1^32^{2H(n)\delta m}.
\end{equation}

Using just the $L^2$ bounds \eqref{abc21} and Sobolev embedding we see that
\begin{equation}\label{rom50}
|J_{m;k,k_1,k_2}^{n,1}|\lesssim \varep_1^3|I_m|2^{(2H(n)-190)\delta m}2^{N(n)k^+-N(n_1)k_1^+-N(n_2)k_2^+}2^{k_1^+}2^{-k_2/2}2^{3\min(k,k_1,k_2)/2},
\end{equation}
since $H(n_1)+H(n_2)\leq H(n)-190$. This suffices to prove the desired bound when $|I_m|\lesssim 1$. If $|I_m|\approx 2^m$ (so $m\in[1/\kappa,L]$), the bound \eqref{rom50} still suffices to control the contribution of the triplets $(k,k_1,k_2)$ for which $\min(k,k_1,k_2)\leq -m+180\delta m$. It remains to prove that
\begin{equation*}
\sum_{k,k_1,k_2\in\mathbb{Z},\,\min(k,k_1,k_2)\geq -m+180\delta m}|J^{n,1}_{m;k,k_1,k_2}|\lesssim \varep_1^32^{2H(n)\delta m}.
\end{equation*}
We notice that this is similar to the bound \eqref{mnb56.1} in Proposition \ref{BootstrapEE4}, essentially with the indices $k$ and $k_2$ reversed. The proof follows by the same integration by parts argument, using just the $L^2$ estimates \eqref{mnb60.1}--\eqref{mnb60.2}. This completes the proof of the proposition.
\end{proof}

\subsection{The main cubic bulk estimate} We prove now suitable bounds on the cubic bulk terms arising in the energy estimates in Propositions \ref{BootstrapEE4} and \ref{BootstrapEE5}, corresponding to the case when all the vector-fields hit one of the profiles.

\begin{lemma}\label{Tat}
Assume that $n\in[0,N_1]$ and $\mathcal{L},\mathcal{L}_2\in\mathcal{V}_n$, $t\in[0,T]$, and $m\in\{0,\ldots,L+1\}$. As in \eqref{abc11} (with $(n_1,n_2,n)=(0,n,n)$), for any $k,k_1,k_2\in\mathbb{Z}$ and $\iota,\iota_1,\iota_2\in\{+,-\}$, let
\begin{equation}\label{tat0}
\begin{split}
I_{m;k,k_1,k_2}:=\int_{I_m}\int_{\mathbb{R}^3\times\mathbb{R}^3}&q_m(s)m_3(\xi)m_1(\xi-\eta)m_2(\eta)\\
&\times\widehat{P_{k_1}U^{kg,\iota_1}}(\xi-\eta,s)\widehat{P_{k_2}U_{\mathcal{L}_2}^{kg,\iota_2}}(\eta,s)\overline{\widehat{P_kU^{wa,\iota}_{\mathcal{L}}}(\xi,s)}\,d\xi d\eta ds,
\end{split}
\end{equation}
where $m_1,m_2,m_3$ are multipliers as in \eqref{abc9}. Then 
\begin{equation}\label{tat1}
\sum_{k,k_1,k_2\in\mathbb{Z}}2^{-k}\big(2^{2N(n)k^+}+2^{e(n)2N(n)k_2^+}2^{e(n)k_1^+}\big)|I_{m;k,k_1,k_2}|\lesssim \varep_1^32^{2H(n)\delta m},
\end{equation}
where $e(0):=0$ and $e(n):=1$ for $n\in[1,N_1]$.
\end{lemma} 

\begin{proof} Let $\underline{k}:=\min(k,k_1,k_2)$ and $\overline{k}:=\max(k,k_1,k_2)$. Using just the $L^2$ bounds \eqref{abc21} and \eqref{vcx1.1}, we have
\begin{equation}\label{wer10.2}
\begin{split}
2^{-k}|I_{m;k,k_1,k_2}|\lesssim \varep_1^3|I_m|2^{2H(n)\delta m}&2^{-k/2}2^{3\min(k,k_1,k_2)/2}2^{k_1^-+\kappa k_1^-}\\
&\times 2^{-N(n)k^+-N(n)k_2^+}2^{-N(0)k_1^++2k_1^+}.
\end{split}
\end{equation}
This suffices bound the contribution of triplets $(k,k_1,k_2)$ for which $\underline{k}\leq -m$ (in the case $n=0$ we use also a similar bound with the roles of $k_1$ and $k_2$ reversed). It also suffices to prove the desired bound \eqref{tat1} when $|I_m|\lesssim 1$.

{\bf{Step 1.}} We show first that if $m\in[1/\delta,L]$ then
\begin{equation}\label{njm1}
\sum_{k,k_1,k_2\in\mathbb{Z},\,\underline{k}\geq -m,\,k\leq -0.6m}2^{-k}\big(2^{2N(n)k^+}+2^{e(n)2N(n)k_2^+}2^{e(n)k_1^+}\big)|I_{m;k,k_1,k_2}|\lesssim \varep_1^32^{2H(n)\delta m}.
\end{equation}
This is the case of small frequencies $2^k$. The estimates \eqref{wer10.2} cleary suffice to control the contribution of the triplets $(k,k_1,k_2)$ for which $k\leq -0.6m$ and $2^{k_1}\lesssim 2^{-0.4m}$. They also suffice to control the contribution of the triplets $(k,k_1,k_2)$ for which $k\leq-0.6$ and $k+k_1^-(1+\kappa)-35k_1^+\leq -m$.

It remains to bound the contribution of the triplets $(k,k_1,k_2)$ for which 
\begin{equation}\label{njm2}
k\in[-m,-0.6m]\qquad\text{ and }\qquad k+k_1^-(1+\kappa)-35k_1^+\geq -m.
\end{equation}
In particular, $k_1\geq -m/2+100$. Let $J:=k_1^-+m-40$ and decompose $P_{k_1}V^{kg,\iota_1}=V^{kg,\iota_1}_{\leq J,k_1}+V^{kg,\iota_1}_{>J,k_1}$ as in \eqref{ku10}. With $\mathcal{Q}$ as in \eqref{OP1}, let
\begin{equation*}
\begin{split}
&I^1_{m;k,k_1,k_2}=\int_{I_m}q_m(s)\mathcal{Q}[P_kV^{wa,\iota}_{\mathcal{L}}(s),V^{kg,\iota_1}_{\leq J,k_1}(s),P_{k_2}V^{kg,\iota_2}_{\mathcal{L}_2}(s)]\,ds,\\
&I^2_{m;k,k_1,k_2}=\int_{I_m}q_m(s)\mathcal{Q}[P_kV^{wa,\iota}_{\mathcal{L}}(s),V^{kg,\iota_1}_{>J,k_1}(s),P_{k_2}V^{kg,\iota_2}_{\mathcal{L}_2}(s)]\,ds.
\end{split}
\end{equation*}  

Using \eqref{abc21} and \eqref{ku10} we estimate
\begin{equation*}
\begin{split}
2^{-k}|I^1_{m;k,k_1,k_2}|&\lesssim 2^m2^{-k}\sup_{s\in I_m}\|P_kV^{wa,\iota}_{\mathcal{L}}(s)\|_{L^2}\|V^{kg,\iota_1}_{\leq J,k_1}(s)\|_{L^\infty}\|P_{k_2}V^{kg,\iota_2}_{\mathcal{L}_2}(s)\|_{L^2}\\
&\lesssim\varep_1^32^{2H(n)\delta m}2^{-m/2}2^{-k/2}2^{-k_1^-/2}2^{-N(n)k^+-N(n)k_2^+}2^{-N_0k_1^++(d'+6)k_1^+}.
\end{split}
\end{equation*}
Therefore, for $(k,k_1,k_2)$ as in \eqref{njm2},
\begin{equation}\label{njm4}
\begin{split}
2^{-k}\big(2^{2N(n)k^+}+2^{e(n)2N(n)k_2^+}&2^{e(n)k_1^+}\big)|I^1_{m;k,k_1,k_2}|\\
&\lesssim \varep_1^32^{2H(n)\delta m}2^{-m/2}2^{-k/2}2^{-k_1^-/2}2^{e(n)N(n)k_2^+}2^{-N(1)k_1^++12k_1^+}.
\end{split}
\end{equation} 

Similarly, using \eqref{abc21}, \eqref{ku10}, and $L^2$ bounds we estimate
\begin{equation*}
\begin{split}
2^{-k}|I^2_{m;k,k_1,k_2}|&\lesssim 2^m2^{-k}2^{3k/2}\sup_{s\in I_m}\|P_kV^{wa,\iota}_{\mathcal{L}}(s)\|_{L^2}\|V^{kg,\iota_1}_{>J,k_1}(s)\|_{L^2}\|P_{k_2}V^{kg,\iota_2}_{\mathcal{L}_2}(s)\|_{L^2}\\
&\lesssim\varep_1^32^{2H(n)\delta m}2^{10\delta m}2^{k-k_1^-}2^{-N(n)k^+-N(n)k_2^+}2^{-N(1)k_1^+}.
\end{split}
\end{equation*}
Therefore, for $(k,k_1,k_2)$ as in \eqref{njm2},
\begin{equation}\label{njm5}
\begin{split}
2^{-k}\big(2^{2N(n)k^+}+2^{e(n)2N(n)k_2^+}&2^{e(n)k_1^+}\big)|I^1_{m;k,k_1,k_2}|\\
&\lesssim \varep_1^32^{2H(n)\delta m}2^{10\delta m}2^{k-k_1^-}2^{e(n)N(n)k_2^+}2^{-N(1)k_1^++k_1^+}.
\end{split}
\end{equation} 

Notice that for $(k,k_1,k_2)$ as in \eqref{njm2} we have $2^{-m/2}2^{-k/2}2^{-k_1^-/2}2^{12k_1^+}\lesssim 2^{-\delta|m+k|+\delta k_1^--k_1^+}$ and $2^{10\delta m}2^{k-k_1^-}2^{k_1^+}\lesssim 2^{-\delta m-k_1^+}$. Therefore the bounds \eqref{njm4}-\eqref{njm5} suffice to bound the remaining contribution of the triplets $(k,k_1,k_2)$ as in \eqref{njm2}, as claimed in \eqref{njm1}.

{\bf{Step 2.}} We show now that if $m\in[1/\delta,L]$ then
\begin{equation}\label{njm6}
\sum_{k,k_1,k_2\in\mathbb{Z},\,\underline{k}\geq -m,\,k\geq -0.6m,\,\overline{k}\leq 4\kappa m}2^{-k}\big(2^{2N(n)k^+}+2^{e(n)2N(n)k_2^+}2^{e(n)k_1^+}\big)|I_{m;k,k_1,k_2}|\lesssim \varep_1^32^{2H(n)\delta m}.
\end{equation}
To prove this we would like to integrate by parts in time (the method of normal forms). We examine the identities \eqref{mnb51.0}--\eqref{mnb50} and estimate $|II^{l}_{k,k_1,k_2}(s)|$, $l\in\{0,1,2,3\}$, using Lemma \ref{pha2} (ii). The bounds we need are
\begin{equation}\label{mnb51}
\begin{split}
\|P_kV^{wa,\iota}_{\mathcal{L}}(s)\|_{L^2}+2^m\|P_k(\partial_sV^{wa,\iota}_{\mathcal{L}})(s)\|_{L^2}&\lesssim \varep_12^{k^-/2}2^{H''_{wa}(n)\delta m}2^{-N(n+1)k^+-5k^+},\\
\|P_{k_2}V^{kg,\iota_2}_{\mathcal{L}_2}(s)\|_{L^2}+2^m\|P_{k_2}(\partial_sV^{kg,\iota_2}_{\mathcal{L}_2})(s)\|_{L^2}&\lesssim \varep_12^{H''_{kg}(n)\delta m}2^{-N(n+1)k_2^+-5k_2^+},\\
\|e^{-is\Lambda_{kg,\iota_1}}P_{k_1}V^{kg,\iota_1}(s)\|_{L^\infty}&\lesssim \varep_12^{k_1^-/2}2^{-m+10\delta m}2^{-N_0k_1^++(d+2)k_1^+},
\end{split}
\end{equation}
which follow from \eqref{abc21}, \eqref{abc30}, and \eqref{abc31}. Using \eqref{pha4}, and recalling the assumptions on the triplets $(k,k_1,k_2)$ in \eqref{njm6}, it follows that, for $s\in I_m$,
\begin{equation}\label{mnb51.1}
\begin{split}
|II^{0}_{k,k_1,k_2}(s)|&+2^m|II^{1}_{k,k_1,k_2}(s)|+2^m|II^{3}_{k,k_1,k_2}(s)|\lesssim\varep_1^32^{-k^-/2}2^{k_1^-/2}2^{-m+\kappa m}.
\end{split}
\end{equation}

To estimate $|II^{n,2}_{k,k_1,k_2}(s)|$ we need an additional $L^\infty$ bound, namely
\begin{equation}\label{mnb51.4}
\|e^{-is\Lambda_{kg,\iota_1}}P_{k_1}(\partial_sV^{kg,\iota_1})(s)\|_{L^\infty}\lesssim \varep_12^{-2m+\kappa m}2^{-N(1)k_1^++5k_1^+},
\end{equation}
which follows from \eqref{abc23.2}--\eqref{abc23} and the identity $e^{-is\Lambda_{kg}}(\partial_sV^{kg})(s)=\mathcal{N}^{kg}(s)$. Using also the $L^2$ bounds in the first two lines of \eqref{mnb51}, together with Lemma \ref{pha2} (ii), we estimate
\begin{equation*}
2^m|II^{2}_{k,k_1,k_2}(s)|\lesssim\varep_1^32^{-k^-/2}2^{-m+2\kappa m}.
\end{equation*}
Using also the bounds \eqref{mnb51.1} and the formula \eqref{mnb51.0}, it follows that
\begin{equation*}
|I_{m;k,k_1,k_2}|\lesssim\varep_1^32^{-k^-/2}2^{-m+2\kappa m},
\end{equation*}
for triplets $(k,k_1,k_2)$ as in \eqref{njm6}. The desired bound \eqref{njm6} follows. 
 
{\bf{Step 3.}} Finally we show that if $m\in[1/\delta,L]$ then
\begin{equation}\label{njm12}
\sum_{k,k_1,k_2\in\mathbb{Z},\,\underline{k}\geq -m,\,k\geq -0.6m,\,\overline{k}\geq 4\kappa m}2^{-k}\big(2^{2N(n)k^+}+2^{e(n)2N(n)k_2^+}2^{e(n)k_1^+}\big)|I_{m;k,k_1,k_2}|\lesssim \varep_1^32^{2H(n)\delta m}.
\end{equation}
Using \eqref{abc12.1} with $(p,p_1,p_2)=(2,\infty,2)$ and the bounds \eqref{abc21} and \eqref{abc23}, we have
\begin{equation}\label{wer10.3}
2^{-k}|I_{m;k,k_1,k_2}|\lesssim \varep_1^32^{10\delta m}2^{2H(n)\delta m}2^{-k/2}2^{k_1^-/2}2^{-N(n)k^+-N(n)k_2^+}2^{-N_0k_1^++(d+2)k_1^+}.
\end{equation}
Notice the factor $2^{-k/2}$ in \eqref{wer10.3}, which is favorable when $k$ is large. These bounds clearly suffice to control the contribution of the triplets $(k,k_1,k_2)$ in \eqref{njm12} for which $k_1=\min(k,k_1,k_2)$. 

We consider now the sum over triplets $(k,k_1,k_2)$ as in \eqref{njm12} for which $k_2=\min(k,k_1,k_2)$. We use \eqref{abc12.1} with $(p,p_1,p_2)=(2,2,\infty)$ and \eqref{abc23}, so
\begin{equation}\label{njm14}
2^{-k}|I_{m;k,k_1,k_2}|\lesssim \varep_1^32^{\kappa m}2^{2H(n)\delta m}2^{-k/2}2^{-N(n)k^+-N(0)k_1^+}2^{-|k_2|/2},
\end{equation} 
provided that $n\leq N_1-1$. The estimates \eqref{wer10.3} suffice to bound the contribution of the  triplets $(k,k_1,k_2)$ for which $k_2=\min(k,k_1,k_2)$ if $n\geq 2$, while the estimates \eqref{njm14} suffice in the remaining cases $n\in\{0,1\}$.

Finally, we consider the sum over triplets $(k,k_1,k_2)$ as in \eqref{njm12} for which $k=\min(k,k_1,k_2)$. We use \eqref{abc12.1} with $(p,p_1,p_2)=(\infty,2,2)$ and \eqref{abc23.2}, so
\begin{equation}\label{njm15}
2^{-k}|I_{m;k,k_1,k_2}|\lesssim \varep_1^32^{\kappa m}2^{2H(n)\delta m}2^{-N(n)k_2^+-N(0)k_1^+}2^{-4k^+},
\end{equation} 
provided that $n\leq N_1-1$. This suffices to complete the proof of \eqref{njm12} if $n\leq N_1-1$. After these reductions, for \eqref{njm12} it remains to show that if $n=N_1$ then 
\begin{equation}\label{njm12.2}
\sum_{k,k_1,k_2\in\mathbb{Z},\,k=\underline{k}\geq -0.6m,\,\overline{k}\geq 4\kappa m}2^{-k}2^{2N(n)k_2^+}2^{k_1^+}|I_{m;k,k_1,k_2}|\lesssim \varep_1^32^{2H(n)\delta m}.
\end{equation}
This follows by the same integration by parts argument as in Step 2 above, using Lemma \ref{pha2} (ii) and the bounds \eqref{mnb51} and \eqref{mnb51.4}. This completesthe proof of the lemma.
\end{proof}

\section{Bounds on the profiles, I: weighted $L^2$ norms}\label{DiEs}

In this section we prove the bounds in \eqref{bootstrap3.2}. These bounds will be derived by elliptic estimates from the bounds \eqref{bootstrap3.1} proved in the previous two sections. We also need two identities that connect the vector-fields $\Gamma_l$ with weighted norms on the profiles.

\begin{lemma}\label{ident}
Assume $\mu\in\{wa,kg\}$ and
\begin{equation}\label{wer0}
(\partial_t+i\Lambda_\mu)U=\mathcal{N},
\end{equation}
on $\mathbb{R}^3\times[0,T]$. If $V(t)=e^{it\Lambda_\mu}U(t)$ and $l\in\{1,2,3\}$ then, for any $t\in[0,T]$,
\begin{equation}\label{wer1}
\widehat{\Gamma_lU}(\xi,t)=i(\partial_{\xi_l}\widehat{\mathcal{N}})(\xi,t)+e^{-it\Lambda_{\mu}(\xi)}\partial_{\xi_l}[\Lambda_{\mu}(\xi)\widehat{V}(\xi,t)].
\end{equation}
\end{lemma}

\begin{proof} We calculate
\begin{equation*}
\begin{split}
&\widehat{\Gamma_lU}(\xi,t)=\mathcal{F}\{x_l\partial_tU+t\partial_lU\}(\xi,t)\\
&=i(\partial_{\xi_l}\widehat{\mathcal{N}})(\xi,t)+\partial_{\xi_l}[\Lambda_{\mu}(\xi)\widehat{U}(\xi,t)]+it\xi_l\widehat{U}(\xi,t)\\
&=i(\partial_{\xi_l}\widehat{\mathcal{N}})(\xi,t)+e^{-it\Lambda_{\mu}(\xi)}\partial_{\xi_l}[\Lambda_{\mu}(\xi)\widehat{V}(\xi,t)]-it(\partial_{\xi_l}\Lambda_{\mu})(\xi)e^{-it\Lambda_{\mu}(\xi)}\Lambda_{\mu}(\xi)\widehat{V}(\xi,t)+it\xi_l\widehat{U}(\xi,t).
\end{split}
\end{equation*}
This gives \eqref{wer1} since $(\partial_{\xi_l}\Lambda_{\mu})(\xi)\Lambda_\mu(\xi)=\xi_l$.
\end{proof}

We prove now the bounds \eqref{bootstrap3.2}.

\begin{proposition}\label{DiEs1}
With the hypothesis in Proposition \ref{bootstrap}, for any $t\in[0,T]$, $n\in[0,N_1-1]$, $k\in\mathbb{Z}$, $\mathcal{L}\in\mathcal{V}_n$, and $l\in\{1,2,3\}$ we have
\begin{equation}\label{wer2}
2^{N(n+1)k^+}\big\{2^{k/2}\|\varphi_k(\xi)(\partial_{\xi_l}\widehat{V^{wa}_{\mathcal{L}}})(\xi,t)\|_{L^2_\xi}+2^{k^+}\|\varphi_k(\xi)(\partial_{\xi_l}\widehat{V^{kg}_{\mathcal{L}}})(\xi,t)\|_{L^2_\xi}\big\}\lesssim\varep_0\langle t\rangle^{H(n+1)\delta}.
\end{equation}
\end{proposition}

\begin{proof} The identity \eqref{wer1} for $U^{wa}_{\mathcal{L}}$ gives
\begin{equation*}
\widehat{\Gamma_lU^{wa}_{\mathcal{L}}}(\xi,t)=i(\partial_{\xi_l}\widehat{\mathcal{N}^{wa}_{\mathcal{L}}})(\xi,t)+e^{-it\Lambda_{wa}(\xi)}\partial_{\xi_l}[\Lambda_{wa}(\xi)\widehat{V^{wa}_{\mathcal{L}}}(\xi,t)].
\end{equation*}
Therefore
\begin{equation*}
e^{-it\Lambda_{wa}(\xi)}\Lambda_{wa}(\xi)(\partial_{\xi_l}\widehat{V^{wa}_{\mathcal{L}}})(\xi)=\widehat{\Gamma_lU^{wa}_{\mathcal{L}}}(\xi)-i(\partial_{\xi_l}\widehat{\mathcal{N}^{wa}_{\mathcal{L}}})(\xi)+e^{-it\Lambda_{wa}(\xi)}(\xi_l/|\xi|)\widehat{V^{wa}_{\mathcal{L}}}(\xi).
\end{equation*}
We multiply all the terms by $2^{-k/2}\varphi_k(\xi)$ and take $L^2$ norms to show that
\begin{equation}\label{wer3}
\begin{split}
2^{k/2}\|\varphi_k(\xi)(\partial_{\xi_l}\widehat{V^{wa}_{\mathcal{L}}})(\xi)\|_{L^2_\xi}&\lesssim 2^{-k/2}\|\varphi_k(\xi)\widehat{\Gamma_lU^{wa}_{\mathcal{L}}}(\xi)\|_{L^2}\\
&+2^{-k/2}\|\varphi_k(\xi)(\partial_{\xi_l}\widehat{\mathcal{N}^{wa}_{\mathcal{L}}})(\xi)\|_{L^2}+2^{-k/2}\|\varphi_k(\xi)\widehat{V^{wa}_{\mathcal{L}}}(\xi)\|_{L^2}.
\end{split}
\end{equation}
It follows from \eqref{abc30.3} and Proposition \ref{BootstrapEE4} that
\begin{equation*}
\begin{split}
2^{-k/2}\|\varphi_k(\xi)(\partial_{\xi_l}\widehat{\mathcal{N}^{wa}_{\mathcal{L}}})(\xi)\|_{L^2}&\lesssim \varep_1^2\langle t\rangle^{H''_{wa}(n)\delta}2^{-N(n)k^++5k^+},\\
2^{-k/2}\|\varphi_k(\xi)\widehat{\Gamma_lU^{wa}_{\mathcal{L}}}(\xi)\|_{L^2}&\lesssim \varep_0\langle t\rangle^{H(n+1)\delta}2^{-N(n+1)k^+},\\
2^{-k/2}\|\varphi_k(\xi)\widehat{V^{wa}_{\mathcal{L}}}(\xi)\|_{L^2}&\lesssim \varep_0\langle t\rangle^{H(n)\delta}2^{-N(n)k^+}.
\end{split}
\end{equation*}
The desired inequality for the wave component in \eqref{wer2} follows using \eqref{wer3}, since $N(n+1)\leq N(n)-d$ and $H(n+1)\geq \max(H''_{wa}(n),H(n))$.

The inequality for the Klein--Gordon component in \eqref{wer2} follows similarly, using the identity \eqref{wer1} for $\mu=kg$, the energy estimates in Proposition \ref{BootstrapEE5}, and the bounds \eqref{abc31.3}.
\end{proof}

\section{Bounds on the profiles, II: the Klein--Gordon $Z$ norm}\label{DiEs2}

In this section we prove the bounds in \eqref{bootstrap3.4} for the Klein--Gordon component. We notice that, unlike the energy norms, the $Z$ norms of the two profiles are not allowed to grow slowly in time. Because of this we need to renormalize the Klein--Gordon profile.

\subsection{Renormalization}\label{Nor1}

We start from the equation $\partial_tV^{kg}=e^{it\Lambda_{kg}}\mathcal{N}^{kg}$ for the profile $V^{kg}=V^{kg,+}$, where $\mathcal{N}^{kg}=uB^{\al\be}\partial_\al\partial_\be v+Euv$. In the Fourier space this becomes
\begin{equation*}
\partial_t\widehat{V^{kg}}(\xi,t)=\frac{1}{(2\pi)^3}\int_{\mathbb{R}^3}e^{it\Lambda_{kg}(\xi)}\widehat{u}(\eta,t)\big[B^{\al\be}\widehat{\partial_\al\partial_\be v}(\xi-\eta,t)+E\widehat{v}(\xi-\eta)\big]\,d\eta.
\end{equation*} 
Recall that $B^{00}=0$. The formulas in the second line of \eqref{on5} show that
\begin{equation*}
\begin{split}
B^{\al\be}\widehat{\partial_\al\partial_\be v}(\rho,t)&=(-B^{jk}\rho_j\rho_k)\frac{i[e^{-it\Lambda_{kg}(\rho)}\widehat{V^{kg,+}}(\rho,t)-e^{it\Lambda_{kg}(\rho)}\widehat{V^{kg,-}}(\rho,t)]}{2\Lambda_{kg}(\rho)}\\
&+(2B^{0k}\rho_k)\frac{i[e^{-it\Lambda_{kg}(\rho)}\widehat{V^{kg,+}}(\rho,t)+e^{it\Lambda_{kg}(\rho)}\widehat{V^{kg,-}}(\rho,t)]}{2}.
\end{split}
\end{equation*}
Therefore
\begin{equation}\label{Nor2}
\begin{split}
\partial_t\widehat{V^{kg}}(\xi,t)=\frac{1}{(2\pi)^3}\int_{\mathbb{R}^3}ie^{it\Lambda_{kg}(\xi)}\widehat{u}(\eta,t)\{&e^{-it\Lambda_{kg}(\xi-\eta)}\widehat{V^{kg,+}}(\xi-\eta,t)\mathfrak{q}_+(\xi-\eta)\\
&+e^{it\Lambda_{kg}(\xi-\eta)}\widehat{V^{kg,-}}(\xi-\eta,t)\mathfrak{q}_-(\xi-\eta)\}\,d\eta,
\end{split}
\end{equation} 
where
\begin{equation}\label{Nor3}
\mathfrak{q}_\pm(\rho):=\mp\frac{B^{jk}\rho_j\rho_k}{2\Lambda_{kg}(\rho)}+B^{0k}\rho_k\pm\frac{E}{2\Lambda_{kg}(\rho)}.
\end{equation}

We would like to eliminate the bilinear interaction between $u$ and $V^{kg,+}$ in the first line of \eqref{Nor2} corresponding to $|\eta|\ll 1$. To extract the main term we approximate, heuristically, 
\begin{equation*}
\begin{split}
\frac{1}{(2\pi)^3}\int_{|\eta|\ll \langle t\rangle^{-1/2}}&ie^{it\Lambda_{kg}(\xi)}\widehat{u}(\eta,t)e^{-it\Lambda_{kg}(\xi-\eta)}\widehat{V^{kg,+}}(\xi-\eta,t)\mathfrak{q}_+(\xi-\eta)\,d\eta\\
&\approx i\widehat{V^{kg,+}}(\xi,t)\mathfrak{q}_+(\xi)\frac{1}{(2\pi)^3}\int_{|\eta|\ll \langle t\rangle^{-1/2}}e^{it\eta\cdot\nabla\Lambda_{kg}(\xi)}\widehat{u}(\eta,t)\,d\eta\\
&\approx i\widehat{V^{kg}}(\xi,t)\mathfrak{q}_+(\xi)u_{low}(t\xi/\Lambda_{kg}(\xi),t),
\end{split}
\end{equation*}
where $u_{low}$ is a suitable low-frequency component of $u$.

In view of this calculation we set $p_0:=0.68$ and define the phase correction
\begin{equation}\label{Nor4}
\Theta(\xi,t):=\mathfrak{q}_+(\xi)\int_0^tu_{low}(s\xi/\Lambda_{kg}(\xi),s)\,ds,\qquad \widehat{u_{low}}(\rho,s):=\varphi_{\leq 0}(\langle s\rangle^{p_0}\rho)\widehat{u}(\rho,s).
\end{equation}  
Then we define the modified Klein--Gordon profile $V_\ast^{kg}$ by
\begin{equation}\label{Nor5}
\widehat{V^{kg}_\ast}(\xi,t):=e^{-i\Theta(\xi,t)}\widehat{V^{kg}}(\xi,t).
\end{equation}
We notice that both the function $u_{low}$ and the multiplier $\mathfrak{q}_+$ are real-valued, thus $\Theta$ is real-valued. With $u_{high}=u-u_{low}$, the formula \eqref{Nor2} shows that
\begin{equation}\label{Nor7}
\begin{split}
\partial_t\widehat{V^{kg}_\ast}(\xi,t)&=e^{-i\Theta(\xi,t)}\{\partial_t\widehat{V^{kg}}(\xi,t)-i\widehat{V^{kg}}(\xi,t)\mathfrak{q}_+(\xi)u_{low}(t\xi/\Lambda_{kg}(\xi),t)\}\\
&=\mathcal{R}_1(\xi,t)+\mathcal{R}_2(\xi,t)+\mathcal{R}_3(\xi,t),
\end{split}
\end{equation}
where
\begin{equation}\label{Nor9}
\mathcal{R}_1(\xi,t):=\frac{e^{-i\Theta(\xi,t)}}{(2\pi)^3}\int_{\mathbb{R}^3}ie^{it\Lambda_{kg}(\xi)}\widehat{u_{low}}(\eta,t)e^{it\Lambda_{kg}(\xi-\eta)}\widehat{V^{kg,-}}(\xi-\eta,t)\mathfrak{q}_{-}(\xi-\eta)\,d\eta,
\end{equation}
\begin{equation}\label{Nor8}
\begin{split}
&\mathcal{R}_2(\xi,t):=\frac{e^{-i\Theta(\xi,t)}}{(2\pi)^3}\int_{\mathbb{R}^3}i\widehat{u_{low}}(\eta,t)\\
&\times\{e^{it(\Lambda_{kg}(\xi)-\Lambda_{kg}(\xi-\eta))}\widehat{V^{kg}}(\xi-\eta,t)\mathfrak{q}_+(\xi-\eta)-e^{it(\xi\cdot\eta)/\Lambda_{kg}(\xi)}\widehat{V^{kg}}(\xi,t)\mathfrak{q}_+(\xi)\}\,d\eta,
\end{split}
\end{equation}
and
\begin{equation}\label{Nor10}
\mathcal{R}_3(\xi,t):=\frac{e^{-i\Theta(\xi,t)}}{(2\pi)^3}\sum_{\iota_1\in\{+,-\}}\int_{\mathbb{R}^3}ie^{it\Lambda_{kg}(\xi)}\widehat{u_{high}}(\eta,t)e^{-it\Lambda_{kg,\iota_1}(\xi-\eta)}\widehat{V^{kg,\iota_1}}(\xi-\eta,t)\mathfrak{q}_{\iota_1}(\xi-\eta)\,d\eta.
\end{equation}

\subsection{Improved control} We prove now our main $Z$-norm estimate for the profile $V^{kg}$.

\begin{proposition}\label{BootstrapZ2}
For any $t\in[0,T]$ we have
\begin{equation*}
\|V^{kg}(t)\|_{Z_{kg}}\lesssim\varep_0.
\end{equation*}
\end{proposition}

The rest of this section is concerned with the proof of this proposition. Since $|\widehat{V^{kg}}(\xi,t)|=|\widehat{V^{kg}_\ast}(\xi,t)|$, in view of the definition \eqref{sec5.1} it suffices to prove that
\begin{equation}\label{Nor40}
\|\varphi_k(\xi)\{\widehat{V^{kg}_\ast}(\xi,t_2)-\widehat{V^{kg}_\ast}(\xi,t_1)\}\|_{L^\infty_\xi}\lesssim \varep_02^{-\delta m/2}2^{-k^-/2+\kappa k^-}2^{-N_0k^++d'k^+}
\end{equation}
and
\begin{equation}\label{Nor41}
\|\varphi_k(\xi)\{\widehat{V^{kg}_\ast}(\xi,t_2)-\widehat{V^{kg}_\ast}(\xi,t_1)\}\|_{L^2_\xi}\lesssim \varep_02^{-\delta m/2}2^{k^-+\kappa k^-}2^{-N_0k^+-(3d-2)k^+},
\end{equation}
for any $k\in\mathbb{Z}$, $m\geq 1$, and $t_1,t_2\in[2^m-2,2^{m+1}]\cap [0,T]$.

\begin{lemma}\label{transfer} 
The bounds \eqref{Nor40} and \eqref{Nor41} hold if $k\geq \kappa m/100-10$ or if $k\leq -\kappa m$.
\end{lemma}

\begin{proof}
Notice that the bound \eqref{Nor41} follows from Proposition \ref{BootstrapEE5} if $2^k\gtrsim 2^{2\delta m}$. It remains to show that if  $k\notin[-\kappa m,\kappa m/100-10]$ and $t\in[2^m-2,2^{m+1}]\cap [0,T]$ then
\begin{equation}\label{nar1}
\|\varphi_k(\xi)\widehat{V^{kg}}(\xi,t)\|_{L^\infty_\xi}\lesssim \varep_02^{-\delta m/2}2^{-k^-/2+\kappa k^-}2^{-N_0k^++d'k^+}.
\end{equation}

{\bf{Step 1.}} It follows from Proposition \ref{DiEs1} that
\begin{equation*}
2^{k^+}\|\varphi_{k}(\xi)(\partial_{\xi_l}\widehat{V^{kg}_{\mathcal{L}}})(\xi,t)\|_{L^2_\xi}\lesssim\varep_0\langle t\rangle^{H(n+1)\delta}2^{-N_0k^++(n+1)dk^+},
\end{equation*}
for any $t\in[0,T]$, $k\in\mathbb{Z}$, $l\in\{1,2,3\}$, and $\mathcal{L}\in\mathcal{V}_n$, $n\in [0,N_1-1]$. Using Lemma \ref{hyt1}, we have
\begin{equation}\label{Nor41.2}
\sup_{j\geq -k^-}2^j\|Q_{jk}V^{kg}_{\mathcal{L}}(t)\|_{L^2}\lesssim\varep_0\langle t\rangle^{H(n+1)\delta}2^{-N_0k^+-k^++(n+1)dk^+}.
\end{equation}
Using this and \eqref{consu2} it follows that
\begin{equation}\label{Nor41.1}
\|\widehat{P_kV^{kg}}(t)\|_{L^\infty}\lesssim \varep_0\langle t\rangle^{H(2)\delta}2^{-k^-/2}2^{-N_0k^+-k^+/2+(3d/2)k^+}.
\end{equation}
The bound \eqref{nar1} follows if $2^k\gtrsim 2^{4H(2)\delta m}$, since $d'=3d/2$. 

It remains to prove \eqref{nar1} when $k\leq -\kappa m$. We use again \eqref{Nor41.2} and \eqref{consu2} to estimate
\begin{equation*}
\|\widehat{P_kV^{kg}}(t)\|_{L^\infty}\lesssim 2^{-3k/2}\big\{\sup_{j\geq -k^-}\|Q_{jk}V^{kg}(t)\|_{H^{0,1}_\Omega}\big\}^{(1-\delta)/2}\big\{\varep_0\langle t\rangle^{H(2)\delta}2^k\big\}^{(1+\delta)/2}.
\end{equation*}
Therefore, recalling that $\kappa^2=400\delta$, for \eqref{nar1} it suffices to prove that
\begin{equation*}
\|P_kV^{kg}(t)\|_{H^{0,1}_\Omega}\lesssim\varep_0\langle t\rangle^{4H(2)\delta}2^{k+10\kappa k}
\end{equation*}
if $k\leq-\kappa m$ and  $t\in[2^m-2,2^{m+1}]\cap [0,T]$. In view of \eqref{hyt5}, for this is suffices to prove that
\begin{equation}\label{nar4}
\|P_lV^{kg}_{\Omega}(t)\|_{L^2}+\sum_{a=1}^3\|\varphi_l(\xi)(\partial_{\xi_a}\widehat{V^{kg}_{\Omega}})(\xi,t)\|_{L^2}\lesssim\varep_0\langle t\rangle^{4H(2)\delta}2^{10\kappa l},
\end{equation}
for any $l\in\mathbb{Z}$ and $t\in[0,T]$, where $\Omega\in\{Id,\Omega_{23},\Omega_{31},\Omega_{12}\}$.

The bound on the first term in the left-hand side of \eqref{nar4} follows from \eqref{Nor41.2}. To bound the remaining terms we use the identities \eqref{wer1}. For \eqref{nar4} it suffices to prove that
\begin{equation}\label{nar5}
\|P_l\Gamma_aU^{kg}_{\Omega}(t)\|_{L^2}+\|\varphi_l(\xi)(\partial_{\xi_a}\widehat{\mathcal{N}^{kg}_{\Omega}})(\xi,t)\|_{L^2}\lesssim\varep_0\langle t\rangle^{4H(2)\delta}2^{10\kappa l},
\end{equation}
for any $l\in\mathbb{Z}$, $t\in[0,T]$, and $a\in\{1,2,3\}$. The term $\|P_l\Gamma_aU^{kg}_{\Omega}(t)\|_{L^2}$ is bounded as claimed due to \eqref{Nor41.2} (with $n=2$). Therefore, it remains to prove the bilinear estimates
\begin{equation}\label{nar6}
\begin{split}
\|P_k\mathcal{N}^{kg}_{\Omega}(\xi,t)\|_{L^2}&\lesssim\varep_0\langle t\rangle^{4H(2)\delta}2^{10\kappa k}\min(\langle t\rangle^{-1},2^{k^-})2^{-2k^+}\\
\|\varphi_k(\xi)(\partial_{\xi_a}\widehat{\mathcal{N}^{kg}_{\Omega}})(\xi,t)\|_{L^2}&\lesssim\varep_0\langle t\rangle^{4H(2)\delta}2^{10\kappa k}
\end{split}
\end{equation}
for any $k\in\mathbb{Z}$, $t\in[0,T]$, and $a\in\{1,2,3\}$.

{\bf{Step 2.}} The bounds \eqref{nar6} are similar to the bounds in Lemma \ref{dtv8}. The only issue is to gain the factors $2^{10\kappa k}$ and we are allowed to lose small powers $\langle t\rangle^{C\delta}$. We may assume $k\leq 0$ and define $I$ as in \eqref{abc9}--\eqref{abc36}. For \eqref{nar6} it suffices to prove that
\begin{equation}\label{nar7}
\sum_{(k_1,k_2)\in\mathcal{X}_k}2^{k_1^+-k_2}\|P_kI[P_{k_1}U^{kg,\iota_1}_{\mathcal{L}_1},P_{k_2}U^{wa,\iota_2}_{\mathcal{L}_2}](t)\|_{L^2}\lesssim \varep_1^2\langle t\rangle^{4H(2)\delta}2^{10\kappa k}\min(\langle t\rangle^{-1},2^{k^-})
\end{equation}
and
\begin{equation}\label{nar8}
\sum_{(k_1,k_2)\in\mathcal{X}_k}2^{k_1^+-k_2}\|\varphi_k(\xi)(\partial_{\xi_a}\mathcal{F}\{I[P_{k_1}U^{kg,\iota_1}_{\mathcal{L}_1},P_{k_2}U^{wa,\iota_2}_{\mathcal{L}_2}]\})(\xi,t)\|_{L^2_\xi}\lesssim\varep_1^2\langle t\rangle^{4H(2)\delta}2^{10\kappa k},
\end{equation}
for any $\iota_1,\iota_2\in\{+,-\}$, $a\in\{1,2,3\}$, $\mathcal{L}_1\in\mathcal{V}_{n_1}$, $\mathcal{L}_2\in\mathcal{V}_{n_2}$, $n_1+n_2\leq 1$.

Using the $L^2$ estimates \eqref{abc22.1} and \eqref{abc21}, we bound
\begin{equation}\label{nar9}
\begin{split}
2^{k_1^+-k_2}\|P_kI[P_{k_1}U^{kg,\iota_1}_{\mathcal{L}_1},P_{k_2}&U^{wa,\iota_2}_{\mathcal{L}_2}](t)\|_{L^2}\lesssim2^{k_1^+-k_2}2^{3\min(k,k_1,k_2)/2}\|P_{k_1}U_{\mathcal{L}_1}^{kg}(t)\|_{L^2}\|P_{k_2}U_{\mathcal{L}_2}^{wa}(t)\|_{L^2}\\
&\lesssim\varep_1^2\langle t\rangle^{2H(2)\delta}2^{k_1-k_2/2}2^{3\min(k,k_1,k_2)/2}2^{-4(k_1^++k_2^+)}.
\end{split}
\end{equation}
This suffices to prove \eqref{nar7} when $2^{k^-}\lesssim\langle t\rangle^{-1}$. On the other hand, if $2^{k^-}\geq \langle t\rangle^{-1}$ then we estimate, using also \eqref{abc23.2}, 
\begin{equation*}
\begin{split}
2^{k_1^+-k_2}\|P_kI[P_{k_1}U^{kg,\iota_1}_{\mathcal{L}_1},P_{k_2}&U^{wa,\iota_2}_{\mathcal{L}_2}](t)\|_{L^2}\lesssim2^{k_1^+-k_2}\|P_{k_1}U_{\mathcal{L}_1}^{kg}(t)\|_{L^2}\|P_{k_2}U_{\mathcal{L}_2}^{wa}(t)\|_{L^\infty}\\
&\lesssim\varep_1^22^{k_1^-}\langle t\rangle^{-1+2H(2)\delta}\min(1,2^{k_2^-}\langle t\rangle)2^{-4(k_1^++k_2^+)}
\end{split}
\end{equation*}
This suffices to bound the contribution of the pairs $(k_1,k_2)\in\mathcal{X}_k$ in \eqref{nar7} for which $2^{k_1}\lesssim 2^{k/10}$. For the remaining pairs we have $\min(k_1,k_2)\geq k/10+10$, $|k_1-k_2|\leq 4$, and we use the decomposition \eqref{ku4} to estimate
\begin{equation*}
2^{k_1^+-k_2}\|P_kI[P_{k_1}U^{kg,\iota_1}_{\mathcal{L}_1},P_{k_2}U^{wa,\iota_2}_{\mathcal{L}_2}](t)\|_{L^2}\lesssim\varep_1^22^{-k_1/2}2^{-k_2/2}\langle t\rangle^{-3/2+\kappa}2^{-4(k_1^++k_2^+)}.
\end{equation*}
This suffices to complete the proof of \eqref{nar7}.

To prove \eqref{nar8} we write $U^{kg,\iota_1}_{\mathcal{L}_1}=e^{-it\Lambda_{kg,\iota_1}}V^{kg,\iota_1}_{\mathcal{L}_1}$ and notice that the $\partial_{\xi_a}$ derivative can hit either the phase $e^{-it\Lambda_{kg,\iota_1}(\xi-\eta)}$, or the multiplier $m_1(\xi-\eta)$, or the profile $\widehat{P_{k_1}V^{kg,\iota_1}_{\mathcal{L}_1}}(\xi-\eta)$. In the first case the derivative effectively corresponds to multiplying by factors $\lesssim \langle t\rangle$, and changing the multiplier $m_1$, in a way that still satisfies \eqref{abc9}. The desired estimates follow from \eqref{nar7}. In the case when the $\partial_{\xi_a}$ derivative hits the function $m_1(\xi-\eta)\widehat{P_{k_1}V^{kg,\iota_1}_{\mathcal{L}_1}}(\xi-\eta)$, it suffices to prove that 
\begin{equation*}
\sum_{(k_1,k_2)\in\mathcal{X}_k}2^{k_1^+-k_2}\|P_kI[U^{kg,\iota_1}_{\mathcal{L}_1,\ast a,k_1},P_{k_2}U^{wa,\iota_2}_{\mathcal{L}_2}](t)\|_{L^2}\lesssim\varep_1^2\langle t\rangle^{4H(2)\delta}2^{10\kappa k},
\end{equation*}
where, as in the proof of Lemma \ref{dtv8}, $\widehat{U^{kg,\iota_1}_{\mathcal{L}_1,\ast a,k_1}}(\xi,t)=e^{-it\Lambda_{kg,\iota_1}(\xi)}\partial_{\xi_a}\{\varphi_{k_1}\cdot m_1\cdot\widehat{V^{kg,\iota_1}_{\mathcal{L}_1}}\}(\xi,t)$. This follows from the $L^2$ bounds \eqref{abc48.3} and \eqref{abc21}.
\end{proof}

We return now to the proof of the main estimates \eqref{Nor40}--\eqref{Nor41}. In view of the identity \eqref{Nor7}, it suffices to prove that, for $a\in\{1,2,3\}$,
\begin{equation}\label{Nor50}
\Big\|\varphi_k(\xi)\int_{t_1}^{t_2}\mathcal{R}_a(\xi,s)\,ds\Big\|_{L^\infty_\xi}\lesssim \varep_02^{-\delta m/2}2^{-k^-/2+\kappa k^-}2^{-(N_0+4d)k^+},
\end{equation}
if $k\in[-\kappa m,\kappa m/100-10]$. These bounds are proved in Lemmas \ref{BootstrapZ3}--\ref{BootstrapZ5} below. 

In some estimates we need to use integration by parts in time (normal forms). For any $s\in[0,T]$ we define the bilinear operators $T^{kg}_{\mu\nu}$ by
\begin{equation}\label{Nor53.1}
T^{kg}_{\mu\nu}[f,g](\xi,s):=\int_{\mathbb{R}^3}\frac{e^{is\Phi_{(kg,+)\mu\nu}(\xi,\eta)}}{\Phi_{(kg,+)\mu\nu}(\xi,\eta)}m_1(\xi-\eta)\widehat{f}(\xi-\eta,s)\cdot m_2(\eta)\widehat{g}(\eta,s)\,d\eta,
\end{equation}
where $\Phi_{(kg,+)\mu\nu}(\xi,\eta)=\Lambda_{kg}(\xi)-\Lambda_{\mu}(\xi-\eta)-\Lambda_{\nu}(\eta)$ (see \eqref{on9.2}) and $m_1,m_2$ are as in \eqref{abc9}.

We will sometimes need $L^\infty$ bounds on the localized profiles $\widehat{V_{j,k}^{kg,\pm}}$, $\widehat{V_{j,k}^{wa,\pm}}$, and the time derivative $\widehat{P_{k}\partial_s V^{kg,\pm}}$. Since $2^{j}\|Q_{j,k}V^{kg}(s)\|_{H^{0,1}_\Omega}\lesssim \varep_12^{H(2)\delta m}2^{-N_0k^++2dk^+}$, see \eqref{abc22}, it follows from \eqref{Linfty3.3} that
\begin{equation}\label{Nor63}
\|\widehat{V_{j,k}^{kg,\pm}}(s)\|_{L^\infty}\lesssim \varep_12^{H(2)\delta m}2^{-j/2}2^{-k}2^{\delta(j+k)}2^{-N_0k^++2dk^+},
\end{equation}
for any $s\in[t_1,t_2]$. Similarly
\begin{equation}\label{Nor63.1}
\|\widehat{V_{j,k}^{wa,\pm}}(s)\|_{L^\infty}\lesssim \varep_12^{H(2)\delta m}2^{-j/2}2^{-3k/2}2^{\delta(j+k)}2^{-N_0k^++2dk^+}.
\end{equation}
Finally, to bound $\widehat{P_{k}\partial_s V^{kg,\pm}}$, recall that
\begin{equation}\label{Nor56.1}
\begin{split}
\sup_{j\geq -k^-}\,\|Q_{jk}(\partial_t V^{kg,\pm})(t)\|_{H^{0,1}_\Omega}\lesssim \varep_1^2\langle t\rangle^{-1+H''_{kg}(1)\delta}2^{-N(2)k^+-5k^+},\\
\sup_{j\geq -k^-}\,2^{j+k}\|Q_{jk}(\partial_t V^{kg,\pm})(t)\|_{H^{0,1}_\Omega}\lesssim \varep_1^22^{k}\langle t\rangle^{H''_{kg}(1)\delta}2^{-N(2)k^+-5k^+},
\end{split}
\end{equation}
if $2^k\gtrsim \langle t\rangle^{-1}$, see Lemma \ref{dtv8} and Lemma \ref{hyt1}. In particular, using \eqref{consu2},
\begin{equation}\label{Nor56.2}
\|\widehat{P_k(\partial_t V^{kg,\pm})}(t)\|_{L^\infty}\lesssim \varep_1^2\langle t\rangle^{-1/2+H(2)\delta}2^{-k^-}2^{-N(2)k^+}.
\end{equation}

\begin{lemma}\label{BootstrapZ3} The bounds \eqref{Nor50} hold if $a=1$, $m\geq 1/\kappa$, and  $k\in[-\kappa m,\kappa m/100]$.
\end{lemma}

\begin{proof} We examine the formula \eqref{Nor9}, substitute $u=i\Lambda_{wa}^{-1}(U^{wa,+}-U^{wa,-})/2$, and decompose the input functions dyadically in frequency. Let $\widehat{U_{low}^{wa,\iota_2}}(\xi,s):=\varphi_{\leq 0}(\langle s\rangle^{p_0}\xi)\widehat{U^{wa,\iota_2}}(\xi,s)$ and $\widehat{V_{low}^{wa,\iota_2}}(\xi,s):=\varphi_{\leq 0}(\langle s\rangle^{p_0}\xi)\widehat{V^{wa,\iota_2}}(\xi,s)$. With $I$ as in \eqref{abc36} and $\iota_2\in\{+,-\}$, it suffices to prove that
\begin{equation}\label{Nor52}
\begin{split}
\sum_{(k_1,k_2)\in\mathcal{X}_k}2^{k_1^+-k_2}\Big\|\varphi_k(\xi)\int_{t_1}^{t_2}e^{is\Lambda_{kg}(\xi)-i\Theta(\xi,s)}&\mathcal{F}\{I[P_{k_1}U^{kg,-},P_{k_2}U^{wa,\iota_2}_{low}]\}(\xi,s)\,ds\Big\|_{L^\infty_\xi}\\
&\lesssim \varep_1^22^{-\delta m/2}2^{-k^-/2+\kappa k^-}2^{-(N_0+4d)k^+}.
\end{split}
\end{equation}

We estimate first, using \eqref{vcx1.1},
\begin{equation}\label{Nor53.5}
\begin{split}
2^{k_1^+-k_2}\big\|\mathcal{F}\{I[P_{k_1}U^{kg,-},&P_{k_2}U^{wa,\iota_2}_{low}]\}(\xi,s)\big\|_{L^\infty_\xi}\lesssim 2^{k_1^+-k_2}\|\widehat{P_{k_1}U^{kg,-}}\|_{L^\infty}\|\widehat{P_{k_2}U^{wa,\iota_2}_{low}}\|_{L^1}\\
&\lesssim \varep_1^22^{k_2}2^{-k_1^-/2}2^{-N_0k_1^++(d'+1)k_1^+}.
\end{split}
\end{equation} 
This suffices to control the contribution of the pairs $(k_1,k_2)$ for which $k_2\leq -1.01m$. Thus it remains to prove that
\begin{equation}\label{Nor54}
\begin{split}
\Big|\int_{t_1}^{t_2}&\int_{\mathbb{R}^3}e^{is\Lambda_{kg}(\xi)-i\Theta(\xi,s)}m_1(\xi-\eta)e^{is\Lambda_{kg}(\xi-\eta)}\widehat{P_{k_1}V^{kg,-}}(\xi-\eta,s)\\
&\times m_2(\eta)e^{-is\Lambda_{wa,\iota_2}(\eta)}\widehat{P_{k_2}V^{wa,\iota_2}_{low}}(\eta,s)\, d\eta ds\Big|\lesssim \varep_1^22^{-\kappa m}2^{k_2},
\end{split}
\end{equation}
for any $\xi$ with $|\xi|\in[2^{k_1-4},2^{k_1+4}]$, provided that
\begin{equation}\label{Nor55}
k_2\in [-1.01m,-p_0m+10],\qquad k_1\in [\kappa m-10,\kappa m/100+10]. 
\end{equation}

To prove \eqref{Nor54} we integrate by parts in time. Notice that $\Phi_{\sigma\mu\nu}(\xi,\eta)\gtrsim 1$ in the support of the integral if $\sigma=(kg,+)$, $\mu=(kg,-),\,\nu=(wa,\iota_2)$.{\footnote{Here it is important that $\mu\neq (kg,+)$, so the phase is nonresonant. The nonlinear correction \eqref{Nor5} was done precisely to weaken the corresponding resonant contribution of the profile $V^{kg,+}$.}} The left-hand side of \eqref{Nor54} is dominated by $C(J_1+J_2+J_3)$, where, with $\mu=(kg,-)$ and $\nu=(wa,\iota_2)$ and $T^{kg}_{\mu\nu}$ defined as in \eqref{Nor53.1},
\begin{equation}\label{Nor56}
\begin{split}
&J_1:=\sum_{s\in\{t_1,t_2\}}|T_{\mu\nu}^{kg}[P_{k_1}V^\mu,P_{k_2}V^\nu_{low}](\xi,s)|+\int_{t_1}^{t_2}|\dot{\Theta}(\xi,s)|\cdot|T_{\mu\nu}^{kg}[P_{k_1}V^\mu,P_{k_2}V^\nu_{low}](\xi,s)|\,ds,\\
&J_2:=\int_{t_1}^{t_2}|T_{\mu\nu}^{kg}[\partial_s(P_{k_1}V^\mu),P_{k_2}V^\nu_{low}](\xi,s)|\,ds,\\
&J_3:=\int_{t_1}^{t_2}|T_{\mu\nu}^{kg}[P_{k_1}V^\mu,\partial_s(P_{k_2}V^\nu_{low})](\xi,s)|\,ds.
\end{split}
\end{equation}

Assuming $k_1,k_2$ as in \eqref{Nor55}, we estimate, as in \eqref{Nor53.5},
\begin{equation*}
\begin{split}
|T_{\mu\nu}^{kg}[P_{k_1}V^\mu,P_{k_2}V^\nu_{low}](\xi,s)|&\lesssim\varep_1^22^{2k_2}2^{\kappa m},\\
2^m|T_{\mu\nu}^{kg}[P_{k_1}V^\mu,\partial_s(P_{k_2}V^\nu_{low})](\xi,s)|&\lesssim\varep_1^22^{2k_2}2^{\kappa m}.
\end{split}
\end{equation*}
Here we used the bounds  $\|\widehat{P_{k_2}V^\nu_{low}}\|_{L^1}\lesssim 2^{2k_2}2^{\delta m}$ and $\|\partial_s\widehat{P_{k_2}V^\nu_{low}}\|_{L^1}\lesssim 2^{2k_2}2^{-m+10\delta m}$, see \eqref{abc21} and \eqref{abc30}. Since $\dot{\Theta}(\xi,s)=\mathfrak{q}_+(\xi)u_{low}(s\xi/\Lambda_{kg}(\xi),s)$, see \eqref{Nor4}, it follows from \eqref{abc23.2} that $|\dot{\Theta}(\xi,s)|\lesssim 2^{k_1^+}2^{-m+12\delta m}$. The desired bounds for the terms $J_1$ and $J_3$ follow.

To estimate $J_2$ we use \eqref{Nor56.2}. Therefore
\begin{equation*}
|T_{\mu\nu}^{kg}[\partial_s(P_{k_1}V^\mu),P_{k_2}V^\nu_{low}](\xi,s)|\lesssim\varep_1^22^{-m/2+\kappa m}2^{2k_2},
\end{equation*} 
and the desired estimates follow since $2^{k_2}\lesssim 2^{-p_0m}$. This completes the proof of the lemma.
\end{proof}

\begin{lemma}\label{BootstrapZ4}  The bounds \eqref{Nor50} hold if $a=2$, $m\geq 1/\kappa$, and  $k\in[-\kappa m,\kappa m/100]$.
\end{lemma}

\begin{proof} We decompose $V^{kg}=\sum_{(k_1,j_1)\in\mathcal{J}}V_{j_1,k_1}^{kg,+}$ as in \eqref{abc20.5}. For \eqref{Nor50} it suffices to prove that
\begin{equation}\label{Nor61}
\sum_{(k_1,j_1)\in\mathcal{J}}|A_{k;j_1,k_1}(\xi,s)|\lesssim \varep_1^22^{-1.005m}
\end{equation}
for any $s\in [2^{m-1},2^{m+1}]\cap [0,T]$ and $k\in[-\kappa m,\kappa m/100]$, where
\begin{equation}\label{Nor61.5}
\begin{split}
A_{k;j_1,k_1}(\xi,s):=\varphi_k(\xi)\int_{\mathbb{R}^3}\widehat{u_{low}}(\eta,t)\{&e^{is(\Lambda_{kg}(\xi)-\Lambda_{kg}(\xi-\eta))}\widehat{V_{j_1,k_1}^{kg,+}}(\xi-\eta,s)\mathfrak{q}_+(\xi-\eta)\\
&-e^{is(\xi\cdot\eta)/\Lambda_{kg}(\xi)}\widehat{V_{j_1,k_1}^{kg,+}}(\xi,s)\mathfrak{q}_+(\xi)\}\,d\eta.
\end{split}
\end{equation}

As a consequence of \eqref{Nor63}, without using the cancellation of the two terms in the integral,
\begin{equation*}
\begin{split}
|A_{k;j_1,k_1}(\xi,s)|&\lesssim \varep_12^{-j_1/2+\delta j_1}2^{H(2)\delta m}2^{-k_1}\|\widehat{u_{low}}(s)\|_{L^1}\lesssim \varep_12^{-j_1/2+\delta j_1}2^{H(2)\delta m}2^{-k_1}2^{-p_0m+\delta m}.
\end{split}
\end{equation*}
Since $2^{-k_1}\lesssim 2^{\kappa m/4}$, this suffices to control the contribution of the terms in \eqref{Nor61} corresponding to large values of $j_1$, i.e. $2^{j_1/2}\gtrsim 2^{(1.01-p_0)m}$. 

On the other hand, if $j_1/2\leq (1.01-p_0)m=0.33 m$ then we estimate
\begin{equation}\label{Nor65}
\begin{split}
\big|e^{is(\Lambda_{kg}(\xi)-\Lambda_{kg}(\xi-\eta))}-e^{is(\xi\cdot\eta)/\Lambda_{kg}(\xi)}\big|&\lesssim 2^{-2p_0m+m},\\
\big|\widehat{V_{j_1,k_1}^{kg,+}}(\xi,s)\mathfrak{q}_+(\xi)-\widehat{V_{j_1,k_1}^{kg,+}}(\xi-\eta,s)\mathfrak{q}_+(\xi-\eta)\big|&\lesssim \varep_12^{j_1/2}2^{\kappa m}2^{-p_0m},
\end{split}
\end{equation}
provided that $|\xi|\approx 2^k$ and $|\eta|\lesssim 2^{-p_0m}$. Indeed, the first bound follows from the observation that $\nabla\Lambda_{kg}(\xi)=\xi/\Lambda_{kg}(\xi)$. The second bound follows from \eqref{Nor63}, once we notice that differentiation in $\xi$ of the localized profile $\widehat{V_{j_1,k_1}^{kg,+}}(\xi,s)$ corresponds essentially to multiplication by a factor of $2^{j_1}$. Therefore, since $\|\widehat{u_{low}}(s)\|_{L^1}\lesssim 2^{-p_0m+\delta m}$ and $2^{j_1/2}\lesssim 2^{0.33m}$,
\begin{equation*}
\begin{split}
|A_{k;j_1,k_1}(\xi,s)|&\lesssim \|\widehat{u_{low}}(s)\|_{L^1}\varep_12^{-p_0m+\kappa m}(2^{j_1/2}+2^{-p_0m+m}2^{-j_1/2})\lesssim\varep_12^{-2p_0m+2\kappa m}2^{0.33m}.
\end{split}
\end{equation*}
The contribution of the pairs $(k_1,j_1)$ for which $2^{j_1/2}\leq 2^{0.33m}$ is therefore bounded as claimed in \eqref{Nor61}. This completes the proof of the lemma.
\end{proof}

\begin{lemma}\label{BootstrapZ5}
The bounds \eqref{Nor50} hold if $a=3$, $m\geq 1/\kappa$, and  $k\in[-\kappa m,\kappa m/100]$.
\end{lemma}

\begin{proof} We examine the formula \eqref{Nor10}, write $u=i\Lambda_{wa}^{-1}(U^{wa,+}-U^{wa,-})/2$, and decompose the input functions dyadically in frequency. Let $U^{wa,\iota_2}_{high}:=U^{wa,\iota_2}-U^{wa,\iota_2}_{low}$ and $V^{wa,\iota_2}_{high}:=V^{wa,\iota_2}-V^{wa,\iota_2}_{low}$. As in the proof of Lemma \ref{BootstrapZ3}, after simple reductions it suffices to prove that
\begin{equation}\label{Nor72}
2^{k_1^+-k_2}\Big\|\varphi_k(\xi)\int_{t_1}^{t_2}e^{is\Lambda_{kg}(\xi)-i\Theta(\xi,s)}\mathcal{F}\{I[P_{k_1}U^{kg,\iota_1},P_{k_2}U^{wa,\iota_2}_{high}]\}(\xi,s)\,ds\Big\|_{L^\infty_\xi}\lesssim \varep_1^22^{-\kappa m}
\end{equation}
for any $\iota_1,\iota_2\in\{+,-\}$ and $k_1,k_2\in[-p_0m-10,m/10]$.

We integrate by parts in time to estimate
\begin{equation*}
\Big|\int_{t_1}^{t_2}e^{is\Lambda_{kg}(\xi)-i\Theta(\xi,s)}\mathcal{F}\{I[P_{k_1}U^{kg,\iota_1},P_{k_2}U^{wa,\iota_2}_{high}]\}(\xi,s)\,ds\Big|\lesssim J'_1(\xi)+J'_2(\xi)+J'_3(\xi),
\end{equation*}
where, with $\mu=(kg,\iota_1)$ and $\nu=(wa,\iota_2)$ and $T^{kg}_{\mu\nu}$ defined as in \eqref{Nor53.1},
\begin{equation*}
\begin{split}
&J'_1(\xi):=\sum_{s\in\{t_1,t_2\}}|T_{\mu\nu}^{kg}[P_{k_1}V^\mu,P_{k_2}V^\nu_{high}](\xi,s)|+\int_{t_1}^{t_2}|\dot{\Theta}(\xi,s)|\cdot|T_{\mu\nu}^{kg}[P_{k_1}V^\mu,P_{k_2}V^\nu_{high}](\xi,s)|\,ds,\\
&J'_2(\xi):=\int_{t_1}^{t_2}|T_{\mu\nu}^{kg}[\partial_s(P_{k_1}V^\mu),P_{k_2}V^\nu_{high}](\xi,s)|\,ds,\\
&J'_3(\xi):=\int_{t_1}^{t_2}|T_{\mu\nu}^{kg}[P_{k_1}V^\mu,\partial_s(P_{k_2}V^\nu_{high})](\xi,s)|\,ds.
\end{split}
\end{equation*}
Since $|\dot{\Theta}(\xi,s)|\lesssim 2^{-m+4\delta m}$ for \eqref{Nor72} it suffices to prove that for any $s\in[2^{m-1},2^{m+1}]$,
\begin{equation}\label{Nor75}
|\varphi_k(\xi)T_{\mu\nu}^{kg}[P_{k_1}V^\mu,P_{k_2}V^\nu_{high}](\xi,s)|\lesssim \varep_1^22^{-2\kappa m}2^{k_2^-},
\end{equation}
\begin{equation}\label{Nor75.2}
2^m|\varphi_k(\xi)T_{\mu\nu}^{kg}[\partial_s(P_{k_1}V^\mu),P_{k_2}V^\nu_{high}](\xi,s)|\lesssim \varep_1^22^{-2\kappa m}2^{k_2^-},
\end{equation}
\begin{equation}\label{Nor75.3}
2^m|\varphi_k(\xi)T_{\mu\nu}^{kg}[P_{k_1}V^\mu,\partial_s(P_{k_2}V^\nu_{high})](\xi,s)|\lesssim \varep_1^22^{-2\kappa m}2^{k_2^-},
\end{equation}
provided that $k\in [-\kappa m,\kappa m/100]$, $k_1,k_2\in[-pm-10,m/10]$, $\mu=(kg,\iota_1)$, $\nu=(wa,\iota_2)$, $\iota_1,\iota_2\in\{+,-\}$.

{\bf{Step 1: proof of \eqref{Nor75}.}} If $k_1\leq -4\kappa m$ (so $k_2\geq -\kappa m-20$) then we can just use $L^2$ bounds \eqref{abc21}--\eqref{abc22.1} on both inputs and Lemma \ref{pha2} (i) to prove \eqref{Nor75}. On the other hand, if  $k_1\geq -4\kappa m$ then we decompose $P_{k_1}V^{\mu}=\sum_{j_1}V_{j_1,k_1}^{\mu}$ and $P_{k_2}V^{\nu}=\sum_{j_2}V_{j_2,k_2}^{\nu}$ as in \eqref{abc20.5}. Let $\overline{k}:=\max(k,k_1,k_2)$ and recall that $|\Phi_{(kg,+)\mu\nu}(\xi,\eta)|\gtrsim 2^{k_2}2^{-2\overline{k}^+}$ in the support of the integrals defining the operators $T_{\mu\nu}^{kg}$ (see \eqref{pha3}). 

The contribution of the pairs $(V_{j_1,k_1}^{\mu},V_{j_2,k_2}^{\nu})$ for which $2^{\max(j_1,j_2)}\leq 2^{0.99m}2^{-6\overline{k}^+}$ is negligible,
\begin{equation}\label{Nor78}
|T^{kg}_{\mu\nu}[V_{j_1,k_1}^{\mu},V_{j_2,k_2}^{\nu}](\xi,s)|\lesssim \varep_1^22^{-2m}\,\,\,\,\text{ if }\,\,\,\,2^{\max(j_1,j_2)}\leq 2^{0.99m}2^{-6\overline{k}^+}.
\end{equation}
Indeed, this follows by integration by parts in $\eta$ (using Lemma \ref{tech5}), the bounds \eqref{Linfty3.4}, and the observation that the gradient of the phase admits a suitable lower bound $|\nabla_{\eta}\{s\Lambda_{kg,\iota_1}(\xi-\eta)+s\Lambda_{wa,\iota_2}(\eta)\}|\gtrsim \langle s\rangle 2^{-2k_1^+}$ in the support of the integral. On the other hand, we estimate
\begin{equation*}
\begin{split}
|T^{kg}_{\mu\nu}&[V_{j_1,k_1}^{\mu},V_{j_2,k_2}^{\nu}](\xi,s)|\lesssim 2^{-k_2}2^{2\overline{k}^+}2^{3k_2/2}\|\widehat{V_{j_1,k_1}^{\mu}}(s)\|_{L^\infty}\|\widehat{V_{j_2,k_2}^{\nu}}(s)\|_{L^2}\\
&\lesssim \varep_1^22^{\kappa m}2^{k_2/2}2^{2k_1^++2k_2^+}\cdot 2^{-k_1}2^{-j_1/2}2^{\delta(j_1+k_1)}2^{-10k_1^+}\cdot 2^{-j_2}2^{-k_2^-/2}2^{-10k_2^+}\\
&\lesssim \varep_1^22^{2\kappa m}2^{-k_1}2^{-j_1/2+\delta j_1}2^{-j_2}2^{-6\overline{k}^+},
\end{split}
\end{equation*}
using \eqref{vcx1.1}, \eqref{Nor63}, and Lemma \ref{pha2} (i). Recalling that $k_1\geq -4\kappa m$, this suffices to estimate the contribution of the pairs $(V_{j_1,k_1}^{\mu},V_{j_2,k_2}^{\nu})$ for which $2^{\max(j_1,j_2)}\geq 2^{0.99m}2^{-6\overline{k}^+}$. The bound \eqref{Nor75} follows.

{\bf{Step 2: proof of \eqref{Nor75.2}.}} Recall that $e^{-it\Lambda_{kg,\iota_1}}\partial_tV^{kg,\iota_1}(t)=\mathcal{N}^{kg}(t)$ for any $t\in[0,T]$. Notice that $\mathcal{F}\{P_{k_1}\mathcal{N}^{kg}\}(s)$ can be written as a sum of terms of the form
\begin{equation*}
\varphi_{k_1}(\gamma)\int_{\mathbb{R}^3}|\rho|^{-1}\langle\gamma-\rho\rangle m_3(\gamma-\rho)\widehat{U^{kg,\iota_3}}(\gamma-\rho)\widehat{U^{wa,\iota_4}}(\rho)\,d\rho,
\end{equation*}
where $\iota_3,\iota_4\in\{+,-\}$ and $m_3$ is a symbol as in \eqref{abc9}. We combine this with the formula \eqref{Nor53.1}. For \eqref{Nor75.2} it suffices to prove that, for any $\xi\in\mathbb{R}^3$,
\begin{equation}\label{Nor81}
\begin{split}
\Big|&\varphi_k(\xi)\int_{\mathbb{R}^3\times\mathbb{R}^3}\frac{\varphi_{k_1}(\xi-\eta)m_1(\xi-\eta)}{\Lambda_{kg}(\xi)-\Lambda_\mu(\xi-\eta)-\Lambda_\nu(\eta)}m_2(\eta)e^{-is\Lambda_\nu(\eta)}\widehat{P_{k_2}V^{wa,\iota_2}}(\eta,s)\\
&\times m_3(\xi-\eta-\rho)\langle\xi-\eta-\rho\rangle|\rho|^{-1}\widehat{U^{kg,\iota_3}}(\xi-\eta-\rho,s)\widehat{U^{wa,\iota_4}}(\rho,s)\,d\eta d\rho\Big|\lesssim \varep_1^22^{k_2^-}2^{-1.005m},
\end{split}
\end{equation} 
provided that $\mu=(kg,\iota_1)$, $\nu=(wa,\iota_2)$, $\iota_1,\iota_2,\iota_3,\iota_4\in\{+,-\}$, and $k_1,k_2\in[-p_0m-10,m/10]$.

We decompose the solutions $U^{kg,\iota_3}$, $U^{wa,\iota_4}$, and $P_{k_2}V^{wa,\iota_2}$ dyadically in frequency and space as in \eqref{abc20.5}. Then we notice that the contribution when one of the parameters $j_3,k_3,j_4,k_4,j_2$ is large can be bounded using just $L^2$ estimates. For \eqref{Nor81} it suffices to prove that
\begin{equation}\label{Nor82}
2^{-k_2^-}2^{k_3^+-k_4}|\mathcal{C}_{kg}[e^{-is\Lambda_\theta}V_{j_3,k_3}^{kg,\iota_3}(s),e^{-is\Lambda_\nu}V_{j_2,k_2}^{wa,\iota_2}(s),e^{-is\Lambda_\vartheta}V^{wa,\iota_4}_{j_4,k_4}(s)](\xi)|\lesssim \varep_1^32^{-1.01m}
\end{equation}
for any $k_2\in[-p_0m-10,m/10]$, $k_3,k_4\leq m/10$, and $j_2,j_3,j_4\leq 2m$, where $\theta=(kg,\iota_3)$, $\vartheta=(wa,\iota_4)$, and, with $m_1,m_2,m_3,m_4$ as in \eqref{abc9},
\begin{equation}\label{Nor83}
\begin{split}
\mathcal{C}_{kg}[f,g,h](\xi):=&\int_{\mathbb{R}^3\times\mathbb{R}^3}\frac{\varphi_k(\xi)\varphi_{k_1}(\xi-\eta)m_1(\xi-\eta)m_2(\eta)}{\Lambda_{kg}(\xi)-\Lambda_\mu(\xi-\eta)-\Lambda_\nu(\eta)}\\
&\times m_3(\xi-\eta-\rho)m_4(\rho)\cdot\widehat{f}(\xi-\eta-\rho)\widehat{g}(\eta)\widehat{h}(\rho)\,d\eta d\rho.
\end{split}
\end{equation}

{\bf{Substep 2.1.}} Assume first that
\begin{equation}\label{pol1}
j_3\geq 0.99 m-3k_3^+.
\end{equation}
Let $Y$ denote the left-hand side of \eqref{Nor82}. Using Lemma \ref{pha2} and \ref{L1easy} (i) we estimate
\begin{equation}\label{pol2}
\begin{split}
Y&\lesssim 2^{k^+_3-k_4}2^{-2k_2+5\max(k^+,k_2^+)}\|V_{j_3,k_3}^{kg,\iota_3}(s)\|_{L^2}\|e^{-is\Lambda_\nu}V_{j_2,k_2}^{wa,\iota_2}(s)\|_{L^\infty}\|V^{wa,\iota_4}_{j_4,k_4}(s)\|_{L^2}\\
&\lesssim \varep_1^32^{\kappa m}2^{-j_3-dk_3^+}2^{-m}2^{-k_2}2^{-k_4/2},
\end{split}
\end{equation}
where in the last line we used bounds from Lemma \ref{dtv6}. Since $2^{-k_2}\lesssim 2^{0.68m}$ and $j_3+3k_3^+\geq 0.99 m$, this suffices to prove \eqref{Nor82} when $k_4\geq -0.56m$. On the other hand, if $k_4\leq -0.56m$, then we estimate in the Fourier space. Using \eqref{pha3}, \eqref{Nor63}, and \eqref{abc21}
\begin{equation}\label{pol3}
\begin{split}
Y&\lesssim 2^{k_3^+-k_4}2^{-2k_2+3\max(k^+,k_2^+)}\|\widehat{V_{j_3,k_3}^{kg,\iota_3}}(s)\|_{L^\infty}2^{3k_2/2}\|\widehat{V_{j_2,k_2}^{wa,\iota_2}}(s)\|_{L^2}2^{3k_4/2}\|\widehat{V^{wa,\iota_4}_{j_4,k_4}}(s)\|_{L^2}\\
&\lesssim \varep_1^32^{\kappa m}2^{-j_3/2+\delta j_3}2^{k_4}2^{-k_3}.
\end{split}
\end{equation}
Since $k_4\leq -0.56m$, this suffices to prove \eqref{Nor82} when $k_3\geq -0.01m$. Finally, if $k_4\leq -0.56m$ and $k_3\leq -0.01m$ then $k_2\geq -\kappa m-10$ (due to the assumption $k\geq -\kappa m$) and a similar estimate gives
\begin{equation}\label{pol4}
\begin{split}
Y&\lesssim 2^{k_3^+-k_4}2^{-2k_2+3\max(k^+,k_2^+)}2^{3k_3/2}\|\widehat{V_{j_3,k_3}^{kg,\iota_3}}(s)\|_{L^2}\|\widehat{V_{j_2,k_2}^{wa,\iota_2}}(s)\|_{L^\infty}2^{3k_4/2}\|\widehat{V^{wa,\iota_4}_{j_4,k_4}}(s)\|_{L^2}\\
&\lesssim \varep_1^32^{0.01 m}2^{-j_3}2^{k_4}.
\end{split}
\end{equation} 
This completes the proof of \eqref{Nor82} when $j_3\geq 0.99 m-3k_3^+$.

{\bf{Substep 2.2.}} Assume now that $j_3\leq 0.99 m-3k_3^+$. We notice that the $\eta$ gradient of the phase $-s\Lambda_\theta(\xi-\eta-\rho)-s\Lambda_\nu(\eta)$ is $\gtrsim 2^m2^{-2k_3^+}$ in the support of the integral in \eqref{Nor83}. Similarly, the $\rho$ gradient of the phase $-s\Lambda_\theta(\xi-\eta-\rho)-s\Lambda_\vartheta(\rho)$ is $\gtrsim 2^m2^{-2k_3^+}$ in the support of the integral. Using Lemma \ref{tech5} (integration by parts in $\eta$ or $\rho$), the contribution is negligible unless
\begin{equation}\label{pol7}
j_2\geq 0.99 m-3k_3^+\quad\text{ and }\quad j_4\geq 0.99 m-3k_3^+.
\end{equation}

Given \eqref{pol7}, we estimate first, as in \eqref{pol3},
\begin{equation*}
\begin{split}
Y&\lesssim 2^{k_3^+-k_4}2^{-2k_2+3\max(k^+,k_2^+)}\|\widehat{V_{j_3,k_3}^{kg,\iota_3}}(s)\|_{L^\infty}2^{3k_2/2}\|\widehat{V_{j_2,k_2}^{wa,\iota_2}}(s)\|_{L^2}2^{3k_4/2}\|\widehat{V^{wa,\iota_4}_{j_4,k_4}}(s)\|_{L^2}\\
&\lesssim \varep_1^32^{\kappa m}2^{-j_2-j_4}2^{-k_2}2^{-k_3}2^{-10k_3^+}.
\end{split}
\end{equation*}
This suffices if $k_3\geq -0.2m$. On the other hand, if $k_3\leq -0.2m$ then we may assume that $\max(k_2,k_4)\geq -\kappa m-10$ (due to the assumption $k\geq -\kappa m$) and estimate as in \eqref{pol2},
\begin{equation*}
\begin{split}
Y&\lesssim 2^{k^+_3-k_4}2^{-2k_2+5\max(k^+,k_2^+)}\|e^{-is\Lambda_\theta}V_{j_3,k_3}^{kg,\iota_3}(s)\|_{L^\infty}\|V_{j_2,k_2}^{wa,\iota_2}(s)\|_{L^2}\|V^{wa,\iota_4}_{j_4,k_4}(s)\|_{L^2}\\
&\lesssim \varep_1^32^{\kappa m}2^{-m}2^{-j_2-j_4}2^{-5k_2/2}2^{-3k_4/2}2^{-10k_3^+}.
\end{split}
\end{equation*}
This suffices to prove \eqref{Nor82} when $k_4\geq -0.1 m$. Finally, if $k_3,k_4\leq -0.1 m$ and $k_2\geq -\kappa m-10$ then we estimate, as in \eqref{pol4} and using also \eqref{Nor63.1},
\begin{equation*}
\begin{split}
Y&\lesssim 2^{k_3^+-k_4}2^{-2k_2+3\max(k^+,k_2^+)}2^{3k_3/2}\|\widehat{V_{j_3,k_3}^{kg,\iota_3}}(s)\|_{L^2}\|\widehat{V_{j_2,k_2}^{wa,\iota_2}}(s)\|_{L^\infty}2^{3k_4/2}\|\widehat{V^{wa,\iota_4}_{j_4,k_4}}(s)\|_{L^2}\\
&\lesssim \varep_1^32^{0.01 m}2^{-j_4}2^{-j_2/2+\kappa j_2}2^{-10k_3^+},
\end{split}
\end{equation*} 
which suffices. This completes the proof of the the bounds \eqref{Nor75.2}.

{\bf{Step 3: proof of \eqref{Nor75.3}.}} Recall that $e^{-it\Lambda_{wa,\iota_2}}\partial_tV^{wa,\iota_2}(t)=\mathcal{N}^{wa}(t)$ for any $t\in[0,T]$. If $k_1\leq -0.01m$ then we may assume that $k_2\geq -\kappa m-10$ (due to the assumption $k\geq -\kappa m$). Using \eqref{abc30} we estimate the left-hand side of \eqref{Nor75.3} by
\begin{equation*}
C2^m2^{-k_2}2^{2\max(k_1^+,k_2^+)}\|\widehat{P_{k_1}V^\mu}\|_{L^2}\|\widehat{P_{k_2}\mathcal{N}^{wa}}\|_{L^2}\lesssim \varep_1^22^{k_1^-}2^{2\kappa m},
\end{equation*}
which suffices. We can also decompose $P_{k_1}V^{kg,\iota_1}=\sum_{j_1\geq -k_1^-}V^{kg,\iota_1}_{j_1,k_1}$, and notice that the contribution of the localized profiles for which $j_1\geq 0.1 m$ can be bounded in a similar way, using \eqref{Nor63}.  After these reductions it remains to prove that
\begin{equation}\label{Nor91}
|\varphi_k(\xi)T_{\mu\nu}^{kg}[V^{kg,\iota_1}_{j_1,k_1}(s),e^{is\Lambda_{wa,\iota_2}}P_{k_2}\mathcal{N}^{wa}(s)](\xi)|\lesssim \varep_1^22^{-1.005m}2^{k_2^-},
\end{equation}
for any $s\in[2^{m-1},2^{m+1}]$, $k_2\in[-p_0m-10,m/10]$, $k_1\geq -0.01m$, and $j_1\leq 0.1 m$.

We examine now the quadratic nonlinearities $\mathcal{N}^{wa}$ in \eqref{on6.1L}. We define the trilinear operators
\begin{equation}\label{Nor92}
\begin{split}
\mathcal{C}'_{kg}[f,g,h](\xi):=&\int_{\mathbb{R}^3\times\mathbb{R}^3}\frac{\varphi_k(\xi)\varphi_{k_2}(\eta)m_1(\xi-\eta)m_2(\eta)}{\Lambda_{kg}(\xi)-\Lambda_\mu(\xi-\eta)-\Lambda_\nu(\eta)}\\
&\times m_3(\eta-\rho)m_4(\rho)\cdot\widehat{f}(\xi-\eta)\widehat{g}(\eta-\rho)\widehat{h}(\rho)\,d\eta d\rho,
\end{split}
\end{equation} 
where $m_1,m_2,m_3,m_4$ are as in \eqref{abc9}. For \eqref{Nor91} it suffices to prove that
\begin{equation}\label{Nor95}
2^{-k_2^-}|\mathcal{C}'_{kg}[e^{-is\Lambda_{\mu}}V^{kg,\iota_1}_{j_1,k_1}(s),e^{-is\Lambda_{kg,\iota_3}}V^{kg,\iota_3}_{j_3,k_3}(s),e^{-is\Lambda_{kg,\iota_4}}V^{kg,\iota_4}_{j_4,k_4}(s)](\xi)|\lesssim \varep_1^22^{-1.01m},
\end{equation}
where $s,k_1,k_2,j_1$ are as in \eqref{Nor91} and $(k_3,j_3),(k_4,j_4)\in\mathcal{J}$.

Using Lemma \ref{pha2} (ii) and \eqref{vcx5}, we estimate the left-hand side of \eqref{Nor95} by
\begin{equation*}
\begin{split}
C2^{-2k_2+5\max(k^+,k_1^+)}&\|e^{-is\Lambda_{\mu}}V^{kg,\iota_1}_{j_1,k_1}(s)\|_{L^\infty}\|V^{kg,\iota_3}_{j_3,k_3}(s)\|_{L^2}\|V^{kg,\iota_4}_{j_4,k_4}(s)\|_{L^2}\lesssim \varep_1^32^{-1.49m}2^{-j_3}2^{-j_4}2^{-2k_2}.
\end{split}
\end{equation*} 
This suffices if $2^{-0.48m}2^{-j_3}2^{-j_4}2^{-2k_2}\lesssim 1$. Otherwise, if $j_3+j_4+0.48 m\leq -2k_2-120$ then we may assume that $j_3\leq j_4$ (so $j_3\leq 0.45 m-50$ since $k_2\geq -0.68m-10$) and use \eqref{vcx5} again to estimate the left-hand side of \eqref{Nor95} by
\begin{equation*}
\begin{split}
C2^{-2k_2+3\max(k^+,k_1^+)}&\|\widehat{V^{kg,\iota_1}_{j_1,k_1}}(s)\|_{L^\infty}2^{3k_2/2}\|e^{-is\Lambda_{kg,\iota_3}}V^{kg,\iota_3}_{j_3,k_3}(s)\|_{L^\infty}\|V^{kg,\iota_4}_{j_4,k_4}(s)\|_{L^2}\\
&\lesssim \varep_1^32^{-1.49m}2^{-k_2^-/2}2^{-k_3^-/2}2^{-j_4}.
\end{split}
\end{equation*}
Since $2^{-k_3^-/2}2^{-j_4}\lesssim 2^{j_3/3-j_4}\lesssim 1$, the bounds \eqref{Nor95} follow. This completes the proof.
\end{proof}

\section{Bounds on the profiles, III: the wave $Z$ norm}\label{DiEs3}

We prove now our main $Z$-norm estimate for the profile $V^{wa}$.

\begin{proposition}\label{BootstrapZ1}
For any $t\in[0,T]$ we have
\begin{equation*}
\|V^{wa}(t)\|_{Z_{wa}}\lesssim\varep_0.
\end{equation*}
\end{proposition}

The rest of the section is concerned with the proof of this proposition. In view of the definition \eqref{sec5} it suffices to prove that for any $m\geq 1$ and $k\in\mathbb{Z}$ we have
\begin{equation}\label{par20}
\sum_{j\geq\max(-k,0)}2^j\|Q_{jk}V^{wa}(t_2)-Q_{jk}V^{wa}(t_1)\|_{L^2}\lesssim \varep_02^{-\delta m}2^{-k^-(1/2+\kappa)}2^{-(N_0-d')k^+}.
\end{equation} 
for any $t_1,t_2\in[2^m-2,2^{m+1}]\cap [0,T]$. This follows from Lemmas \ref{BootstrapZ11}--\ref{BootstrapZ13} below.

\begin{lemma}\label{BootstrapZ11}
For any $m\geq 1$ and $k\in\mathbb{Z}$ let $J_0:=m(1+\kappa)+2|k|+10$. Then, for any $t_1,t_2\in[2^m-2,2^{m+1}]\cap [0,T]$
\begin{equation}\label{par20.5}
\sum_{j\geq J_0}2^j\|Q_{jk}V^{wa}(t_2)-Q_{jk}V^{wa}(t_1)\|_{L^2}\lesssim \varep_02^{-\delta m}2^{-k^-(1/2+\kappa)}2^{-(N_0-d')k^+}.
\end{equation} 
\end{lemma}

\begin{proof}
This is a bound on the contribution of large $j$ in the sum in \eqref{par20}. To prove it we use an approximate finite speed of propagation argument. Since $\partial_tV^{wa}=e^{it\Lambda_{wa}}\mathcal{N}^{wa}$, for \eqref{par20.5} it suffices to prove that
\begin{equation*}
2^{j(1+\delta)+m}\|\varphi_j\cdot e^{it\Lambda_{wa}}P_k\mathcal{N}^{wa}(t)\|_{L^2}\lesssim \varep_02^{-\delta m}2^{-k^-(1/2+\kappa)}2^{-N_0k^++d'k^+},
\end{equation*}
 for any $t\in[2^{m}-2,2^{m+1}]\cap[0,T]$ and $j\geq J_0$. With $I$ as in \eqref{abc36}, it suffices to show that
\begin{equation}\label{par26}
\sum_{(k_1,k_2)\in\mathcal{X}_k}\|\varphi_j\cdot e^{it\Lambda_{wa}}P_kI[P_{k_1}U^\mu,P_{k_2}U^\nu](t)\|_{L^2}\lesssim \varep_1^22^{-j(1+\delta)}2^{-m(1+\delta)}2^{-N_0k^++d'k^+},
\end{equation}
for any $m\geq 1$, $k\in\mathbb{Z}$, $j\geq J_0$, $t\in[2^{m}-2,2^{m+1}]\cap[0,T]$, and $\mu,\nu\in\{(kg,+),(kg,-)\}$.

Notice first that the contribution of the pairs $(k_1,k_2)$ with $\max(k_1,k_2)\geq j$ or $\min(k_1,k_2)\leq -j$ can be controlled easily using just the $L^2$ bounds \eqref{abc21} and \eqref{abc22.1}. On the other hand, if $k_1,k_2\in[-j,j]$ then we decompose 
\begin{equation*}
P_{k_1}U^\mu=\sum_{j_1\geq -k_1^-}e^{-it\Lambda_\mu}V^{\mu}_{j_1,k_1},\qquad P_{k_2}U^\nu=\sum_{j_2\geq -k_2^-}e^{-it\Lambda_\nu}V^{\nu}_{j_2,k_2}
\end{equation*}
as in \eqref{abc20.5}. For \eqref{par26} it suffices to prove that 
\begin{equation}\label{par25}
\begin{split}
\|\varphi_j\cdot e^{it\Lambda_{wa}}P_k&I[e^{-it\Lambda_\mu}V^{\mu}_{j_1,k_1}(t),e^{-it\Lambda_\nu}V^{\nu}_{j_2,k_2}(t)]\|_{L^2}\\
&\lesssim \varep_1^22^{-j(1+\delta)}2^{-m(1+\delta)}2^{-\delta j_1}2^{-\delta j_2}2^{-N(1)k^++2k^+},
\end{split}
\end{equation}
for any $m,k,j,t,\mu,\nu$ as before, and any $(k_1,j_1),(k_2,j_2)\in\mathcal{J}$ with $k_1,k_2\in[-j,j]$.

If $\min(j_1,j_2)\geq j(1-\delta)$ then we estimate the left-hand side of \eqref{par25} by
\begin{equation*}
C2^{3k/2}\|V^{\mu}_{j_1,k_1}(t)\|_{L^2}\|V^{\nu}_{j_2,k_2}(t)\|_{L^2}\lesssim \varep_1^22^{3k/2}2^{-j_1-j_2}2^{-N(1)k_1^+-N(1)k_2^+}2^{20\delta m},
\end{equation*}
using \eqref{abc22}. This gives the desired bound \eqref{par25}, since $j\geq J_0\geq m(1+\kappa)$. 

On the other hand, if $\min(j_1,j_2)\leq j(1-\delta)$ then the left-hand side of \eqref{par25} is negligible. Indeed, we may assume that $j_1\leq j(1-\delta)$ and write
\begin{equation*}
\begin{split}
\varphi_j(x)\cdot e^{it\Lambda_{wa}}&P_kI[e^{-it\Lambda_\mu}V^{\mu}_{j_1,k_1}(t),e^{-it\Lambda_\nu}V^{\nu}_{j_2,k_2}(t)](x)=C\varphi_j(x)\int_{\mathbb{R}^3\times\mathbb{R}^3}\varphi_k(\xi)e^{ix\cdot\xi}\\
&\times e^{it(\Lambda_{wa}(\xi)-\Lambda_\mu(\xi-\eta)-\Lambda_\nu(\eta))}m_1(\xi-\eta)\widehat{V^{\mu}_{j_1,k_1}}(\xi-\eta,t)m_2(\eta)\widehat{V^{\nu}_{j_2,k_2}}(\eta,t)\,d\xi d\eta.
\end{split}
\end{equation*}
We integrate by parts in $\xi$ many times, using Lemma \ref{tech5}; at each integration by parts we gain a factor of $|x|\approx 2^j$ and lose a factor $\lesssim 2^m+2^{|k|}+2^{j_1}$. It follows that the left-hand side of \eqref{par25} is bounded by $C2^{-100 j}2^{-j_2}$ if $j_1\leq j(1-\delta)$. The desired bounds \eqref{par25} follow, which completes the proof of the lemma.
\end{proof}

We prove now the bounds \eqref{par20} when $k$ is not close to $0$, using Proposition  \ref{DiEs1}.

\begin{lemma}\label{BootstrapZ12}
The bounds \eqref{par20} hold if $k\geq \kappa m/100-10$ or $k\leq -\kappa m$.
\end{lemma}

\begin{proof} It follows from Proposition \ref{DiEs1} that
\begin{equation*}
2^{k/2}\|\varphi_{k}(\xi)(\partial_{\xi_l}\widehat{V^{wa}})(\xi,t)\|_{L^2_\xi}\lesssim\varep_0\langle t\rangle^{H(1)\delta}2^{-N_0k^++dk^+},
\end{equation*}
for any $t\in[0,T]$, $k\in\mathbb{Z}$, $l\in\{1,2,3\}$. Using Lemma \ref{hyt1}, we have
\begin{equation}\label{par21}
\sup_{j\geq -k^-}2^j\|Q_{jk}V^{wa}(t)\|_{L^2}\lesssim\varep_0\langle t\rangle^{H(1)\delta}2^{-N_0k^++dk^+}2^{-k/2}.
\end{equation}
Given the assumption on $k$, this suffices to control the contribution of the sum over $j\leq J_0$ in \eqref{par20}. The remaining part of the sum is suitably bounded because of Lemma \ref{BootstrapZ11}.
\end{proof}

Finally, we also prove the bounds \eqref{par20} when $k$ is close to $0$..

\begin{lemma}\label{BootstrapZ13}
The bounds \eqref{par20} hold if $m\geq 1/\kappa$ and $k\in[-\kappa m,\kappa m/100]$.
\end{lemma}

\begin{proof} In view of Lemma \ref{BootstrapZ11}, it suffices to show that
\begin{equation*}
2^{J_0}\|P_kV^{wa}(t_2)-P_kV^{wa}(t_1)\|_{L^2}\lesssim \varep_02^{-\delta m}2^{-k^-(1/2+\kappa)}2^{-(N_0-d')k^+}.
\end{equation*}
Recall that $\partial_tV^{wa}=e^{it\Lambda_{wa}}\mathcal{N}^{wa}$ and $k\in[-\kappa m,\kappa m/100]$. It remains to show that
\begin{equation}\label{par28}
\sum_{(k_1,k_2)\in\mathcal{X}_k}\Big\|\varphi_k(\xi)\int_{t_1}^{t_2}e^{is\Lambda_{wa}(\xi)}\mathcal{F}\{I[P_{k_1}U^{kg,\iota_1},P_{k_2}U^{kg,\iota_2}]\}(\xi,s)\,ds\Big\|_{L^2_\xi}\lesssim  \varep_1^22^{-m-3\kappa m}
\end{equation}
for any $t_1,t_2\in[2^{m-1},2^{m+1}]\cap [0,T]$ and $\iota_1,\iota_2\in\{+,-\}$, where $I$ is as in \eqref{abc36}.

It is easy to bound the contribution of the pairs $(k_1,k_2)$ in \eqref{par28} for which either $\min(k_1,k_2)\leq -0.9m$ or $\max(k_1,k_2)\geq m/10$, using just the $L^2$ estimates in \eqref{vcx1.1}. For the remaining pairs we would like to integrate by parts in time. We define the bilinear operators $T^{wa}_{\mu\nu}$ by
\begin{equation}\label{par29}
T^{wa}_{\mu\nu}[f,g](\xi,s):=\int_{\mathbb{R}^3}\frac{e^{is\Phi_{(wa,+)\mu\nu}(\xi,\eta)}}{\Phi_{(wa,+)\mu\nu}(\xi,\eta)}m_1(\xi-\eta)\widehat{f}(\xi-\eta,s)\cdot m_2(\eta)\widehat{g}(\eta,s)\,d\eta,
\end{equation}
where $s\in[0,T]$, $\mu,\nu\in\{(kg,+),(kg,-)\}$, $\Phi_{(wa,+)\mu\nu}(\xi,\eta)=\Lambda_{wa}(\xi)-\Lambda_{\mu}(\xi-\eta)-\Lambda_{\nu}(\eta)$ (see \eqref{on9.2}), and $m_1,m_2$ are as in \eqref{abc9}. Compare with the definition \eqref{Nor53.1}. We integrate by parts in time to estimate
\begin{equation*}
\Big|\int_{t_1}^{t_2}e^{is\Lambda_{wa}(\xi)}\mathcal{F}\{I[P_{k_1}U^{\mu},P_{k_2}U^{\nu}]\}(\xi,s)\,ds\Big|\lesssim J''_1(\xi)+J''_2(\xi)
\end{equation*}
where, with $\mu=(kg,\iota_1)$, $\nu=(nu,\iota_2)$, and $T^{kg}_{\mu\nu}$ defined as in \eqref{par29},
\begin{equation*}
\begin{split}
&J''_1(\xi):=\sum_{s\in\{t_1,t_2\}}|T_{\mu\nu}^{wa}[P_{k_1}V^\mu,P_{k_2}V^\nu](\xi,s)|\\
&J''_2(\xi):=\int_{t_1}^{t_2}|T_{\mu\nu}^{wa}[\partial_s(P_{k_1}V^\mu),P_{k_2}V^\nu](\xi,s)|\,ds+\int_{t_1}^{t_2}|T_{\mu\nu}^{wa}[P_{k_1}V^\mu,\partial_s(P_{k_2}V^\nu)](\xi,s)|\,ds.
\end{split}
\end{equation*}
Using also the symmetry in $k_1,k_2$, for \eqref{par28} it suffices to prove that for 
\begin{equation}\label{par30}
\|\varphi_k(\xi)T_{\mu\nu}^{wa}[P_{k_1}V^\mu,P_{k_2}V^\nu](\xi,s)\|_{L^2_\xi}\lesssim \varep_1^22^{-m-4\kappa m},
\end{equation}
\begin{equation}\label{par31}
2^m\|\varphi_k(\xi)T_{\mu\nu}^{wa}[P_{k_1}V^\mu,(\partial_sP_{k_2}V^\nu)](\xi,s)\|_{L^2_\xi}\lesssim \varep_1^22^{-m-4\kappa m},
\end{equation}
provided that $s\in[2^{m-1},2^{m+1}]$, $k\in [-\kappa m,\kappa m/100]$, $k_1,k_2\in[-0.9m,m/10]$, $\mu=(kg,\iota_1)$, $\nu=(kg,\iota_2)$, and $\iota_1,\iota_2\in\{+,-\}$.

To prove prove \eqref{par30}--\eqref{par31} we notice that, as a consequence of \eqref{pha4} and Lemma \ref{L1easy} (ii),
\begin{equation}\label{par35}
\|\varphi_l(\xi)T_{\mu\nu}^{wa}[P_{l_1}f,P_{l_2}g](\xi,s)\|_{L^2_\xi}\lesssim 2^{-l}2^{4\max(l_1^+,l_2^+)}\|e^{-is\Lambda_\mu}P_{l_1}f\|_{L^{p_1}}\|e^{-is\Lambda_\nu}P_{l_2}g\|_{L^{p_2}}
\end{equation}
for any $f(s),g(s)\in L^2(\mathbb{R}^3)$, $l,l_1,l_2\in\mathbb{Z}$, and $(p_1,p_2)\in\{(2,\infty),(\infty,2)\}$. Therefore
\begin{equation*}
\begin{split}
\|\varphi_k(\xi)T_{\mu\nu}^{wa}[P_{k_1}V^\mu,P_{k_2}V^\nu](\xi,s)\|_{L^2_\xi}+2^m\|\varphi_k(\xi)T_{\mu\nu}^{wa}[P_{k_1}V^\mu,(\partial_sP_{k_2}V^\nu)](\xi,s)\|_{L^2_\xi}\\
\lesssim \varep_1^22^{-k}2^{k_1^-/2}2^{-m+\kappa m}
\end{split}
\end{equation*}
using also \eqref{abc23}, \eqref{abc21}, and \eqref{abc31}. This suffices to prove \eqref{par30}--\eqref{par31} if $k_1\leq -0.1m$.

On the other hand, if $k_1\in[-0.1m,0.1m]$ then we set $2^J=2^{k_1^--30}2^m$ and decompose $P_{k_1}V^{kg,\iota_1}=V_{\leq J,k_1}^{kg,\iota_1}+V_{> J,k_1}^{kg,\iota_1}$, as in \eqref{abc20.5}.
We have, see \eqref{ku10},
\begin{equation}\label{par40}
\begin{split}
\|e^{-is\Lambda_{kg,\iota_1}}V^{kg,\iota_1}_{\leq J,k_1}(s)\|_{L^\infty}&\lesssim \varep_12^{-k_1^-/2}2^{-3m/2}2^{-10k_1^+},\\
\|V^{kg,\iota_1}_{>J,k_1}(s)\|_{L^2}&\lesssim \varep_12^{-k_1^-}2^{-m+\kappa m}2^{-10k_1^+}.
\end{split}
\end{equation}
Therefore, using \eqref{par35} with $(p_1,p_2)=(\infty,2)$, and the bounds \eqref{abc21} and \eqref{abc31},
\begin{equation}\label{par48}
\begin{split}
\|\varphi_k(\xi)T_{\mu\nu}^{wa}[V_{\leq J,k_1}^{\mu},P_{k_2}V^\nu](\xi,s)\|_{L^2_\xi}+2^m\|\varphi_k(\xi)T_{\mu\nu}^{wa}[V_{\leq J,k_1}^{\mu},(\partial_sP_{k_2}V^\nu)](\xi,s)\|_{L^2_\xi}\\
\lesssim \varep_1^22^{-k}2^{-k_1^-/2}2^{-3m/2+\kappa m}.
\end{split}
\end{equation}
In addition, using \eqref{par35} with $(p_1,p_2)=(2,\infty)$, and the $L^\infty$ bounds \eqref{abc23} and \eqref{mnb51.4},
\begin{equation}\label{par49}
\begin{split}
\|\varphi_k(\xi)T_{\mu\nu}^{wa}[V_{>J,k_1}^{\mu},P_{k_2}V^\nu](\xi,s)\|_{L^2_\xi}+2^m\|\varphi_k(\xi)T_{\mu\nu}^{wa}[V_{>J,k_1}^{\mu},(\partial_sP_{k_2}V^\nu)](\xi,s)\|_{L^2_\xi}\\
\lesssim \varep_1^22^{-k}2^{-k_1^-}2^{-2m+3\kappa m}.
\end{split}
\end{equation}
The desired bounds \eqref{par30}--\eqref{par31} follow if $|k_1|\leq 0.1m$ from \eqref{par48}--\eqref{par49}. This completes the proof of the lemma.
\end{proof}

\end{document}